\documentclass[english,11pt]{article}

\usepackage[english]{babel}
\usepackage{amsmath,amssymb,enumerate,amsthm}
\usepackage[latin1]{inputenc}
\usepackage{mathrsfs}
\usepackage{psfig,labelfig}
\usepackage{graphicx}
\usepackage[dvips]{epsfig}
\usepackage{lineno} 

\newcommand{\beqref}[1]{\eqref{#1}}
\newcommand{\bbref}[2]{#1\;\ref{#2}}

\newcommand{\itemup}[1]{\item[\textup{#1}]}
\newcommand{\widebar}[2]{\mkern #1mu\overline{\mkern-#1mu#2\mkern-#1mu}\mkern #1mu}

\numberwithin{equation}{section}

\newtheorem{thm}{Theorem}[section]
\newtheorem{cor}[thm]{Corollary}
\newtheorem{lem}[thm]{Lemma}

\theoremstyle{definition}
\newtheorem{rem}[thm]{Remark}
\newtheorem{defi}[thm]{Definition}

\newcommand{\R}{{\mathbb R}}
\newcommand{\Sp}{{\mathbb S}}
\newcommand{\Z}{{\mathbb Z}}

\DeclareMathOperator{\arcsinh}{arcsinh}

\DeclareMathOperator{\dist}{dist}
\DeclareMathOperator{\capa}{cap}
\DeclareMathOperator{\area}{area}
\DeclareMathOperator{\sys}{sys}
\newcommand{\vol}{\textup{vol}}
\DeclareMathOperator{\cl}{cl}
\newcommand{\scp}[2]{\big\langle #1,#2 \big\rangle}
\newcommand{\smscp}[2]{\langle #1,#2 \rangle}
\newcommand{\pt}{\rule{1pt}{0pt}}
\newcommand{\lin}{{\scriptscriptstyle \mathsf{lin}}}
\newcommand{\nlin}{{\scriptscriptstyle \mathsf{nl}}}

\newcommand{\sF}{{\scriptscriptstyle\mathsf F}}
\newcommand{\sG}{{\scriptscriptstyle\mathsf G}}
\newcommand{\ssR}{{\scriptscriptstyle\mathscr{R}}}

\newcommand{\FG}{\text{\tiny\raisebox{0.3ex}{$\Lambda$}}}
\newcommand{\lef}{\mathsf{L}}
\newcommand{\rig}{\mathsf{R}}
\newcommand{\Ess}{\mathscr{E}}
\newcommand{\Esss}[1]{\mathscr{E}_{#1}}

\newcommand{\mT}{\tau}
\newcommand{\sR}{\mathscr{R}}
\newcommand{\Hv}[1]{H^{1,v}(#1,\R)}
\newcommand{\Hstar}{\hspace{0.1ex}{\star}}
\newcommand{\defeq}{\overset{\text{\tiny def}}{=}}
\newcommand{\ssqrt}[1]{\sqrt{\text{\raisebox{-0.4ex}{\rule{0ex}{2.3ex}}} \smash{#1}}}

\title{Energy distribution of harmonic 1-forms and Jacobians of Riemann surfaces with a short closed geodesic}
\author{Peter Buser, Eran Makover, Bjoern Muetzel and Robert Silhol}

\begin{document}
\maketitle

\begin{abstract}
We study the energy distribution of harmonic 1-forms on a compact hyperbolic Riemann surface $S$ where a short closed geodesic is pinched. If the geodesic separates the surface into two parts, then the Jacobian variety of $S$ develops into a variety that splits. If the geodesic is nonseparating then the Jacobian degenerates. The aim of this work is to get insight into this process and give estimates in terms of geometric data of both the initial surface $S$ and the final surface, such as its injectivity radius and the lengths of geodesics that form a homology basis. The Jacobians in this paper are represented by Gram period matrices.

As an invariant we introduce new families of symplectic matrices that compensate for the lack of full dimensional Gram-period matrices in the noncompact case.
\end{abstract}

\tableofcontents

\section{Introduction}\label{sec:Introd}

This paper grew out of the following question: given a compact hyperbolic Riemann surface $S$ of genus $g \geq 2$ and a simple closed geodesic $\gamma$ separating it into two parts $S_1$, $S_2$; if $\gamma$ is fairly short, e.g.\ if it has length $\ell(\gamma) \leq \frac{1}{2}$, does that imply that the Jacobian of $S$ comes close to a direct product? 

We shall answer this and related questions by analysing the energy distribution of real harmonic 1-forms on thin handles. 

\medskip

{\bf Energy distribution.}\; In what follows a differential form or 1-form on a Riemann surface $S$ (compact or not) shall always be understood to be a \emph{real} differential 1-form. If it is harmonic we shall call it a harmonic form. Given two differential forms $\omega$, $\eta$ we denote by $\omega \wedge \eta$ the wedge product and by $\star \eta$ the Hodge star operator applied to $\eta$. The \emph{pointwise} norm squared of $\omega$ is denoted by $\Vert \omega \Vert^2 = \omega \wedge \star \omega$. When necessary we write $\Vert \omega(p) \Vert^2$ to indicate that it is evaluated at point $p \in S$. The integral of an integrable function $f$ over a domain $D \subseteq S$ shall be denoted by $\int_D f$. Provided the integral exists we define the \emph{energy} of a differential form $\omega$ over  $D$ as
 \begin{equation}\label{eq:energy}%
E_D(\omega) = \int_D \Vert \omega \Vert^2.
 \end{equation}%
When $D = S$ we write, more simply,  $E_S(\omega) = E(\omega)$ and call it \emph{the} energy of $\omega$. We also work with the scalar product for 1-forms $\omega$, $\eta$ of class $L^2$ on $S$,
 \begin{equation}\label{eq:scalpr}%
\scp{\omega}{\eta} = \int_S \omega \wedge \star \eta.
 \end{equation}%

In this paper we make extensively use of the fact that when $S$ is compact, then in each cohomology class of closed 1-forms on $S$ there is a unique harmonic form and that the latter is an energy minimiser, i.e.\  the harmonic form is the unique element in its cohomology class that has minimal energy.

Our main technical tool will be the analysis of the energy distribution of a harmonic form in a cylinder (Section \ref{sec:HarmCy}). By the Collar Lemma for hyperbolic surfaces, discussed in \bbref{Section}{sec:ColCus}, any simple closed geodesic $\gamma$ on $S$ is embedded in a cylindrical neighbourhood $C(\gamma)$ whose width is defined by $\gamma$, the so called standard collar. The gray shaded area in \bbref{Fig.}{fig:gamma_small} is an illustration in the case where $\gamma$ is separating. A first application of this analysis towards the  question asked above is the following, where the constant $\mu(\gamma)$ is defined as
 \begin{equation}\label{eq:mugamma_intro}%
 \mu (\gamma) = \exp\left(-2\pi^2\left(\frac{1}{\ell(\gamma)}-\frac{1}{2}\right)\right).
 \end{equation}%
 \begin{thm}[Vanishing theorem]\label{thm:energy_distribution}%
Let $S$ be a compact  hyperbolic Riemann surface and let $\gamma$ be a simple closed geodesic of length $\ell(\gamma) \leq \frac{1}{2}$ separating $S$ into two parts $S_1$, $S_2$. If $\sigma$ is a non-trivial harmonic 1-form on $S$ all of whose periods over cycles in $S_2$ vanish, then
\begin{enumerate}
\itemup{a)}
$E_{S_2}(\sigma) \leq \mu(\gamma) E(\sigma)$
\itemup{b)} $E_{S_2 \smallsetminus C(\gamma)}(\sigma) \leq \mu^2(\gamma) E_{C(\gamma)}(\sigma)$.
\end{enumerate}
 \end{thm}%

In other words, if $\sigma$ ``lives on $S_1$'' then its energy decays rapidly along $C(\gamma)$ and stays small in the rest of $S_2$. The proof is given in \bbref{Section}{sec:SepCas}. For technical reasons it is split: statement a) appears as the first part of \bbref{Theorem}{thm:small_blocks} and statement b) appears as \bbref{Theorem}{thm:EngInC}.

\medskip

{\bf Jacobian and Gram period matrix.}\; We denote by $H_1(S,\Z)$ the first homology group and by $H^1(S,\R)$ the vector space of all real harmonic 1-forms on $S$. Together with the scalar product \beqref{eq:scalpr} $H^1(S,\R)$ is a Euclidean vector space of dimension $2g$, and $H_1(S,\Z)$ is an Abelian group of rank $2g$.

For any homology class $[\alpha] \in H_1(S,\Z)$ and any closed 1-form $\omega$ the \emph{period} of $\omega$ over $[\alpha]$ is the path integral $\int_{\alpha}\omega$, defined independently of the choice of the cycle $\alpha$ in the homology class and we write it also in the form $\int_{\alpha}\omega = \int_{[\alpha]}\omega$. When $[\alpha]$ is the class of a simple closed curve we usually take $\alpha$ to be a geodesic (with respect to the hyperbolic metric on $S$). We call a set of $2g$ oriented simple closed geodesics

 \begin{equation*}%
 {\rm A} = (\alpha_1, \alpha_{2},\ldots,\alpha_{2g-1},\alpha_{2g})
 \end{equation*}%
a \emph{canonical homology basis} if for $k = 1, \dots, g$,
$\alpha_{2k-1}$ intersects $\alpha_{2k}$ in exactly one point, with algebraic  intersection number +1, and if for all other couples $i < j$ one has $\alpha_i \cap \alpha_j = \emptyset$. For given $\rm A$ we let $\left( {\sigma_k } \right)_{k = 1,\ldots,2g} \subset H^1(S,\R)$ be the \emph{dual basis of harmonic forms} 

 \begin{equation*}%
\int_{[\alpha_i]} {\sigma_k } = \delta_{ik} , \quad i,k = 1,\dots, 2g.
 \end{equation*}%
It is at the same time a basis of the lattice $\Lambda = \{n_1\sigma_1 + \dots + n_{2g}\sigma_{2g} \mid n_1,\dots,n_{2g} \in \Z \}$, where $\Lambda$ is also defined intrinsically as the set of all $\lambda \in H^1(S,\R)$ for which all periods are integer numbers. The quotient $J(S) = H^1(S,\R)/\Lambda$ of the Euclidean space $H^1(S,\R)$ by the lattice $\Lambda$ is a real $2g$-dimensional flat torus, the (real) \emph{Jacobian variety} of the Riemann surface $S$ (for the relation with the complex form see, e.g., \cite[Chapter III]{FK92}). Up to isometry, the lattice $\Lambda$ is determined by the Gram matrix

 \begin{equation*}%
P_S=\left(\left\langle {\sigma_i,\sigma_j} \right\rangle  \right)_{i,j= 1,\ldots,2g}=\left( \int_S {\sigma_i  \wedge \star \sigma_j } \right)_{i,j= 1,\ldots,2g}
 \end{equation*}%
(suppressing the mentioning of ${\rm A}$). As a Riemannian manifold $J(S)$ is isometric to the standard torus $\R^{2g}/\Z^{2g}$ endowed with the flat Riemannian metric induced by the matrix $P_S$. Hence, knowing $J(S)$ is the same as knowing $P_S$. We call $P_S$ the \emph{Gram period matrix} of $S$ with respect to ${\rm A}$.
\medskip

{\bf Going to the boundary of $\mathcal{M}_g$.}\; Let $\mathcal{M}_g$ be the moduli space of all isometry classes of compact Riemann surfaces of genus $g$, $g\geq2$ and let $\partial \mathcal{M}_g$ be its boundary in the sense of Deligne-Mumford. The members of $\partial \mathcal{M}_g$ are \emph{Riemann surfaces with nodes} where some collars are degenerated into cusps (e.g. Bers \cite{Be74}).

Several authors have studied the Jacobian under the aspect of \emph{opening up nodes} i.e., by giving variational formulas in terms of $t$ for the period matrix along paths $F_t$ in $\mathcal{M}_g \cup \partial \mathcal{M}_g$ that start on the boundary, for $t=0$, and lead inside $\mathcal{M}_g$ for $t \neq 0$; see e.g. \cite{Fay73}, \cite{Ya80}, \cite[Appendix]{IT92}, \cite{Fa07} and, more recently, the accounts \cite{GKN17}, \cite{HN18}, where one may also find additional literature. 

In contrast to this, our starting point is a surface $S$ inside $\mathcal{M}_g$, a certain amount away from the boundary. Therefore, in order to formulate an answer to the question asked above we first have to  determine some element on the boundary $\partial\mathcal{M}_g$ whose (limit) Jacobian we may use for comparison.  We shall apply two constructions that yield paths from $S$ to such an element on the boundary, one uses Fenchel-Nielsen coordinates (e.g.\ \cite[Chapter 1.7]{Bu92}),  the other construction is grafting.

\medskip

\textit{Fenchel-Nielsen construction}\;  

This way of going to the boundary seems to have been applied for the first time in \cite{CC89} in connection with the small eigenvalues of the Laplace-Beltrami operator. 

Fix a set of Fenchel-Nielsen coordinates corresponding to a partition of $S$  into hyperbolic pairs of pants, where the partition is chosen such that $\gamma$ is among the partitioning geodesics and  $\ell(\gamma)$ is one of the coordinates. For any $t > 0$, in particular for $t \in (0, \ell(\gamma)]$, there is a Riemann surface $S_t$ that has the same Fenchel-Nielsen coordinates as $S$ for the same partitioning pattern, with the sole exception that the coordinate $\ell(\gamma)$ is changed to $\ell(\gamma_t) = t$. Furthermore, there is a natural marking homeomorphism $\phi_t : S \to S_t$ that sends the partitioning geodesics of $S$ to those of $S_t$ and preserves the twist parameters. Hence, a path $\big(S_t\big)_{t \leq \ell(\gamma)}$ in $\mathcal{M}_g$, where ``$\ell(\gamma)$ is shrinking to zero.'' 

If $\gamma$ is a nonseparating geodesic then the limit for $t \to 0$ is a noncompact hyperbolic surface $S^{\sF}$ of genus $g-1$ with two cusps, and the marking homeomorphisms $\phi_t$ degenerate into a marking homeomorphism $\phi^{\sF} : S \smallsetminus \gamma \to S^{\sF}$. 

If $\gamma$ separates $S$ into two bordered surfaces $S_1$, $S_2$, say of genera $g_1$, $g_2$, then the limit is a pair of noncompact hyperbolic surfaces $S_1^{\sF}$, $S_2^{\sF}$ of genera $g_1$, $g_2$ each with one cusp, and the marking homeomorphisms $\phi_t$ degenerate into a couple of marking homeomorphisms $\phi_i^{\sF}: S_i \smallsetminus \gamma \to S_i^{\sF}$, $i=1,2$.

\begin{figure}[h!]
\SetLabels
\L(.50*.93) $\gamma$\\
\L(.40*.76) $S_1$\\
\L(.61*.76) $S_2$\\
\L(.38*.58) $S_t$\\
\L(.61*.58) $S(l)$\\
\L(.33*.20) $S^{\sF}$\\
\L(.67*.20) $S^{\sG}$\\
\L(.23*.0) $S_1^{\sF}$\\
\L(.44*.0) $S_2^{\sF}$\\
\L(.54*.0) $S^{\sG}_1$\\
\L(.825*.0) $S_2^{\sG}$\\
\L(.25*.70) $\mathcal{M}_g$\\
\L(.90*.37) $\partial \mathcal{M}_g$\\
\endSetLabels
\AffixLabels{%
\centerline{%
\includegraphics[height=7.2cm,width=12.6cm]{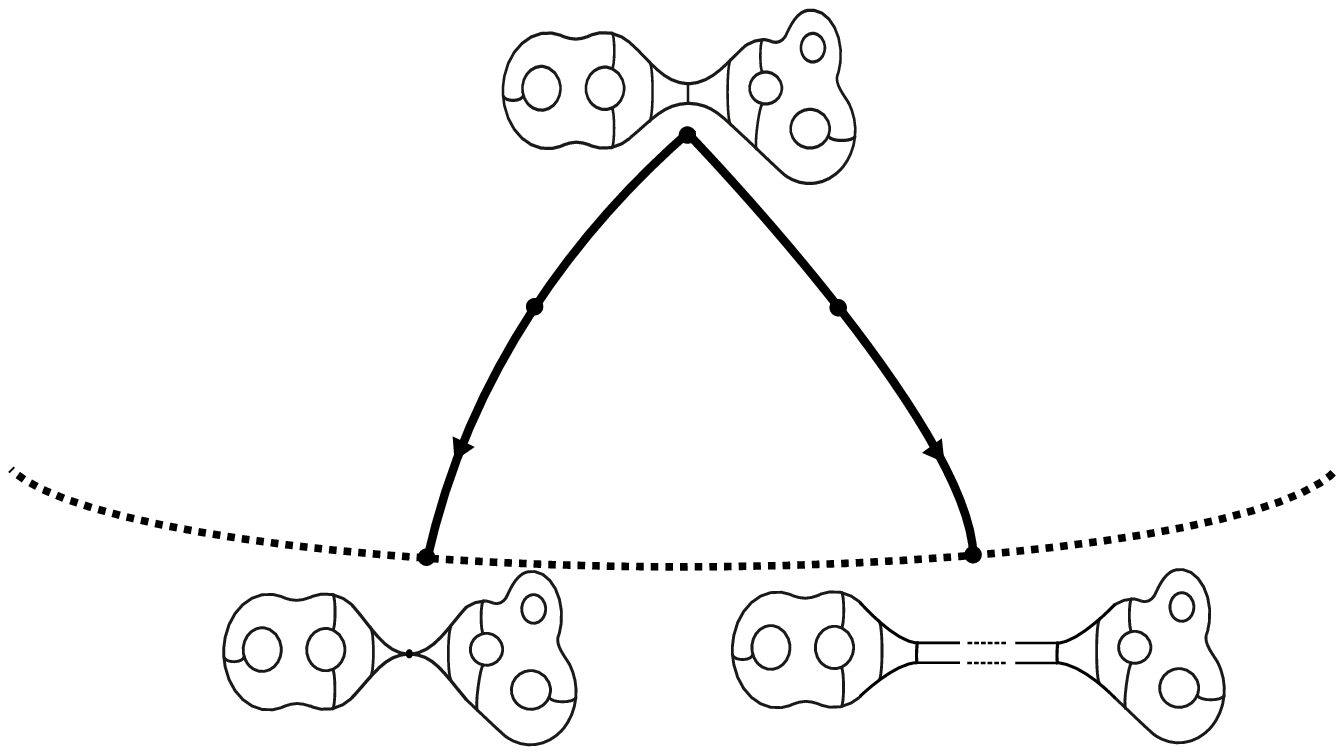}}}
\caption{Starting with a surface $S$ that has a separating simple closed geodesic $\gamma$, we approach the boundary $\partial \mathcal{M}_g$ of moduli space in two different ways.}
\label{fig:boundary}
\end{figure}

\textit{Grafting construction}\;  This is more widespread in the literature (see e.g.\ \cite{McM98} or \cite{DW07} for an account). The idea is to cut $S$ open along $\gamma$ and insert a Euclidean cylinder  $[0,l] \times \gamma$, $l \geq 0$ with the appropriate twist parameter (details below \beqref{eq:Lgamma}). The resulting family of surfaces $\big(S(l)\big)_{l \geq 0}$ is a path in $\mathcal{M}_g$ that converges to a boundary point which is a surface $S^{\sG}$ of genus $g-1$ with two cusps if $\gamma$ is nonseparating and a pair of surfaces $S_1^{\sG}, S_2^{\sG}$ of genera $g_1$, $g_2$ each with one cusp if $\gamma$ is separating. Here too we have natural marking homeomorphisms $\phi_{(l)}: S \to S(l)$ and their limits $\phi^{\sG} : S \smallsetminus \gamma \to S^{\sG}$ respectively, $\phi_i^{\sG}: S_i \smallsetminus \gamma \to S_i^{\sG}$, $i=1,2$. 

The hyperbolic metrics in the conformal classes of $S^{\sG}$, respectively $S_1^{\sG}$, $S_2^{\sG}$ are given by the Uniformization theorem, but there is no explicit way known to compute them. This ``drawback'' is compensated by the better comparison with the limit Jacobians in the results below.

Figure \ref{fig:boundary} illustrates the two ways to go to the  boundary in the case of a separating geodesic. The limit surfaces are depicted as noded surfaces with the two cusps pasted together at the node, i.e. along their ideal points at infinity.

\medskip

{\bf Separating case.}\; This is the context of the question asked in the beginning. We shall compare the Jacobian of $S$ with the direct products of the Jacobians of $S_1^{\sF}, S_2^{\sF}$ respectively, $S_1^{\sG}, S_2^{\sG}$ in terms of Gram period matrices.

To this end we choose the homology basis ${\rm A} = \big(\alpha_i \big)_{i=1,\dots,2g}$ in such a way that $\alpha_1, \dots, \alpha_{2g_1} \subset S_1$ and $\alpha_{2g_1+1}, \dots, \alpha_{2g} \subset S_2$. On $S_1^{\sF}$, more precisely on the one point compactification $\widebar{1.5}{S_1^{\sF}}$, the homology classes $[\phi_1^{\sF} \circ \alpha_i]$, $i=1, \dots, 2g_1$ form a canonical homology basis. We let $\big(\sigma_i^{\sF}\big)_{i=1,\dots,2g_1}$ be the corresponding dual basis of harmonic forms on $\widebar{1.5}{S_1^{\sF}}$ and denote the associated Gram period matrix by $P_{S_1^{\sF}}$. The matrices $P_{S_2^{\sF}}$, $P_{S_1^{\sG}}$, $P_{S_2^{\sG}}$ are defined in the same way.

With this setting we have the following result, where the members in the block matrices 
${
\left(\begin{smallmatrix}
R_1&\Omega\\
{}^{\rm t} \Omega&R_2
\end{smallmatrix}\right)}
$
are enumerated consecutively disregarding their relative positions in the blocks, so that e.g. the first line is $r_{11}, \dots, r_{1,2g_1}, \omega_{1,2g_1+1}, \dots, \omega_{1,2g}$ and the last is $\omega_{2g,1}, \dots, \omega_{2g,2g_1}, r_{2g,2g_1+1}, \dots, r_{2g,2g}$.

\begin{thm}\label{thm:small_scg_intro}
Assume that the separating geodesic $\gamma$ has length $\ell(\gamma) \leq \frac{1}{2}$. Then
 \begin{equation*}%
P_S=
\left(\rule{-2pt}{0pt} {\begin{array}{*{20}c}
   P_{S_1^{\sF}} & 0  \\
   0 & P_{S_2^{\sF}}  
\end{array}} \rule{-2pt}{0pt}\right)
+
\left(\rule{-2pt}{0pt} {\begin{array}{*{20}c}
   R_1^{\sF} & \Omega  \vphantom{P_{S_1^{\sF}}}\\
   {}^{\rm t} \Omega & R_2^{\sF}  \vphantom{P_{S_1^{\sF}}}
\end{array}} \rule{-2pt}{0pt}\right),
\quad
P_S=
\left(\rule{-2pt}{0pt} {\begin{array}{*{20}c}
   P_{S_1^{\sG}} & 0  \\
   0 & P_{S_2^{\sG}}  
\end{array}} \rule{-2pt}{0pt}\right)
+
\left(\rule{-2pt}{0pt} {\begin{array}{*{20}c}
   R_1^{\sG} & \Omega  \vphantom{P_{S_1^{\sG}}}\\
   {}^{\rm t} \Omega & R_2^{\sG}  \vphantom{P_{S_1^{\sG}}}
\end{array}} \rule{-2pt}{0pt}\right),
\end{equation*}%
where the remainder matrices $\Omega = \big(\omega_{ij}\big)$, $R_k^{\sF} = \big(r_{ij}^{\sF} \big)$, $R_k^{\sG} = \big(r_{ij}^{\sG} \big)$ have the following bounds,
\begin{enumerate}
\itemup{i)} $\vert \omega_{ij} \vert \leq e^{-2\pi^2(\frac{1}{\ell(\gamma)}-\frac{1}{2})} \sqrt{p_{ii}p_{jj}\mathstrut}$,
\itemup{ii)} $\vert r_{ij}^{\sG}\vert  \leq  e^{-4\pi^2(\frac{1}{\ell(\gamma)}-\frac{1}{2})} \sqrt{p_{ii}p_{jj}\mathstrut}$,
\itemup{iii)} $\vert r_{ij}^{\sF}\vert  \leq  6 \ell(\gamma)^2 \sqrt{p_{ii}p_{jj}\mathstrut}$,
\end{enumerate}
and for the $p_{ii}$, $p_{jj}$ we may take either, the diagonal elements of $P_S$ or the diagonal elements of the corresponding limit matrix.
\end{thm}

\medskip

{\bf Nonseparating case.}\; Here we adapt the homology basis ${\rm A} = (\alpha_1, \dots, \alpha_{2g})$ of $S$ to the given setting by requiring that $\alpha_2$ coincides with the nonseparating geodesic $\gamma$. 

The nonseparating case is quite a different world. Two complications arise from the fact that the (compactified) limit surface has genus $g-1$ instead of $g$: the rank of the homology and the dimension of the Gram period matrix are not the same as for $S$ and, in addition, the condition $\int_{[\alpha_1]}\sigma_j = 0$ for the members $\sigma_j$ of the dual basis of harmonic forms disappears. 

A way to overcome the first difficulty could be to find a useful concept of $2g$ by $2g$ generalised Gram period matrix that can intrinsically be attributed to the limit surface. We did not quite succeed in this. However, in \bbref{Section}{sec:TwoCusps} we shall 
find an ersatz for the lack of such a matrix in form of a package of geometric invariants that are defined in terms of harmonic differentials. Based on this we shall define, a concept of ``blown up'' Gram period matrices consisting of a one parameter family of symmetric $2g$ by $2g$ matrices $P_{\sR}(\lambda)$, $\lambda > 0$, attributable to any Riemann surface $\sR$ of genus $g-1$ with two cusps that is marked by the selection of a homology basis and a ``twist at infinity'' (\bbref{Sections}{sec:TwistInf} and \bbref{}{sec:Pmatpar}). To formulate the main results we outline here the geometric invariants for  $P_{\sR}(\lambda)$ in the case of the grafting limit $\sR = S^{\sG}$.

First, we describe a variant of the grafting construction better adapted to our needs. Inside the standard collar $C(\gamma)$ introduced above there is a smaller such cylindrical neighborhood $C'(\gamma)$ with boundary curves $c_1$, $c_2$ of constant distance to $\gamma$ that have lengths $\ell(c_1) = \ell(c_2) = 1$.  Removing $C'(\gamma)$ from $S$ we obtain a compact surface $\frak{M}$ of genus $g-1$ with two boundary curves $c_1$, $c_2$, the ``main part'' of $S$. The closure of the hyperbolic cylinder $C'(\gamma)$ is conformally equivalent to the Euclidean cylinder $[0, L_{\gamma}] \times \Sp_1$, where $\Sp_1 = \R/\Z$ is the circle of length 1 and
 \begin{equation}\label{eq:Lgamma}%
L_{\gamma} = \frac{\pi}{\ell(\gamma)}-\frac{2\arcsin \ell(\gamma)}{\ell(\gamma)}  \text{ \ with \ } \lim_{\ell(\gamma) \to 0} \frac{2\arcsin \ell(\gamma)}{\ell(\gamma)} = 2
 \end{equation}%
(redefined in \beqref{eq:lengcyl}). From the conformal point of view, $S$ is the Riemann surface obtained out of $\frak{M}$ by re-attaching $[0,L_{\gamma}] \times \Sp_1$. We now define $S_L$, for any $L> 0$, to be the Riemann surface obtained out of $\frak{M}$ by attaching, in the same way, $[0,L] \times \Sp_1$ instead of $[0,L_{\gamma}] \times \Sp_1$. The short hand ``in the same way'' means that if points $p \in c_1$, $q \in c_2$ on the boundary of $\frak{M}$ are connected by the straight line  $[0,L_{\gamma}] \times \{y_0\}$ in $[0,L_{\gamma}] \times \Sp_1 \subset S$ then they are connected by the straight line $[0,L] \times \{y_0\}$ in $[0,L] \times \Sp_1 \subset S_L$.

The part $\big(S_L\big)_{L\geq L_{\gamma}}$ of the family of surfaces described in this way differs from the previously defined $\big(S(l)\big)_{l \geq 0}$ only by an additive change of parameter. The limit surface $S^{\sG}$ is $\frak{M}$ with two copies of the infinite Euclidean cylinders $[0,\infty) \times \Sp_1$ attached.

 \begin{figure}[b]
 \vspace{0pt}
 \begin{center}
 \leavevmode
 \SetLabels
(0.15*0.59) $\alpha_1^{\sG}$\\
(0.352*0.56) $\alpha_2^{\sG}$\\
(0.48*0.17) $\alpha_3^{\sG}$\\
(0.425*-0.05) $\alpha_4^{\sG}$\\
(0.62*0.50) $\alpha_5^{\sG}$\\
(0.585*-0.01) $\alpha_6^{\sG}$\\
(0.65*0.94) $\frak{M}$\\
(0.72*0.25) $S^{\sG}$\\
 \endSetLabels
 \AffixLabels{
 \includegraphics{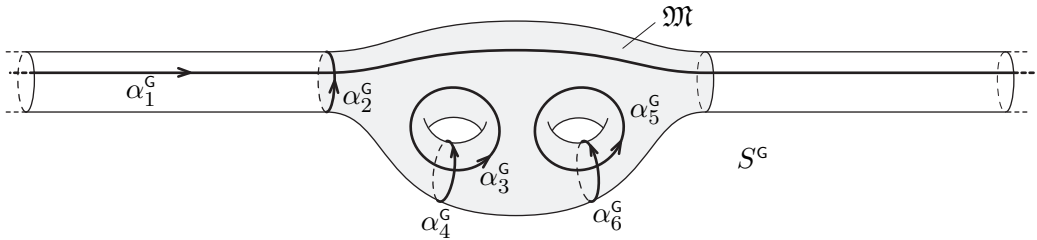} }
 \end{center}
 \vspace{-0pt}
 \caption{\label{fig:curvesSG} The main part $\frak{M}$ with two copies of $[0,) \times \Sp_1$ attached.}
 \end{figure}

The natural marking homeomorphism $\phi_L : S \to S_L$ is taken to be the mapping that acts as the identity from $\frak{M} \subset S$ to $\frak{M} \subset S_L$ and performs a linear stretch from $[0,L_{\gamma}] \times \Sp_1$ to $[0,L] \times \Sp_1$. In the limit,  $\phi^{\sG}:S\smallsetminus \gamma \to S^{\sG}$ is defined similarly but with the stretching of the cylinders fixed in some nonlinear way. 

For any $L>0$ the mapping $\phi_L$ sends the given homology basis on $S$ to a homology basis on $S_L$ and we have the associated Gram period matrix $P_{S_L}$. For $S^{\sG}$ this is no longer fully the case and we make the following adaptions. For $j=3,\dots,2g$, we set $\alpha_j^{\sG} = \phi^{\sG}(\alpha_j)$. These curves form a canonical homology basis of the compactified surface 

 \begin{equation*}%
\widebar{1.5}{S^{\sG}} = S^{\sG} \cup\{q,q'\},
 \end{equation*}%
where $q$, $q'$ symbolise the two ideal points at infinity of $S^{\sG}$. There is the dual basis of harmonic forms $\big(\sigma_j^{\sG}\big)_{j=3,\dots,2g}$ on $\widebar{1.5}{S^{\sG}}$ and we denote by

 \begin{equation*}%
P_{\widebar{4}{S^{\sG}}} = \Big(\;q_{ij}^{\sG}\;\Big)_{i,j=3,\dots,2g}
 \end{equation*}%
the corresponding Gram period matrix. We append two curves to the system $\alpha_3^{\sG},\dots,\alpha_{2g}^{\sG}$, setting $\alpha_1^{\sG} = \phi^{\sG}(\alpha_1')$, $\alpha_2^{\sG} = \phi^{\sG}(\alpha_2')$, where $\alpha_1'$ is $\alpha_1$ minus the intersection point with $\gamma = \alpha_2$, and $\alpha_2'$ is one of the boundary curves of $\frak{M}$ oriented in such a way that it lies in the homology class of $\alpha_2$. On $\widebar{1.5}{S^{\sG}}$ $\alpha_2^{\sG}$ is a null homotopic curve and $\alpha_1^{\sG}$ is an arc from $q$ to $q'$. 

A word is to be said about the latter. In \bbref{Section}{sec:TwistInf} we shall define a concept of ``twist at infinity'' by markings with particular classes of curves that go from cusp to cusp. From that point of view our limit surface is $S^{\sG}$ marked with the class of $\alpha_1^{\sG}$.

In addition to the curves we also append two harmonic forms $\mT_1^{\sG}$, $\sigma_2^{\sG}$, of infinite energy, to the system $\sigma_3^{\sG},\dots, \sigma_{2g}^{\sG}$. The first one is $\mT_1^{\sG} = df_1^{\sG}$, where $f_1^{\sG}$ is a (real valued) harmonic function on $S^{\sG}$ with logarithmic poles in $q$, $q'$ (\cite[Theorem II.4.3]{FK92}). It is made unique by the condition that $\int_{\alpha_2^{\sG}} \star \mT_1^{\sG} = 1$. Since $\mT_1^{\sG}$ is exact all its periods vanish. The second form is defined by the condition that $\sigma_2^{\sG}-\star \mT_1^{\sG}$ is harmonic on $\widebar{1.5}{S^{\sG}}$ and that it has the periods
 \begin{equation}\label{eq:line2}%
\int_{\alpha_j^{\sG}}\sigma_2^{\sG}=\delta_{2j}, \quad j=2,\dots,2g.
 \end{equation}%
Even though $\sigma_2^{\sG}$ is not an $L^2$-form the following integrals are well defined 

\begin{equation*}
q_{2j}^{\sG}=q_{j2}^{\sG}=\int_{S^{\sG}} \sigma_2^{\sG} \wedge \star \sigma_{j}^{\sG}, \quad j=3,\dots,2g.
 \end{equation*}%
For $j=2$ the above integral does not converge. Hence, there is no $q_{22}^{\sG}$. But one can define the ``gist of  $q_{22}^{\sG}$'': since $\sigma_2^{\sG}$ and $\star \mT_1^{\sG}$ have the same singularities the following quantity is well defined,

 \begin{equation*}
\pi_{22}^{\sG} = \int_{S^{\sG}} \Vert \sigma_2^{\sG} - \star\mT_1^{\sG} \Vert^2.
 \end{equation*}%
In a certain way $\pi_{22}^{\sG}$ measures how much additional energy it costs for $\sigma_2^{\sG}$ to satisfy the duality condition \beqref{eq:line2}. There is also a ``gist of $p_{11}^{\sG}$'' : if we denote, for $L > 0$, by $S_L^{\sG} \subset S^{\sG}$ the subsurface obtained out of $\frak{M}$ by attaching two copies of $[0, L/2] \times \Sp_1$, then the energy $E_{S_L^{\sG}}(\mT_1^{\sG})$ grows asymptotically like \textit{constant} + $L$ and we set 

 \begin{equation*}
\frak{m}^{\sG} = \lim_{L\to \infty}(E_{S_L^{\sG}}(\mT_{1}^{\sG}) - L).
 \end{equation*}%
Finally, we introduce the following limit periods, where we point out that the integral is well defined also for $j = 2$ despite the singularities of $\sigma_2^{\sG}$.%

 \begin{equation*}
\kappa_j^{\sG} = -\int_{\alpha_1^{\sG}}\sigma_j^{\sG}, \quad j=2,\dots,2g.
 \end{equation*}%
For an overview we put the invariants together into an array.%

 \begin{equation}\label{eq:ArrayPG}%
\frak{P}_{S^{\sG}}=
\boxed{\begin{matrix}
\frak{m}^{\sG}&\kappa_2^{\sG}&\kappa_3^{\sG}&\dots&\kappa_{2g}^{\sG}\\
\kappa_{2}^{\sG}&\pi_{22}^{\sG}&q_{23}^{\sG}&\dots&q_{2,2g}^{\sG}\\
\kappa_3^{\sG}&q_{32}^{\sG}   \\
\vdots & \vdots&&\smash{\raisebox{7pt}{$P_{\widebar{4}{S^{\sG}}}$}}\\
\kappa_{2g}^{\sG}&q_{2g,2}^{\sG}
\end{matrix}
}
 \end{equation}%
$\frak{P}_{S^{\sG}}$ is not meant to be regarded as a matrix. However, we shall define a one parameter family of ``true'' matrices out of it in \beqref{eq:ErsatzThmIntro} and, more generally, in \bbref{Section}{sec:Pmatpar}.

All entries of $\frak{P}_{S^{\sG}}$ except $\frak{m}^{\sG}$ are defined intrinsically by the conformal structure of $S^{\sG}$ and the given curve system $\alpha_1^{\sG}, \dots, \alpha_{2g}^{\sG}$. The entry $\frak{m}^{\sG}$, however, makes use of the description of $S^{\sG}$ in terms of the main part $\frak{M}$ of $S$.

In contrast to the case of \bbref{Theorem}{thm:small_scg_intro} the bounds for the error terms in the nonseparating case do not have a simple expression. We shall therefore use $\mathcal{O}$-terms to formulate the results. We make the following definition to indicate on what quantities the bounds depend: for real valued functions $\rho(x)$, $F(x)$, $x \in (0, + \infty)$, we say that 
 \begin{equation}\label{eq:OAsymbol}%
\rho(x) = \mathcal{O}_{\rm{A}}(F(x)),
 \end{equation}%
if $\vert \rho(x)\vert \leq c_{\rho} F(x)$, $x \in (0, + \infty)$, for some constant $c_{\rho}$ that can be explicitly estimated in terms of $g$, a lower bound on $sys_{\gamma}(S)$ and an upper bound on the lengths $\ell(\alpha_1), \dots, \ell(\alpha_{2g})$ on $S$, where $\sys_{\gamma}(S)$ is the length of the shortest simple closed geodesic on $S$ different from $\gamma$.

For the grafting limit the result is as follows.

\begin{thm}\label{thm:nonsepGIntro}%
Let $p_{ij}(S_L)$ be the entries of the Gram period matrix $P_{S_L}$. Then for any $L \geq 4$ we have the following comparison,
\begin{enumerate}
\item[{}]%
$p_{11}(S_L) = \dfrac{1}{\frak{m}^{\sG}+L} + \mathcal{O}_{\rm A}(\frac{1}{L}e^{-2\pi L})$, \quad
$p_{12}(S_L) = \dfrac{\kappa_2^{\sG}}{\frak{m}^{\sG}+L} + \mathcal{O}_{\rm A}(\frac{1}{L}e^{-\pi L}),$\\
$p_{1j}(S_L) = \dfrac{\kappa_j^{\sG}}{\frak{m}^{\sG}+L} + \mathcal{O}_{\rm A}(\frac{1}{L}e^{-2\pi L})$, \quad $j=3, \dots, 2g$,
\item[{}]%
$p_{22}(S_L) = \frak{m}^{\sG}+L + \pi_{22}^{\sG} + \dfrac{(\kappa_2^{\sG})^2}{\frak{m}^{\sG}+L} + \mathcal{O}_{\rm A}(\frac{1}{L}e^{-\pi L})$,
\item[{}]%
$p_{ij}(S_L)=q_{ij}^{\sG} + \dfrac{\kappa_i^{\sG}\kappa_j^{\sG}}{\frak{m}^{\sG}+L}+\mathcal{O}_{\rm A}(e^{-2\pi L})$, \quad $i,j = 2,\dots,2g$, $(i,j) \neq (2,2)$.
\end{enumerate}
\end{thm}%

For $S$ itself in this theorem one has to take $L=L_{\gamma}$ with $L_{\gamma}$ as in \beqref{eq:Lgamma}.

For the Fenchel-Nielsen limit the array of constants $\frak{P}_{S^{\sF}}$ is similar to $\frak{P}_{S^{\sG}}$ except that it shall be defined in terms of the hyperbolic metric of $S^{\sF}$. In analogy to \beqref{eq:Lgamma} we abbreviate, for $t>0$, 

 \begin{equation*}
L_t = \frac{\pi}{t}-\frac{2\arcsin t}{t}  \text{ \ with \ } \lim_{t \to 0} \frac{2\arcsin t}{t} = 2.
 \end{equation*}%
The result is then as follows, where for $S$ itself one has to take $t = \ell(\gamma)$.

 \begin{thm}\label{thm:nonsepFIntro}%
For any surface in the family $\big(S_t\big)_{t\leq \frac{1}{4}}$ of the Fenchel-Nielsen construction we have the comparison.
\begin{enumerate}
\item[{}]%
$p_{11}(S_t) = \dfrac{1}{\frak{m}^{\sF}+L_t} + \mathcal{O}_{\rm A}(t^4)$, \quad
$p_{12}(S_t) = \dfrac{\kappa_2^{\sF}}{\frak{m}^{\sF}+L_t} + \mathcal{O}_{\rm A}(t^2)$\\ $p_{1j}(S_t) = \dfrac{\kappa_j^{\sF}}{\frak{m}^{\sF}+L_t} + \mathcal{O}_{\rm A}(t^3)$, \quad $j=3, \dots, 2g$,
\item[{}]%
$p_{22}(S_t) = \frak{m}^{\sF}+L_t + \pi_{22}^{\sF} + \dfrac{(\kappa_2^{\sF})^2}{\frak{m}^{\sF}+L_t} + \mathcal{O}_{\rm A}(t^2)$,
\item[{}]%
$p_{ij}(S_t)=q_{ij}^{\sF} + \dfrac{\kappa_i^{\sF}\kappa_j^{\sF}}{\frak{m}^{\sF}+L_t}+\mathcal{O}_{\rm A}(t^2)$, \quad $i,j = 2,\dots,2g$, $(i,j) \neq (2,2)$.
\end{enumerate}.
 \end{thm}%

In \bbref{Theorem}{thm:nonsepGIntro} all occurrences of $L$ are in the combined form $\frak{m}^{\sG}+L$ (and a similar remark holds for \bbref{Theorem}{thm:nonsepFIntro}). This observation suggests to translate the array $\frak{P}_{S^{\sG}}$ in \beqref{eq:ArrayPG} (and in a similar way $\frak{P}_{S^{\sF}}$) into a parametrized family of matrices $P_{S^{\sG}}(\lambda)$, $\lambda \in (0,+\infty)$, that may be regarded as an ersatz for the missing Gram period matrix for $S^{\sG}$.

For the translation we regard the constants $q_{ij}^{\sG}$ as constant functions $q_{ij}^{\sG}(\lambda) = q_{ij}^{\sG}$, $i,j = 2,\dots, 2g$, $(i,j) \neq (2,2)$; then define $q_{22}^{\sG}(\lambda) = \pi_{22}^{\sG}+\lambda$; and, for completion, $q_{1j}^{\sG}(\lambda) = q_{j1}^{\sG}(\lambda) = 0$, $j=1,\dots,2g$. We also complete the list $\kappa_2^{\sG}, \dots, \kappa_{2g}^{\sG}$ setting $\kappa_1^{\sG} = 1$. The definition then is
 \begin{equation}\label{eq:ErsatzPSGIntro}%
P_{S^{\sG}}(\lambda) = \left(\vphantom{\rule{0pt}{15pt}}q_{ij}^{\sG}(\lambda)\right)_{i,j=1, \dots, 2g}
+
\left(\frac{\kappa_i^{\sG}\kappa_j^{\sG}}{\lambda}\right)_{i,j=1, \dots, 2g}
 \end{equation}%
\bbref{Theorem}{thm:nonsepGIntro} now states that 
 \begin{equation}\label{eq:ErsatzThmIntro}%
P_{S_L}=P_{S^{\sG}}(\frak{m}^{\sG}+L) + \rho^{\sG}(L),
\end{equation}%
with the bounds for the entries of $\rho^{\sG}(L)$ as given. There is a similar expression for \bbref{Theorem}{thm:nonsepFIntro}.

We shall prove in \bbref{Corollary}{cor:symplectic} that the $P_{S^{\sG}}(\lambda)$ are symplectic matrices. We do not know, however, whether or not they are Gram period matrices of compact Riemann surfaces.
\bigskip

The paper is divided into seven parts: \bbref{Section}{sec:Introd} is the introduction. In  \bbref{Section}{sec:MapDeg} we construct mappings used to embed Riemann surfaces with a small geodesic into the limit surfaces with cusps obtained by the Fenchel-Nielsen construction. These mappings shall be used to compare the energies of corresonding harmonic forms with each other.
In \bbref{Section}{sec:HarmCy} we give decay and dampening down estimates for harmonic forms on flat cylinders. These are later used to provide the exponential decay estimates given in \bbref{Theorem}{thm:energy_distribution} - \bbref{}{thm:nonsepGIntro}. \bbref{Section}{sec:SepCas} treats the case where the shrinking curve $\gamma$ is separating and \bbref{Sections}{sec:sigT1} - \bbref{}{sec:TwoCusps} the case where $\gamma$ is nonseparating. \bbref{Theorems}{thm:nonsepGIntro} and \bbref{}{thm:nonsepFIntro} are proven as \bbref{Theorems}{thm:nonsepG} and \bbref{}{thm:nonsepF} in \bbref{Section}{sec:ProofThmNonsep}.

At various places we estimate energies of harmonic forms via test forms and then simplify the results ``by elementary considerations''. For the latter it is -- and was -- helpful to use a computer algebra system.

\section{Mappings for degenerating Riemann surfaces}\label{sec:MapDeg}

In this section we construct mappings used to embed Riemann surfaces with a small geodesic into the limit surfaces with cusps obtained by the Fenchel-Nielsen construction. These mappings shall be used to compare the energies of corresonding harmonic forms with each other.

\subsection{Collars and cusps}\label{sec:ColCus}

We first collect a number of facts about collars and cusps that follow from the so-called \emph{Collar Theorem}. For proofs we refer e.g.\ to \cite[Chapter 4]{Bu92}. In the present section $S$ is a surface of signature $(g,m;n)$ with a metrically complete hyperbolic metric for which the $m$ boundary components are simple closed geodesics and the $n$ punctures are cusps. We assume, furthermore, that $2g + m + n \geq 3$. 

For any simple closed geodesic $\gamma$ on $S$ (on the boundary or in the interior) the \emph{standard collar} $C(\gamma)$ is defined as the neighbourhood 

 \begin{equation}\label{eq:stacol}%
C(\gamma) :=\left\{p \in S \mid \dist(p,\gamma) < \cl(\ell(\gamma)) \right\}, \quad \textup{where}
 \end{equation}%
 \begin{equation}\label{eq:clfctn}%
\cl(s):=  \arcsinh\left(\frac{1}{\sinh(\frac{s}{2})}\right) > \ln\left(\frac{4}{s}\right).
 \end{equation}%
By the collar theorem $C(\gamma)$ is an embedded cylinder; moreover, the geodesic arcs of length $\cl(\gamma)$ emanating perpendicularly from $\gamma$ are pairwise disjoint (except for the common initial points of pairs of opposite arcs). This allows us to use \emph{Fermi coordinates} for points in $C(\gamma)$ by taking $\gamma$ as base curve, parametrised in the form $t \mapsto \gamma(t)$, $t \in [0,1]$, with constant speed $\ell(\gamma)$. The parameter $t$ will also be interpreted as running through $\Sp_1$, where we set
 \begin{equation}\label{eq:Sp1}%
\Sp_1 := \R/\Z.
 \end{equation}%
For any $p \in C(\gamma)$ we then have Fermi coordinates $(\rho,t)$ defined such that $\vert \rho \vert$ is the distance from $p$ to $\gamma$ and $p$ lies on the orthogonal geodesic through $\gamma(t)$. If $\gamma$ lies in the interior of $S$ then all $\rho$ have negative signs on one side of $\gamma$ and positive signs on the other. For the case where $\gamma$ is a boundary geodesic, we adopt the \emph{convention} that all $\rho$  shall be nonpositive. The metric tensor in these coordinates is
 \begin{equation}\label{eq:tensorcollar}%
g_C = d \rho^2 + \ell(\gamma)^2 \cosh^2(\rho) d t^2.
 \end{equation}%
Hence, if $\gamma$ lies in the interior of $S$ then $C(\gamma)$ is isometric to the cylinder 
 \begin{equation}\label{eq:coocollar}%
({-}\cl(\ell(\gamma)),\cl(\ell(\gamma))){}\times \Sp_1
 \end{equation}%
endowed with the metric tensor \beqref{eq:tensorcollar}. If $\gamma$ is a boundary geodesic then $C(\gamma)$ is isometric to $({-}\cl(\ell(\gamma)),0]\times \Sp_1$ endowed with this metric. In this latter case we also call $C(\gamma)$ a \emph{half collar}.

In a similar way, also by the collar theorem, there exists for any puncture $q$ (to be understood as an isolated ideal boundary point at infinity) an open neighbourhood $V(q)$ that is isometric to the cylinder 
\begin{equation}\label{eq:coocusp}
({-}\ln(2),{+}\infty){}\times \Sp_1
 \end{equation}%
endowed with the metric tensor
\begin{equation}\label{eq:tensorcusp}
g_V = d r^2 + e^{-2r} d t^2.
\end{equation}%
Under this isometry any curve $\{r\} \times \Sp_1$, for $r \in ({-}\ln(2),{+}\infty)$, corresponds to a horocycle $\eta_{\lambda}$ of length $\lambda = e^{-r}$ on $V(q)$ and the couples $(r,t)$ are Fermi coordinates for points in $V(q)$ based on the curve $\eta_1$. 

We shall also look at parts of cusps: for $a > b \geq 0$ the \emph{truncated cusp} $V_a^b$ is defined as the annular region

\begin{equation}\label{eq:trunccusp}
V_a^b := \{p \in V(q) \mid  \textup{$p$ lies between $\eta_a$ and $\eta_b$} \},
\end{equation}

where `between'' shall allow that $p$ lies on $\eta_a$, or $\eta_b$. When $b=0$ the truncated cusp becomes a punctured disk and we denote it by $V_a$.

\bigskip

Finally, the collar theorem further states that all cusps are pairwise disjoint, no cusp intersects a collar, and collars belonging to disjoint simple closed geodesics are disjoint.
\subsection{Conformal mappings of collars and cusps}\label{sec:ConfCC}

We shall make use of the standard conformal mappings $\psi_{\gamma}$, $\psi_q$ that map collars and cusps into the flat cylinder
 \begin{equation}\label{eq:flatcylinder}%
Z = ({-}\infty,{+}\infty) \times \Sp_1, \quad \text{with the Euclidean metric $g_E = d x^2 + d y^2$,}
 \end{equation}%
where we use $(x,y)$ as coordinates for points in $Z$ with $x \in ({-}\infty, {+}\infty)$, $y \in \Sp_1 = \R/\Z$, and $y$ is treated as a real variable. It is straightforward to check that for any simple closed geodesic $\gamma$ on $S$ the mapping $\psi_{\gamma} : C(\gamma) \to Z$ defined in Fermi coordinates by 
 \begin{equation}\label{eq:psigamma}%
\psi_{\gamma}(\rho,t) = (F_{\gamma}(\rho),t)
 \end{equation}%
with
 \begin{equation}\label{eq:Fgamma}
 \begin{aligned}%
F_{\gamma}(\rho) = \frac{2}{\ell(\gamma)} \arctan\left(\tanh\left(\frac{\rho}{2}\right)\right) &= \text{sign}(\rho)\frac{1}{\ell(\gamma)} \arccos\left( \frac{1}{\cosh(\rho)}\right)\\
&= \text{sign}(\rho)\frac{1}{\ell(\gamma)} \left\{\frac{\pi}{2}-\arcsin\left( \frac{1}{\cosh(\rho)}\right)\right\}
 \end{aligned}
 \end{equation}%
is conformal. Similarly, for any cusp $V(q)$ we have a conformal mapping $\psi_q : V(q) \to Z$ defined in Fermi coordinates by 
 \begin{equation}\label{eq:psiq}%
\psi_q (r,t) = (e^r,t).
 \end{equation}%
For later use we note that by  \beqref{eq:stacol}, \beqref{eq:clfctn}, \beqref{eq:Fgamma} $\psi_{\gamma}$ maps $C(\gamma)$ onto $({-}M,M) \times \Sp_1$, where $M = M(\ell(\gamma))$ satisfies

 \begin{equation}\label{eq:Lleta}%
M(\ell(\gamma)) = \frac{1}{\ell(\gamma)} \left\{\frac{\pi}{2} - \arcsin\left(\tanh\left(\frac{\ell(\gamma)}{2}\right)\right)\right\} \geq \frac{\pi}{2\ell(\gamma)} -\frac{1}{2},
 \end{equation}%
and the inequality follows from elementary consideration. ($M$ is not the same as $L_{\gamma}$ defined in \beqref{eq:Lgamma} and later used in \beqref{eq:lengcyl}.) In the next subsection we shall combine $\psi_{\gamma}$ with $\psi_q^{-1}$ to conformally map half collars into cusps. 
 
\subsection{Mappings of Y-pieces}\label{sec:MapYpc}

The results of this subsection are presented in  \cite{bmms14}. We summarize them for convenience and slightly extend them for the present needs. A \emph{Y-piece} or \emph{hyperbolic pair of pants} is a surface as in the preceding subsection that has signature $(0,3;0)$. To include limit cases we shall also admit signatures $(0,2;1)$ and $(0,1;2)$.\\
We denote by $Y=Y_{l_1,l_2,l_3}$ the Y-piece with boundary geodesics $\gamma_1,\gamma_2,\gamma_3$ of lengths $\ell(\gamma_i)=l_i$, for $i \in \{1,2,3\}$. For any pair $\gamma_{i-1}, \gamma_i$, $i=1,2,3$, (indices $\mod 3$) there is a unique connecting geodesic arc $a_i$ orthogonal to $\gamma_{i-1}$ and $\gamma_i$ at its end points. These arcs are called the \emph{common perpendiculars} or \emph{orthogeodesics} of $Y$. They decompose $Y$ into two identical right angled geodesic hexagons, and there is an orientation reversing isometry $\sigma_Y : Y \to Y$, the \emph{natural symmetry} of $Y$, that keeps the orthogeodesics pointwise fixed.

The boundary geodesics are parametrized with constant speed $\gamma_i: [0,1] \rightarrow \partial Y$, with positive boundary orientation (with respect to a fixed orientation on $Y$) and such that $\gamma_i(0)$ coincides with the endpoint of $a_i$ (see \bbref{Fig.}{fig:quasiY}). We call this the \textit{standard parametrization}. By the aforementioned collar theorem the sets

\begin{equation*}
   \mathcal{C}_i = \{ p \in Y \mid \dist(p,\gamma_i) < \cl(\ell(\gamma_i)) \} 
\end{equation*}
are pairwise disjoint embedded cylinders. Moreover, for $0 \leq \rho \leq \cl(\ell(\gamma_i))$ the equidistant curves

 \begin{equation*}%
\gamma_i^{\rho} = \{ p \in Y \mid \dist(p,\gamma_i) = \rho \}
 \end{equation*}%
are embedded circles. We parametrize them with constant speed in the form $\gamma_i^{\rho}:[0,1] \rightarrow Y_{l_1,l_2,l_3}$, such that there is an orthogonal geodesic arc of length $\rho$ from $\gamma_i(t)$ to $\gamma_i^{\rho}(t), t \in [0,1]$. This too shall be called a \textit{standard parametrization}.\\
We extend these conventions to degenerated Y-pieces with cusps writing symbolically `$\ell(\gamma_i)=0$', if the $i$-th end is a puncture. In this case the equidistant curves are horocycles and the collar $\mathcal{C}_i$ is isometric to a cusp whose boundary length is equal to 2.\\
We also consider \textit{restricted Y-pieces}, where parts of the collars have been cut off: Take $Y_{l_1,l_2,l_3}$. Select in each $\mathcal{C}_i$ an equidistant curve $\beta_i = \gamma_i^{\rho_i}$ (respectively, a horocycle if $\mathcal{C}_i$ is a cusp) of some length $\lambda_i$, possibly $\beta_i= \gamma_i$, and cut away the outer part of the collar along this curve. The thus \textit{restricted Y-piece}, $Y_{l_1,l_2,l_3}^{\lambda_1,\lambda_2,\lambda_3}$, is the closure of the connected component of $Y_{l_1,l_2,l_3} \smallsetminus \{\beta_1,\beta_2,\beta_3\}$ that has signature $(0,3;0)$.\\
Our mappings shall have the following two properties.  A homeomorphism

 \begin{equation*}%
 \phi: Y \rightarrow Y'
 \end{equation*}%
of, possibly restricted, Y-pieces is called {\textit{boundary coherent}} if for corresponding boundary curves $\beta_i$ of $Y$ and $\beta'_i$ of $Y'$ in standard parametrization one has $\phi(\beta_i(t))=\beta'_i(t),t \in [0,1]$. The mapping $\phi$ is called \emph{symmetric} if it is compatible with the natural symmetries, i.e.\ if $\phi \circ \sigma_Y = \sigma_{Y'} \circ \phi$. Since the fixed point sets of $\sigma_Y$ and $\sigma_{Y'}$  are the orthogeodesics it follows, in particular, that a symmetric $\phi$ sends orthogeodesics to orthogeodesics.\\

We need, in particular, the following restriction of Y-pieces. Assume first that for some $i \in \{1,2,3\}$ we have  $0<\ell(\gamma_i) = \epsilon < 2$. Then we define the \emph{reduced collar}

 \begin{equation*}%
 \widehat{\mathcal{C}}_i = \{p \in Y \mid \dist(p, \gamma_i) < \widehat{w}(\epsilon)\},
 \end{equation*}%
where, using \beqref{eq:clfctn},
 \begin{equation}\label{eq:whatepsilon}%
\widehat{w}(\epsilon) := \log\left(\frac{2}{\epsilon}\right) < \cl(\epsilon)-\log 2.
 \end{equation}%
The inner boundary curve of $\widehat{\mathcal{C}}_i$ has length
 \begin{equation}\label{eq:echwhat}%
 \epsilon \cosh(\widehat{w}(\epsilon)) = 1 + \frac{\epsilon^2}{4}.
 \end{equation}%
By \beqref{eq:whatepsilon} we have $\widehat{\mathcal{C}}_i \subset \mathcal{C}_i$. For convenient notation we complete the definition by setting $\widehat{\mathcal{C}}_i = \emptyset$ if $\ell(\gamma_i) \geq 2$.

We now extend the definition of $\widehat{\mathcal{C}}_i$ to the case where $\gamma_i$ is a puncture. In this case $\widehat{\mathcal{C}}_i$ is defined as the subset of all points in the cusp $\mathcal{C}_i$ that lie outside the horocycle of length 1. With this set up we define, for all cases, the \emph{reduced} Y-piece as 

 \begin{equation*}
\widehat{Y} = Y \smallsetminus (\widehat{\mathcal{C}}_1 \cup \widehat{\mathcal{C}}_2 \cup \widehat{\mathcal{C}}_3).
 \end{equation*}%

We are ready to introduce the mappings. Let us first consider the case of a Y-piece where ``one boundary geodesic is shrinking'', i.e where $l_1,l_2 \geq 0$, and $l_3 = \epsilon > 0$ with $\epsilon$ small.  In \cite[Theorem 5.1]{bmms14}, it is shown that there exists a symmetric  boundary coherent quasiconformal homeomorphism $\phi : \widehat{Y}_{l_1,l_2,\epsilon} \to \widehat{Y}_{l_1,l_2,0}$ of dilatation $\leq 1+ 2\epsilon^2$. We extend $\phi$ to all of $Y$ by letting it be an isometry on $\widehat{\mathcal{C}}_1$, $\widehat{\mathcal{C}}_2$ and a conformal mapping on $\widehat{\mathcal{C}}_3$. To compute the image set we use the preceding subsection: the mapping $\psi_{\gamma}$ as in \beqref{eq:psigamma}
with $\gamma = \gamma_3$ maps the closure of $\hat{C}_3$ conformally to the flat cylinder $[-\hat{M},0] \times \Sp_1$, where by \beqref{eq:Fgamma} and \beqref{eq:echwhat}
 \begin{equation}\label{eq:Lhat}
\hat{M} = F_{\gamma_3}(\hat{w}(\epsilon)) = \frac{1}{\epsilon}\left(\frac{\pi}{2} - \arcsin\left(\frac{\epsilon}{1+\frac{\epsilon^2}{4}}\right)\right) > \frac{\pi}{2\epsilon}  - 1,   \text{\; for }  0<\epsilon \leq2,
 \end{equation}
and the inequality follows from elementary consideration. We then shift the cylinder forward to $[1, \hat{M}+1] \times \Sp_1$ and apply the inverse of the conformal mapping $\psi_q$ defined in \beqref{eq:psiq}, where $q$ is the corresponding puncture of $Y_{l_1,l_2,0}$. The combination of the three mappings is our extension of $\phi$ to $\widehat{C}_3$. From the above expression for $\hat{M}$ and the form of the metric tensor in \beqref{eq:tensorcusp} it follows that
 \begin{equation}\label{eq:epsilonhat}%
\phi(\widehat{C}_3) = V_1^{\hat{\epsilon}} \quad \textup{with} \quad \hat{\epsilon} := \frac{1}{\hat{M}+1}.
 \end{equation}%
In particular, the boundary geodesic $\gamma_3$ of length $\epsilon$ goes to the horocycle $h_{\hat{\epsilon}}$ of length $\hat{\epsilon}$, where asymptotically $\hat{\epsilon} \sim \frac{2}{\pi}\epsilon$, as $\epsilon \to 0$.

\begin{figure}[h!]
\SetLabels
\L(.145*.60) $Y_{l_1,l_2,\epsilon}$\\%
\L(.78*.58) $Y_{l_1,l_2,0}^{l_1,l_2,\hat{\epsilon}}$\\%
\L(.49*.67) $\phi_Y$\\%
\L(.064*.37) $\,\gamma_1(0)$\\%
\L(.315*.83) $\gamma_3(0)$\\%
\L(.37*.03) $\gamma_2(0)$\\%
\L(.515*.37) $\gamma'_1(0)$\\%
\L(.74*.76) $\,h_{\hat{\epsilon}}$\\%
\L(.81*.025) $\gamma'_2(0)$\\%
\L(.21*.49) $a_1$\\%
\L(.28*.17) $a_2$\\%
\L(.32*.48) $a_3$\\
\endSetLabels
\AffixLabels{%
\centerline{%
\includegraphics[height=5cm,width=12cm]{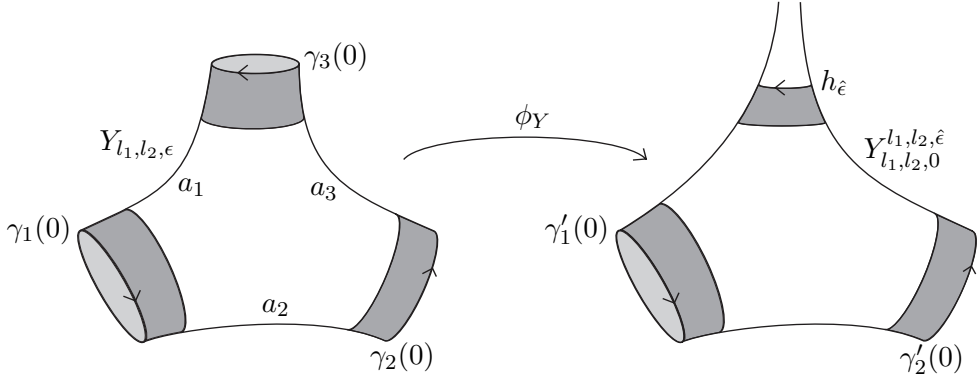}}}
\caption{$Y_{l_1,l_2,\epsilon}$ is quasi-conformally embedded into the Y-piece $Y_{l_1,l_2,0}$ with a cusp. The image of $\gamma_3$ is the horocycle $h_{\hat{\epsilon}}$ of length $\hat{\epsilon}$.}
\label{fig:quasiY}
\end{figure}

With $\phi$ thus extended we have the following variant of Theorem 5.1 in \cite{bmms14}:

\begin{thm}\label{thm:quasiY} Let $0 \leq l_1,l_2$, $0 < \epsilon \leq \frac{1}{2}$, and define $\hat{\epsilon}$ as in \beqref{eq:epsilonhat}. Then there exists a symmetric boundary coherent quasi-conformal homeomorphism

 \begin{equation*}%
 \phi_Y: Y_{l_1,l_2,\epsilon} \rightarrow Y_{l_1,l_2,0}^{l_1,l_2,\hat{\epsilon}}
 \end{equation*}%
with dilatation $q_{\phi_Y} \leq 1+ 2\epsilon^2$ and the property that the restriction of $\phi_Y$ to $\widehat{\mathcal{C}}_i$ is an isometry for $i=1,2$, and a conformal mapping for $i=3$.
\end{thm}

The theorem is illustrated in \bbref{Fig.}{fig:quasiY}. Note that the bound is independent of $l_1$ and $l_2$. Applying the theorem twice we also have the following:

\begin{thm}\label{thm:quasiYY} Let $0 \leq l_1$, $0 < \epsilon_2, \epsilon_3 \leq \frac{1}{2}$, and define $\hat{\epsilon}_2$, $\hat{\epsilon}_3$ as in \beqref{eq:epsilonhat}. Then there exists a symmetric boundary coherent quasi-conformal homeomorphism

 \begin{equation*}%
   \psi_Y: Y_{l_1,\epsilon_2,\epsilon_3} \rightarrow Y_{l_1,0,0}^{l_1,\hat{\epsilon}_2,\hat{\epsilon}_3}
 \end{equation*}%
with dilatation $q_{\psi_Y} \leq (1+ 2\epsilon_2^2)(1+ 2\epsilon_3^2)$ and the property that the restriction of $\psi_Y$ to $\widehat{\mathcal{C}}_i$ is an isometry for $i=1$, and a conformal mapping for $i=2,3$.
\end{thm}

\subsection{Mappings between $S^{\sF}$ and $S^{\sG}$}\label{sec:MpSFSG}
We extend the mappings of the preceding paragraph to quasi conformal mappings between $S^{\sF}$ and $S^{\sG}$. For future reference the setting is slightly more general. The term \emph{marking preserving}  in the next theorem is explained in the proof.
\begin{thm}\label{thm:psRGRF} Let $R$ be a hyperbolic surface of signature $(g,m;n)$ and assume that for some positive integer $l \leq m$ and some $\epsilon \in (0,\frac{1}{2}]$, the boundary geodesics $\gamma_1, \dots, \gamma_l$ satisfy

\begin{equation*}
\ell(\gamma_1), \dots, \ell(\gamma_l) \leq \epsilon.
\end{equation*}
Let $R^{\sG}$ be obtained by attaching flat cylinders $[0,\infty) \times \gamma_i$ to $\gamma_i$ for $i=1,\dots,l$, and let $R^{\sF}$ be obtained by choosing a partition of $R$ and shrinking $\gamma_1, \dots, \gamma_l$ to zero while keeping the remaining Fenchel-Nielsen parameters fixed. Then there exists a marking preserving quasi-conformal homeomorphism

\begin{equation*}
\psi : R^{\sG} \to R^{\sF}
\end{equation*}
with dilatation $q_{\psi} \leq (1+ 2\epsilon^2)^2$ that, moreover, embeds the cylinders $ [0,\infty) \times \gamma_i$ conformally into $R^{\sF}$ for $i=1,\dots, l$.
\end{thm}
\begin{proof}
By construction, we may label the Y-pieces of $R$ and $R^{\sF}$ and the geodesics along which they are pasted together in the form $Y_1, \dots, Y_q$ and $\gamma_{m+1}, \dots, \gamma_{m+p}$ respectively, $Y'_1, \dots, Y'_q$ and $\gamma'_{m+1}, \dots, \gamma'_{m+p}$ (with $p=3g-3+m+n$, $q=2g-2+m+n$) in such a way that, whenever $Y_i$ is pasted to $Y_j$ along $\gamma_k$ with some twist parameter $\vartheta_k$, then $Y'_i$ is pasted to $Y'_j$ along $\gamma'_k$ with twist parameter $\vartheta_k$. 

There exists then, in the same way as in \bbref{Section}{sec:Introd}, a marking homeomorphism $\phi^{\sF}: R \to R^{\sF}$ that sends the $Y_j$ to the corresponding $Y'_j$, $j=1,\dots,q$, and preserves the twist parameters. (This is also a marking homeomorphism in the sense of Teichm\"uller theory  with $R$ understood as base surface, e.g.\ \cite[Section 6.1]{Bu92}.) There is a similar marking homeomorphism $\phi^{\sG}: R \to R^{\sG}$ and we shall say that a homeomorphism $\varphi : R^{\sG} \to R^{\sF}$ is \emph{marking preserving} if the mappings $\varphi \circ \phi^{\sG}$ and $\phi^{\sF}$ are isotopic.

Now, if, for a given $i$,  the boundary geodesics of $Y_i$ are distinct from $\gamma_1, \dots, \gamma_l$, then there exists an isometry $\psi_{Y_i} : Y_i \to Y'_i$ that preserves the labelling of the boundary geodesics. If, on the other hand, some of the boundary geodesics of $Y_i$ are amongst $\gamma_1, \dots, \gamma_l$, then, by \bbref{Theorems}{thm:quasiY} and \bbref{}{thm:quasiYY}, there exists a symmetric boundary coherent quasi-conformal embedding $\psi_{Y_i} : Y_i \to Y'_i$ that preserves the labelling of the interior geodesics and has dilatation $q_{\psi_{Y_i}}\leq (1+ 2\epsilon^2)^2$. Since the mappings $\psi_{Y_i}$, for $i=1,\dots,q$, are boundary coherent, and since the twist parameters for the pastings in $R$ are the same as those for the corresponding pastings in $R^{\sF}$, it follows that the $\psi_{Y_i}$ match along the boundaries and together define a quasi-conformal embedding $\psi : R \to R^{\sF}$ with dilatation $q_{\psi}\leq (1+ 2\epsilon^2)^2$. 

The mapping $\psi$ sends $\gamma_1, \dots, \gamma_l$ to horocycles in the cusps of $R^{\sF}$ and, again by boundary coherence, match with the conformal mappings between the attached cylinders and the outer parts of the horocycles. We may thus extend $\psi$ to a $(1+ 2\epsilon^2)^2$-quasi conformal homeomorphism $\psi : R^{\sG} \to R^{\sF}$ that acts conformally on the cylinders $ [0,\infty) \times \gamma_i$. Furthermore the so extended $\psi$ sends the orthogeodesics of the pairs of pants of the partition of $R^{\sG}$ to the orthogeodesics of the pairs of pants of the partition of $R^{\sF}$ and is thus marking preserving. 
\end{proof}

\section{Harmonic forms on a  cylinder}\label{sec:HarmCy}%

In this section we give decay and dampening down estimates  for harmonic forms on the flat cylinder $Z = ({-}\infty,{+}\infty) \times \Sp_1$ (c.f. \beqref{eq:flatcylinder}), where $\Sp_1 = \R/\Z$.

We use $(x,y)$ as coordinates for points in $Z$ with $x \in ({-}\infty, {+}\infty)$ and $y \in \Sp_1$, where $y$ is treated as a real variable.
Any real harmonic function $h(x,y)$ on $Z$ may be written as a series
\begin{multline}\label{eq:hseries}
   h(x,y) = a_0 + b_0 x\\
+ \sum_{n=1}^{\infty} (a_n\cos(2\pi ny) + b_n\sin(2\pi ny)) e^{-2\pi nx}
+    \sum_{n=1}^{\infty} (c_n\cos(2\pi ny) + d_n \sin(2\pi ny))e^{2\pi nx}
\end{multline}
(e.g.\ \cite[ p. 253]{Ax86}) via the conformal mapping $(x,y) \to \exp(2\pi(x+iy))$). We decompose $h(x,y)$ into 
\begin{equation}\label{eq:nonlinear_flat}
 h(x,y) = a_0 + b_0 x + h^{\nlin}(x,y), \quad h^{\nlin}(x,y) = h^{-}(x,y) + h^{+}(x,y),
\end{equation}
where $h^{-}$ is the first sum in \beqref{eq:hseries} and $h^{+}$ the second.
We call $a_0 + b_0 x$ the \emph{linear part of $h$} and $h^{\nlin} = h^{-}\!+h^{+}$ the \emph{nonlinear part }. The differential of $h$ is given by 
 \begin{equation*}%
d h(x,y) = A(x,y) d x + B(x,y) d y
 \end{equation*}%
with
\begin{align*}
&A(x,y) =b_0 + 
2\pi \sum_{n=1}^{\infty} n(-a_n\cos(2\pi ny) -b_n \sin(2\pi ny))e^{-2\pi nx}\\
&\rule{20ex}{0pt}+  2\pi \sum_{n=1}^{\infty} n(c_n\cos(2\pi ny)+d_n \sin(2\pi ny))e^{2\pi nx}\\
&B(x,y)  = 2\pi \sum_{n=1}^{\infty}n(-a_n\sin(2\pi ny) + b_n \cos(2\pi ny))e^{-2\pi nx}\\
&\rule{20ex}{0pt}+ 2\pi \sum_{n=1}^{\infty}n( -c_n\sin(2\pi ny)+d_n\cos(2\pi ny)) e^{2\pi nx}.
\end{align*}
Its pointwise squared norm is
 \begin{equation*}
    \| d h(x,y) \|^2 = A^2(x,y) + B^2(x,y).
 \end{equation*}%
Integrating $\| d h(x,y) \|^2$  over $\Sp_1$ we obtain

 \begin{equation}\label{eq:energy_y}%
\begin{aligned}
\int \limits_0^1 \| d h(x,y) \|^2 \,d y &= b_0^2 + 4\pi^2 \sum_{n=1}^{\infty} n^2 (a_n^2  + b_n^2) e^{-4\pi nx} + 4\pi^2 \sum_{n=1}^{\infty} n^2(c_n^2 + d_n^2) e^{4\pi nx}\\
&= b_0^2 + 
\int \limits_0^1\| d h^{-}(x,y) \|^2 \,d y
+
\int \limits_0^1\| d h^{+}(x,y) \|^2 \,d y.
\end{aligned}
 \end{equation}%

We now prove the following pointwise and $L^2$-decay estimates that shall be used at various places. 
\begin{lem}\label{thm:lem_flat} Assume $l > 0$ and let $Z_l$ be the subset $Z_l = [-l,l] \times  \Sp_1$ of $Z$. For $\delta_{\lef}, \delta_{\rig}\geq 0$  with $\delta_{\lef} + \delta_{\rig} < 2l$, set
\begin{eqnarray*}
  D_{\lef} &=& \{(x,y) \in Z_l \mid  -l \leq x \leq -l+\delta_{\lef} \} \text{ \, \ } D_{\rig} = \{(x,y) \in Z_l \mid  l-\delta_{\rig} \leq x \leq l \} \text{ \ and \ } \\
  D_{int} &=& Z_l \smallsetminus \{ D_{\lef} \cup D_{\rig} \}    \text{ \ \ (see \bbref{Fig.}{fig:Euclidean_cyl}). \ \ }
\end{eqnarray*}
Then for any harmonic function $h$ on $Z_l$ with nonlinear part $h^{-}\!+h^{+}$ as in \beqref{eq:hseries}, \beqref{eq:nonlinear_flat} we have
\begin{enumerate}
\itemup{i)} $E_{D_{int}}(d h^{-}+d h^{+}) = E_{D_{int}}(d h^{-}) + E_{D_{int}}(d h^{+})$.
\itemup{ii)} $8\pi^2 \int_{D_{int}} | h^{-}|^2 \leq E_{D_{int}}(d h^{-}) \leq  e^{-4\pi \delta_{\lef}} E_{Z_l}(d h^{-})$.
\itemup{iii)}  $8\pi^2\int_{D_{int}} | h^{+}|^2 \leq E_{D_{int}}(d h^{+}) \leq  e^{-4\pi \delta_{\rig}} E_{Z_l}(d h^{+})$.
\itemup{iv)}  $4\pi^2 \int_{D_{int}} |h^{\nlin}|^2 \leq E_{D_{int}}(d h^{\nlin}) \leq  e^{-4\pi \delta} E_{Z_l}(d h^{\nlin})$,\; $\delta := \min\{\delta_{\lef},\delta_{\rig}\}$.
\itemup{v)} If $l \geq 1$ and $\delta_{\lef} \geq \frac{1}{2}$ then, for any $(x,y) \in D_{int}$,

 \begin{equation*}%
| h^{-}(x,y)|^2 < e^{-4\pi (x+l)}  E_{Z_l}(d h^{-}), \quad \| d h^{-}(x,y) \|^2 < 52 \pt e^{-4\pi(x+l)}  E_{Z_l}(d h^{-}).
 \end{equation*}%
\itemup{vi)} If $l \geq 1$ and $\delta_{\rig} \geq \frac{1}{2}$ then, for any $(x,y) \in D_{int}$,

 \begin{equation*}%
| h^{+}(x,y)|^2 < e^{-4\pi (l-x)}  E_{Z_l}(d h^{+}), \quad \| d h^{+}(x,y) \|^2 < 52 \pt e^{-4\pi (l-x)}  E_{Z_l}(d h^{+}).
 \end{equation*}%
\itemup{vii)} 
If $l \geq 1$ and $\delta:= \min\{\delta_{\lef}, \delta_{\rig}\} \geq \frac{1}{2}$, then for any $(x,y) \in D_{int}$,
 \begin{equation*}%
| h^{\nlin}(x,y)|^2 < 2\pt e^{-4\pi \delta}  E_{Z_l}(d h^{\nlin}), \quad \| d h^{\nlin}(x,y) \|^2 < 104\pt e^{-4\pi\delta}  E_{Z_l}(d h^{\nlin}). 
 \end{equation*}%

\end{enumerate}
\end{lem}

\begin{figure}[h!]
\SetLabels
\L(.255*.01) $\,\delta_{\lef}$\\
\L(.25*.78) $D_{\lef}$\\
\L(.48*.78) $D_{int}$\\
\L(.73*.78) $D_{\rig}$\\
\L(.715*.01) $\delta_{\rig}$\\
\endSetLabels
\AffixLabels{%
\centerline{%
\includegraphics[height=3.5cm,width=12cm]{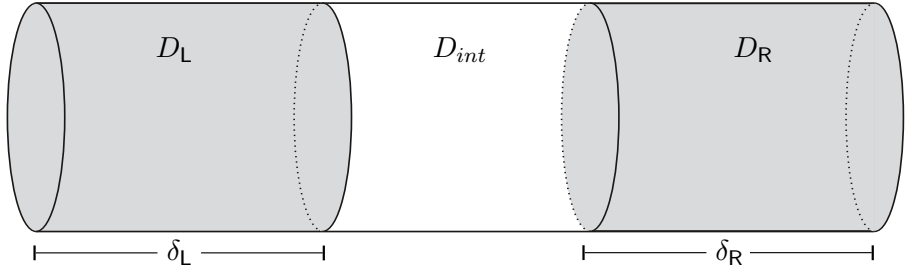}}}
\caption{The Euclidean cylinder $Z_l$ with ends $D_{\lef}$ and $D_{\rig}$.}
\label{fig:Euclidean_cyl}
\end{figure}

 \begin{proof}%
Statement i) is an immediate consequence of \beqref{eq:energy_y}. Furthermore, \beqref{eq:energy_y} applied to $h^{-}$ and integration over $[-l, l]$ yields

\begin{equation}\label{eq:Eh_sinh} 
  E_{Z_l}(d h^{-}) = \pi \sum_{n=1}^{\infty} n \left( a_n^2  + b_n^2  \right)  (e^{4\pi n l}-e^{-4\pi n l}).
\end{equation}
Similarly,
\begin{equation}
  E_{D_{int}}(d h^{-}) =   \pi \sum_{n=1}^{\infty} n \left( a_n^2  + b_n^2  \right)  (e^{4\pi n (l-\delta_{\lef})}-e^{-4\pi n (l-\delta_{\rig})}).
\label{eq:Eint_sinh}  
\end{equation}
From this and observing that
\begin{align*}
  e^{4\pi n(l-\delta_{\lef})}-e^{-4\pi n(l-\delta_{\rig})}\leq e^{4\pi n(l-\delta_{\lef})}-e^{-4\pi n l} \leq e^{-4\pi n \delta_{\lef}}(e^{4\pi n l}-e^{-4\pi n l})
\end{align*}

we get the second inequality in ii). Integrating $| h^{-}(x,y)|^2$  and $\| d h^{-}(x,y) \|^2$  over $\Sp_1$ we obtain in a similar fashion as in \beqref{eq:energy_y}:

\begin{equation*}
\int \limits_0^1 | h^{-}(x,y)|^2 \,d y =  \frac{1}{2} \sum_{n=1}^{\infty}  (a_n^2  + b_n^2) e^{-4\pi nx}, \quad
\int \limits_0^1 \| d h^{-}(x,y) \|^2 \,d y =  4\pi^2 \sum_{n=1}^{\infty}  n^2(a_n^2  + b_n^2) e^{-4\pi nx}
\end{equation*}
which implies the first. Item iii) is obtained in the same way. Combining ii), iii) and using that $(u+v)^2 \leq 2(u^2+v^2)$ we get iv). 
For item v) we first have

\begin{equation*}
   |h^{-}(x,y)| \leq \sum_{n=1}^{\infty} (|a_n| + |b_n| ) e^{-2\pi n x}   
\end{equation*}
and then estimate the sums for $a_n$ and $b_n$ individually, beginning with $a_n$. Let, with a glance at \beqref{eq:Eh_sinh},

 \begin{equation*}%
s_a =  \sum_{n=1}^{\infty} |a_n|  e^{-2\pi nx},
\quad
E_a= \pi \sum_{n=1}^{\infty} n  a_n^2  (e^{4\pi nl} -  e^{-4\pi nl}). 
 \end{equation*}%

With $\alpha_n = \vert a_n \vert e^{2\pi n l}(n(1-e^{-8\pi n l}))^{\frac{1}{2}}$,  the two terms may be written in the form

 \begin{equation*}%
s_a =  \sum_{n=1}^{\infty} \alpha_n (n(1-e^{-8\pi n l}))^{-\frac{1}{2}}  e^{-2\pi n(x+l)},
\quad
E_a= \pi \sum_{n=1}^{\infty} \alpha_n^2.
 \end{equation*}%
By the Cauchy-Schwarz inequality,

 \begin{equation*}%
s_a^2 \leq \sum_{n=1}^{\infty}\alpha_n^2 \cdot \sum_{n=1}^{\infty}\frac{1}{n(1-e^{-8\pi n l})}\, e^{-4\pi n(x+l)}.
 \end{equation*}%
The second sum is further bounded above by

 \begin{equation*}%
\frac{1}{1-e^{-8\pi l}}\sum_{n=1}^{\infty}\frac{1}{n}e^{-4\pi n(x+l)}=\frac{-1}{1-e^{-8\pi l}}\ln(1-e^{-4\pi (x+l)})<\frac{3}{2}e^{-4\pi (x+l)},
 \end{equation*}%
where the simplification ob the right uses that $l \geq 1$ and $4\pi(x+l) \geq 4\pi\delta_{\lef} \geq 2\pi$. Bringing everything together we get $s_a^2 < \frac{1}{2}E_a e^{-4\pi (x+l)}$. An analogous inequality holds for the  $b_n$. Using that $(u+v)^2 \leq 2(u^2 + v^2)$ and recalling \beqref{eq:Eh_sinh} we get the first inequality in v). The proof of the second inequality is following the same scheme beginning with 

\begin{equation*}
   \Big | \frac{\partial h^{-}(x,y)}{\partial x}\Big| \leq 2\pi \sum_{n=1}^{\infty} n(|a_n| + |b_n| ) e^{-2\pi n x},
\quad   
   \Big | \frac{\partial h^{-}(x,y)}{\partial y}\Big| \leq 2\pi \sum_{n=1}^{\infty} n(|a_n| + |b_n| ) e^{-2\pi n x}.
\end{equation*}
The sum $s_a$ is now replaced by $t_a = 2\pi \sum_{n=1}^{\infty}n \vert a_n \vert e^{-2\pi n x}$ , and instead of the above sum $\sum_{n=1}^{\infty}\frac{1}{n}e^{-4\pi n(x+l)}$ we now deal with $\sum_{n=1}^{\infty}n e^{-4\pi n(x+l)} = e^{4\pi (x+l)}(e^{4\pi (x+l)}-1)^{-2}$.  The inequality $s_a^2 < \frac{1}{2}E_a e^{-4\pi (x+l)}$ is taken over by $t_a^2 < 13 E_a e^{-4\pi (x+l)}$ with its analog for the $b_n$, and the result is

 \begin{equation*}%
\left \vert \frac{\partial h^{-}(x,y)}{\partial x} \right\vert^2 < 26\pt e^{-4\pi (x+l)}  E_{Z_l}(d h^{-}),
\quad
\left \vert \frac{\partial h^{-}(x,y)}{\partial y} \right\vert^2 < 26\pt e^{-4\pi (x+l)}  E_{Z_l}(d h^{+}).
 \end{equation*}%
Hence, the second inequality. The proof of vi) is the same. For the proof of vii) one uses v) and vi) and the fact that $E_{Z_l}(d h^{\nlin}) = E_{Z_l}(d h^{-}) + E_{Z_l}(d h^{+})$.
\end{proof}%

We also need the following orthogonality lemma, where $Z' = [x_0,x_1] \times \Sp_1$. The proof is similar to that of the first part of the preceding lemma and is omitted. The term ``sufficiently strongly'' means that the series and the term by term differentiated series are absolutely uniformly convergent.

\begin{lem}\label{thm:lemorthogonal}
(Orthogonality lemma) Let $H, \tilde{H} : Z' \to \R$ be sufficiently strongly convergent series

 \begin{gather*}
H(x,y) = \sum_{n=1}^{\infty}\varphi_n(x) \cos(2\pi n y) + \sum_{n=1}^{\infty}\psi_n(x) \sin(2\pi n y),\\
\tilde{H}(x,y) = \sum_{n=1}^{\infty}\tilde{\varphi}_n(x) \cos(2\pi n y) + \sum_{n=1}^{\infty}\tilde{\psi}_n(x) \sin(2\pi n y),
 \end{gather*}%
and $\omega$, $\tilde{\omega}$ on $Z'$ the forms

 \begin{gather*}
\omega = b(x) d x + c\, d y + d H,\\
\tilde{\omega} = \tilde{b}(x) d x + \tilde{c}\, d y + d \tilde{H},
 \end{gather*}%
where $b, \tilde{b} : [x_0,x_1] \to \R$ are continuous functions and $c, \tilde{c} \in \R$ are constants. Then

 \begin{equation*}
\int_{Z'} \omega \wedge \star \tilde{\omega} = \int\limits_{x_0}^{x_1}(c \tilde{c} + b(x) \tilde{b}(x)) d x + \int_{Z'} d H \wedge \star d \tilde{H}.
 \end{equation*}%
 \end{lem}%
\begin{lem}\label{lem:dampening}
(First dampening down lemma) Let $Z_l = [-l,l] \times \Sp_1$ with subsets $D_{\lef}$, $D_{int}$, $D_{\rig}$ be as in \bbref{Lemma}{thm:lem_flat}, $h : Z_l \to \R$ with nonlinear part $h^{\nlin}=h^{-}+h^{+}$ a harmonic function as in \beqref{eq:nonlinear_flat} and $\omega$ the harmonic form 

 \begin{equation*}
\omega = b_0 d x + c_0 d y + d h^{\nlin},
 \end{equation*}%
where $b_0, c_0 \in \R$ are constants. Set $\delta := \min\{\delta_{\lef}, \delta_{\rig}  \}$ and $w = 2l -(\delta_{\lef}+\delta_{\rig})$. In  $D_{int}$ we dampen down $\omega$ to

 \begin{equation*}
\omega_{\chi} = b_0 d x + c_0 d y + d (\chi \cdot h^{\nlin}),
 \end{equation*}
where $\chi : Z_l \to \R$ is the cut off function with the property that $\chi = 1$ on $D_{\lef}$, $\chi = 0$ on $D_{\rig}$ and $\chi$ is a linear function of $x$ on $D_{int}$. Then for the three functions $h^{\texttt{\#}}= h^{-}, h^{+}, h^{\nlin}$ we have 
\begin{enumerate}
\itemup{i)} %
$E_{D_{int}}(d(\chi \cdot h^{\texttt{\#}})) 
\leq  E_{D_{int}}(d h^{\texttt{\#}}) + \left(1+\frac{1}{2 \pi^2 w^2}\right)e^{-4\pi \delta_{\texttt{\#}}}E_{Z_l}(d h^{\texttt{\#}})$,
\end{enumerate}
where, respectively, $\delta_{\texttt{\#}} = \delta_{\lef}, \delta_{\rig}, \delta$. The dampened down form $\omega_{\chi}$ satisfies

\begin{enumerate}
\itemup{ii)} $E_{D_{int}}(\omega_{\chi}) = E_{D_{int}}(b_0 d x + c_0 d y) + E_{D_{int}}(d(\chi \cdot h^{\nlin}))$,
\itemup{iii)} $E_{Z_l}(\omega_{\chi}) \leq E_{Z_l}(\omega) + \left(1+\frac{1}{2 \pi^2 w^2}\right)e^{-4\pi \delta}E_{Z_l}(\omega)$.
\end{enumerate}
\end{lem}%

 \begin{proof}%
Statement ii) is an instance of \bbref{Lemma}{thm:lemorthogonal}. For the inequality in i) we first have, pointwise, %
\begin{equation}\label{eq:triangle}
\| d(\chi h^{\texttt{\#}})\| = \| \chi d h^{\texttt{\#}} + h^{\texttt{\#}} d \chi \| \leq  \| \chi d h^{\texttt{\#}}\| + \|h^{\texttt{\#}} d \chi \| \leq  \|d h^{\texttt{\#}}\| + \frac{1}{w} |h^{\texttt{\#}} |.
\end{equation}
Using that $(u+v)^2 \leq 2(u^2 + v^2)$ we further get

 \begin{equation*}%
 \| d(\chi h^{\texttt{\#}})\|^2  \leq 2( \|d h^{\texttt{\#}}\|^2 + \frac{1}{w^2} |h^{\texttt{\#}} |^2).
 \end{equation*}%

Integrating over $D_{int}$ we conclude the proof of i) using \bbref{Lemma}{thm:lem_flat}. (For $h^{\texttt{\#}} = h^{-}, h^{+}$ we can actually replace $(1+ \frac{1}{2 \pi^2 w^2})$ with $(1+ \frac{1}{4 \pi^2 w^2})$, but this will not improve our later results). Statement iii) is a consequence of the preceding ones.
\end{proof}%

The choice of $\frac{1}{100}$ for $w$ in the next lemma has been made so as to keep the increase of energy in the dampening down process small.

\begin{lem}\label{lem:dampening2}
(Second dampening down lemma) Let $Z_l = [-l,l] \times \Sp_1$ and $h : Z_l \to \R$ be as in 
the first lemma, but now assume that $l \geq 1$, $\delta_{\lef} = \frac{1}{2}$ and $w = 2l - (\delta_{\lef} + \delta_{\rig}) =\frac{1}{100}$. In  $D_{int}$ we partially dampen down $\omega$ to

 \begin{equation*}
\omega_{\chi}^{+} = b_0 d x + c_0 d y + d h^{-}\! + d (\chi \cdot h^{+}),
 \end{equation*}
where $\chi : Z_l \to \R$ is the same cut off function as in \bbref{Lemma}{lem:dampening}. Then
\begin{enumerate}
\itemup{i)} $E_{Z_l}(\omega_{\chi}^{+}) = E_{Z_l}(b_0 d x + c_0 d y) + E_{Z_l}(d h^{-} + d(\chi \cdot h^{+}))$,
\itemup{ii)} $E_{Z_l}(d h^{-} + d(\chi \cdot h^{+})) \leq E_{Z_l}(d h^{-} + d h^{+}) +  e^{-8\pi l +5}E_{Z_l}(d h^{-})$.
\end{enumerate}
\end{lem}%
 \begin{proof}%
i) is again an instance of \bbref{Lemma}{thm:lemorthogonal}. For the proof of ii) we abbreviate

 \begin{equation*}%
\hat{h} = h^{-}\! + \chi \cdot h^{+},
\quad
E^{-} = E_{Z_l}(d h^{-}),
\quad
E^{+} = E_{Z_l}(d h^{+}).
 \end{equation*}%
By \bbref{Lemma}{thm:lem_flat} v) and vi) (applicable since $l \geq 1$, $\delta_{\lef}=\frac{1}{2}$) we have the following bounds for $(x,y) \in D_{int}$:

\begin{alignat}{2}
&| h^{-}(x,y)| < e^{-2\pi \delta_{\lef}}  \sqrt{E^{-}}, \quad& &\| d h^{-}(x,y) \| < 8e^{-2\pi\delta_{\lef}}  \sqrt{E^{-}},%
\label{eq:honDint1}\\
&| h^{+}(x,y)| < e^{-2\pi \delta_{\rig}}  \sqrt{E^{+}}, \quad& &\| d h^{+}(x,y) \| < 8e^{-2\pi \delta_{\rig}}  \sqrt{E^{+}}.%
\label{eq:honDint2}
\end{alignat}
Furthermore, since $\| d \chi(x,y) \| = \frac{1}{w}$ for $(x,y) \in D_{int}$ and $\vert \chi(x)\vert \leq 1$, the triangle inequality yields

 \begin{align*}%
\Vert d \hat{h}(x,y)\Vert &\leq \Vert d h^{-}(x,y)\Vert + \frac{1}{w} \vert h^{+}(x,y) \vert + \Vert d h^{+}(x,y) \Vert 
\\
&\leq \Vert d h^{-}(x,y)\Vert + (\frac{1}{w}+8) e^{-2\pi \delta_{\rig}} \sqrt{E^{+}}.
 \end{align*}%
Taking the squares and integrating over $D_{int}$ we get, using \beqref{eq:honDint1},

 \begin{equation*}%
E_{D_{int}}(d \hat{h}) \leq E_{D_{int}}(d h^{-})+ \alpha \sqrt{E^{+}} + \beta E^{+}
 \end{equation*}%
with

 \begin{equation}\label{eq:honDint3}%
 \alpha = 16(\frac{1}{w}+8)e^{-2\pi (\delta_{\lef}+\delta_{\rig})}\sqrt{E^{-}} \area(D_{int}),
 \quad
  \beta = (\frac{1}{w}+8)^2 e^{-4\pi \delta_{\rig}}\area(D_{int}).
 \end{equation}%
Since $E_{D_{int}}(d h^{-}) = E_{D_{int}}(d h^{\nlin}) - E_{D_{int}}(d h^{+})$ and $E_{D_{\rig}}(d h^{\nlin}) = E_{D_{\rig}}(d h^{-})+E_{D_{\rig}}(d h^{+})$ we get

 \begin{equation*}%
E_{Z_l}(d \hat{h}) \leq E_{Z_l}(d h^{\nlin}) + \alpha \sqrt{E^{+}} + \beta E^{+} - E_{Z_l \smallsetminus D_{\lef}}(d h^{+}).
 \end{equation*}%
Now $E_{Z_l \smallsetminus D_{\lef}}(d h^{+})$ is close to $E_{Z_l }(d h^{+})$ : By \beqref{eq:energy_y} 

 \begin{equation*}%
E_{Z_l\smallsetminus D_{\lef} }(d h^{+}) = 4\pi^2 \int\limits_{-l+\delta_{\lef}}^l \sum_{n=1}^{\infty}n^2(c_n^2 + d_n^2)  e^{4\pi n x} d x. 
 \end{equation*}%
Introducing the lower bound $\kappa$ in the following elementary inequality:

 \begin{equation*}%
\frac{\int_{-l+\delta_{\lef}}^l e^{4\pi n x} d x}{\int_{-l}^l e^{4\pi n x} d x}
\geq
\frac{\int_{-l+\delta_{\lef}}^l e^{4\pi x} d x}{\int_{-l}^l e^{4\pi x} d x} \overset{\textup{def}}{=} \kappa 
 \end{equation*}%
we get $E_{Z_l \smallsetminus D_{\lef}}(d h^{+}) \geq \kappa E_{Z_l }(d h^{+})$ and therefore
$E_{Z_l}(d \hat{h}) \leq E_{Z_l}(d h^{\nlin}) + \alpha \sqrt{E^{+}} + (\beta-\kappa) E^{+}$.
We are now using that $l \geq 1$, $\delta_{\lef} = \frac{1}{2}$, and $w =\frac{1}{100}$. For these values, it is easily checked that $\beta$ is close to 0 and $\kappa$ is close to 1, and we have the rough but sufficient estimate $\beta-\kappa < -\frac{3}{4}$. We get
 \begin{equation*}%
E_{Z_l}(d \hat{h}) \leq E_{Z_l}(d h^{\nlin}) + \alpha \sqrt{E^{+}} -\frac{3}{4} E^{+}. \end{equation*}%

Finally, for $t \in [0,\infty)$ the function $\alpha \sqrt{t} - \frac{3}{4}t$ has the upper bound $\frac{1}{3}\alpha^2$, and plugging in the values of $w$, $\delta_{\lef}$, $\delta_{\lef}$ into $\alpha$ (see \beqref{eq:honDint3}) we conclude the proof of ii) by elementary simplification.
 \end{proof}%

\section{Separating case }\label{sec:SepCas}
The setting for this section is as follows. $S$ is a compact hyperbolic Riemann surface of genus $g \geq 2$ that contains a simple closed geodesic $\gamma$ of length

 \begin{equation}\label{eq:lgam12}%
\ell(\gamma) \leq \frac{1}{2}
 \end{equation}%
separating $S$ into two surfaces $S_1$ and $S_2$  of signatures $(g_1,1)$ and $(g_2,1)$, respectively (see \bbref{Fig.}{fig:gamma_small}). We let $\left( {\alpha _i } \right)_{i = 1,\dots,2g}$ be a canonical homology basis of $S$, such that $(\alpha_1,..,\alpha_{2g_1}) \subset S_1$ and $(\alpha_{2g_1 +1},..,\alpha_{2(g_1+g_2)}) \subset S_2$  and let $P_S$ be the Gram period matrix with respect to this basis. In the first part of this section we look at the energy distribution of harmonic forms that have zero periods on $S_2$ respectively, $S_1$. In the second part we make use of this to prove \bbref{Theorem}{thm:small_scg}.

\begin{figure}[h!]
\SetLabels
\L(.514*.61) $\gamma$\\
\L(.26*.855) $S_1$\\
\L(.71*.87) $S_2$\\
\L(.40*.79) $Y_1$\\
\L(.625*.815) $Y_2$\\
\L(.475*.34) $C=C(\gamma)$\\
\L(.12*.475) $\alpha_1$\\
\L(.843*.245) $\alpha_{2g}$\\
\endSetLabels
\AffixLabels{%
\centerline{%
\includegraphics[height=4.2cm]{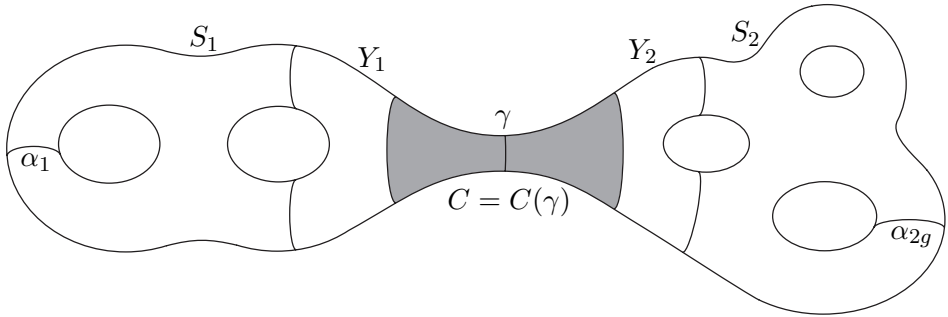}}}
\caption{A Riemann surface $S$ with a short separating geodesic and adjacent Y-pieces.}
\label{fig:gamma_small}
\end{figure}

\subsection{Concentration of the energies of harmonic forms}\label{sec:ConcEn}
In the first part of the next theorem we prove point a) of \bbref{Theorem}{thm:energy_distribution}. We recall $\mu(\gamma)$ from \beqref{eq:mugamma_intro}, for convenience,

 \begin{equation}\label{eq:mugamma}%
 \mu (\gamma) = \exp\left(-2\pi^2\left(\frac{1}{\ell(\gamma)}-\frac{1}{2}\right)\right).
 \end{equation}%

\begin{thm}\label{thm:small_blocks}
Let $\sigma \in H^1(S,\R)$ be a real harmonic 1-form, such that $\int_{[\alpha]} \sigma=0$ for all $[\alpha] \in H_1(S_2,\Z)$.  Then 

 \begin{equation*}%
E_{S_2}(\sigma) \leq \mu(\gamma)E_S(\sigma). \tag{i}
 \end{equation*}%
Let, furthermore, $\tau \in H^1(S,\R)$ be a real harmonic 1-form  such that $\int_{[\beta]} \tau=0$ for all $[\beta] \in H_1(S_1,\Z)$. Then  

 \begin{equation*}%
 \big\vert  \int_S \sigma \wedge \star \tau \, \big\vert \leq  \mu(\gamma) \sqrt{ E_S(\sigma) E_S(\tau)}.  \tag{ii}
 \end{equation*}%
\end{thm}

 \begin{proof}%
We look at the decay of $\sigma$ in the standard collar  $C=C(\gamma)$ of $\gamma$ (shaded  area in \bbref{Fig.}{fig:gamma_small}, see \beqref{eq:stacol}, \beqref{eq:clfctn} for the definition). To make use of the preceding section we shall map $C(\gamma)$ conformally to a flat cylinder using the mapping $\psi_{\gamma}$ as in \beqref{eq:psigamma} which we repeat here, for convenience.
 \begin{equation*}%
   \psi_{\gamma}: C=C(\gamma) \rightarrow Z_M, \text{ \ where \ } Z_M = ({-}M,M) \times \Sp_1 \end{equation*}%
and by \beqref{eq:Lleta}

 \begin{equation}\label{eq:Llgamma}
M=M(\ell(\gamma)) = \frac{1}{\ell(\gamma)} \left\{\frac{\pi}{2} - \arcsin\left(\tanh\left(\frac{\ell(\gamma)}{2}\right)\right)\right\} \geq \frac{\pi}{2\ell(\gamma)} -\frac{1}{2}  > 2.
 \end{equation}%
We denote the inverse of $\psi_{\gamma}$ by $\phi_C$:

 \begin{equation*}%
\phi_C = (\psi_{\gamma})^{-1} : Z_M \rightarrow C.
 \end{equation*}%
On $C$ the form $\sigma$ has only zero periods and can therefore be written as $\sigma = d F$, with a harmonic function $F:C \to \R$.
Let $h = F \circ \phi_C$ be the corresponding harmonic function on $Z_M$. We decompose it into a sum

 \begin{equation}\label{eq:htildeh}%
h(x,y) =a_0 + b_0 x + h^{\nlin}(x,y), 
 \end{equation}%
as in the preceding section, where $h^{\nlin}$ is the nonlinear part of $h$. The constant $a_0$ may be arbitrarily chosen, we take $a_0 = 0$. With $w = \frac{1}{8}$  (a constant that has proved to be practical) we set $\delta = M-w$ and dampen down $h$ to $H$ setting

\begin{equation}\label{eq:H_separating}  
H(x,y) =
\begin{cases}
    b_0 x + \chi(x)  h^{\nlin}(x,y), & \text{for $x \in ({-}M,0]$}\\
    0 & \text{for $x \in [0,M),$}
    \end{cases}  
\end{equation}
where $\chi$ is the cut off function as in \bbref{Lemma}{lem:dampening} that goes linearly from 1 to 0 in the interval $[-w,0]$. Setting  $\widetilde{D} = [-w,0] \times \Sp_1$ we obtain, by that lemma, clause i), with $l=M$, $\delta_{\lef} = M-w$, $\delta_{\rig} = M$, $\delta_{\texttt{\#}} = \min\{\delta_{\lef}, \delta_{\rig}  \} =: \delta$ and $h^{\texttt{\#}} = h^{\nlin}$,
 \begin{equation*}%
E_{\widetilde{D}}(d H) \leq E_{\widetilde{D}}(d h) + (1+\frac{1}{2\pi^2 w^2})e^{-4\pi \delta}E_{Z_M}(d h).
 \end{equation*}%
In addition, we have used that by the \bbref{Orthogonality Lemma}{thm:lemorthogonal} we have $E_{\widetilde{D}}(dh^{\nlin}) \leq E_{\widetilde{D}}(dh)$ and $E_{Z_M}(dh^{\nlin}) \leq E_{Z_M}(dh)$. An elementary calculation using that $M(\ell(\gamma)) \geq \frac{\pi}{2\ell(\gamma)}-\frac{1}{2}$ (see \beqref{eq:Llgamma}) yields 
 \begin{equation*}
E_{\widetilde{D}}(d H) \leq E_{\widetilde{D}}(d h) +e^{-2\pi^2\left(\frac{1}{\ell(\gamma)}-\frac{1}{2}\right)} E_{Z_M}(d h) = E_{\widetilde{D}}(d h) + \mu(\gamma)E_{Z_M}(d h).
 \end{equation*}%
Now we set
\begin{equation}\label{eq:damped}
s =
\begin{cases}
\sigma & \text{on $S_1 \smallsetminus C$}\\
d(H \circ \psi_{\gamma}) & \text{on $C$}\\
0 & \text{on $S_2 \smallsetminus C$}.
\end{cases}
\end{equation}
Since all periods of $\sigma$ over cycles in $S_2$ vanish the two forms are in the same cohomology class and so we have $E_S(\sigma) \leq E_S(s)$ by the minimizing property of harmonic forms. On the other hand, using that $\phi_C$ is conformal,

 \begin{equation}\label{eq:boundES}%
E_S(s) = E_{S_1 \smallsetminus \phi_C(\widetilde{D})}(\sigma) +E_{\widetilde{D}}(d H) \leq E_{S_1}(\sigma) + \mu(\gamma) E_{Z_M}(d h).
\end{equation}%
Furthermore, $E_{Z_M}(d h) \leq E_S(\sigma)$. Hence, altogether

 \begin{equation}\label{eq:boundES2}%
E_{S_2}(\sigma) \leq \mu(\gamma)E_S(\sigma).
 \end{equation}%
This proves part (i) of \bbref{Theorem}{thm:small_blocks}. For part (ii) we first remark that $s = \sigma + d G_{\sigma}$ for some piecewise smooth function $G_{\sigma} : S \to \R$, given that $\sigma$ and $s$ are in the same cohomology class. By orthogonality and by Equation \beqref{eq:boundES} this function satisfies

 \begin{equation*}%
E_S(\sigma) + E_S(d G_{\sigma}) = E_S(s) \leq E_S(\sigma) + \mu(\gamma)E_S(\sigma).
 \end{equation*}%
Hence, $E_S(d G_{\sigma}) \leq \mu(\gamma)E_S(\sigma)$. Now consider $\tau$. Repeating the procedure with the roles of $S_1$, $S_2$ reversed we get a dampened down form $t$ that vanishes on $S_1$ and can be written as $t = \tau + d G_{\tau}$ with some function $G_{\tau} : S \to \R$ satisfying $E_S(d G_{\tau}) \leq \mu(\gamma)E_S(\tau)$. Since $t$ vanishes on $S_1$ and $s$ vanishes on $S_2$ we have, again by orthogonality,

 \begin{equation*}%
0 = \int_S s \wedge \star t = \int_S \sigma \wedge \star \tau + \int_S d G_{\sigma} \wedge \star d G_{\tau}.
 \end{equation*}%
By Cauchy-Schwarz, $\vert \int_S d G_{\sigma} \wedge \star d G_{\tau}  \vert^2 \leq E_S(d G_{\sigma})E_S(d G_{\tau}) \leq \mu(\gamma)^2 E_S(\sigma)E_S(\tau)$, and (ii) follows.
 \end{proof}%
 
It is interesting to observe that for $\sigma$ as in \bbref{Theorem}{thm:small_blocks} almost all of the energy decay in the direction of $S_2$ takes place in the collar $C(\gamma)$: with $\mu(\gamma)$ as in \beqref{eq:mugamma} we have the following which is part b) of \bbref{Theorem}{thm:energy_distribution}
 \begin{thm}\label{thm:EngInC}%
Let $S$ with separating geodesic $\gamma$ and $\sigma \in H^1(S,\R)$ with all periods over cycles in $S_2$ vanishing be as in \bbref{Theorem}{thm:small_blocks}. Then

 \begin{equation*}%
E_{S_2 \smallsetminus C(\gamma)}(\sigma) \leq \mu(\gamma)^2 E_{C(\gamma)}(\sigma).
 \end{equation*}%
 \end{thm}%
In particular, the upper bound for the energy of $\sigma$ on $S_2 \smallsetminus C(\gamma)$ is much smaller than that for $\sigma$ on $S_2$
 \begin{proof}%
We take over the setting of the proof of \bbref{Theorem}{thm:small_blocks}. However, this time the constant $a_0$ for the function $h$ in \beqref{eq:htildeh} is chosen such that $a_0 - b_0 M  = 0$, i.e.\ such that the linear part $h^{\lin}(x) = a_0 + b_0 x$ vanishes at the left end of the cylinder $Z_M$. 

In a first step we take  $\delta'_{\lef}=\frac{1}{2}$, $w' = \frac{1}{100}$ and partially dampen down $d h$ as in \bbref{Lemma}{lem:dampening2} setting

 \begin{equation*}
H' =  h^{-}\! + \chi' \cdot h^{+},
 \end{equation*}
where $h^{-} + h^{+} = h^{\nlin}$ is the decomposition of the nonlinear part of $h$ as in 
\beqref{eq:hseries}, \beqref{eq:nonlinear_flat}, \bbref{Lemma}{lem:dampening2}, and $\chi'$ is the cut off function that goes linearly from 1 at $x=-M+\delta'_{\lef}$ to 0 at $x=-M+\delta'_{\lef} + w'$. By \bbref{Lemma}{lem:dampening2} (applicable since by \beqref{eq:Llgamma} $M > 1$),

 \begin{equation*}%
E_{Z_M}(d H') \leq E_{Z_M}(d h^{-}\!+d h^{+}) + e^{-8\pi M+5}E_{Z_M}(d h^{-}).
 \end{equation*}%
In a second step we use \bbref{Lemma}{lem:dampening} again, taking $w'' = \frac{1}{8}$, $\delta''_{\lef} = 2M -w''$, and further dampen down $d H'$ setting 

 \begin{equation*}%
H'' =  \chi'' \cdot h^{-}\! + \chi' \cdot h^{+},
 \end{equation*}%
where $\chi''$ is the cut off function that goes linearly from 1 at $x=M-w''$ to 0 at $x=M$. By \bbref{Lemma}{lem:dampening}, applied to $h^{\texttt{\#}} = h^{-}$,

 \begin{equation*}%
E_{Z_M}(d H'') \leq E_{Z_M}(d H') + (1+\frac{1}{2 \pi^2 w''^2})e^{-4\pi (2M-w'')}E_{Z_M}(d h^{-})
\leq
E_{Z_M}(d H') + e^{-8\pi M + 4}E_{Z_M}(d h^{-}).
 \end{equation*}%
We now go back to $S$ letting $s''$ be the test form that coincides with $\sigma$ on 
$S_1 \smallsetminus C(\gamma)$, is the pull-back of $d H''$ on $C(\gamma)$ via $\psi_{\gamma}$ and vanishes on $S_2 \smallsetminus C(\gamma)$. Then $E_S(s'') = E_{S_1 \smallsetminus C(\gamma)}(\sigma) + E_{Z_M}(d H'') \leq E_{S_1 \smallsetminus C(\gamma)}(\sigma)  + E_{Z_M}(d h^{-}\!+d h^{+}) +e^{-8\pi M+6}E_{Z_M}(d h^{-})$ (simplifying the numerical constants). Now 
 \begin{equation*}%
E_S(s'') \geq E(\sigma) = E_{S_1 \smallsetminus C(\gamma)}(\sigma) + E_{Z_M}(d h^{\lin})  + E_{Z_M}(d h^{-}+d h^{+})+ E_{S_2 \smallsetminus C(\gamma)}(\sigma). 
 \end{equation*}%
This yields, altogether,
 \begin{equation}\label{eq:linsmall}%
E_{Z_M}(d h^{\lin}) + E_{S_2 \smallsetminus C(\gamma)}(\sigma) \leq e^{-8\pi M+6}E_{Z_M}(d h^{-}),
 \end{equation}%
where $M$ is from \beqref{eq:Llgamma}. The theorem now follows by elementary simplification.  
\end{proof}%

We remark that, by \beqref{eq:linsmall}, the linear part of $\sigma$ contributes very little to the total energy of $\sigma$ in $C(\gamma)$. 

For later use we also note that, by the inequalities preceding \beqref{eq:linsmall} (and using that by the \bbref{Orthogonality Lemma}{thm:lemorthogonal} $E_{Z_M}(d h^{-}+d h^{+}) \leq E_{Z_M}(d h)$) the test form $s''$ satisfies, in particular, %
 \begin{equation}\label{eq:Engspp}%
E_S(s'') - E_S(\sigma) \leq e^{-8\pi M+6}E_S(\sigma).
 \end{equation}%
 \begin{figure}[b!]
 \vspace{0pt}
 \begin{center}
 \leavevmode
 \SetLabels
 \L(.07*.82) $C(\gamma)$\\
\L(.04*.40) $S_1$\\
\L(.28*.63) $Y_1$\\
\L(.13*.075) $\alpha_1$\\
\L(.225*.23) $\alpha_2$\\
\L(.21*.91) $\gamma$\\
\L(.365*.40) $S_1^{\sF}$\\
\L(.61*.63) $Y_1^{0}$\\
\L(.525*.99) $q$\\
\L(.41*.82) $V(q)$\\
\L(.458*.07) $\alpha_1^{\sF}$\\
\L(.555*.23) $\alpha_2^{\sF}$\\
\L(.685*.40) $\widebar{1.5}{S_1^{\sF}}$\\
\L(.782*.07) $\alpha_1^{\sF}$\\
\L(.879*.23) $\alpha_2^{\sF}$\\
\L(.845*.67) $q$\\
 \endSetLabels
 \AffixLabels{
 \includegraphics{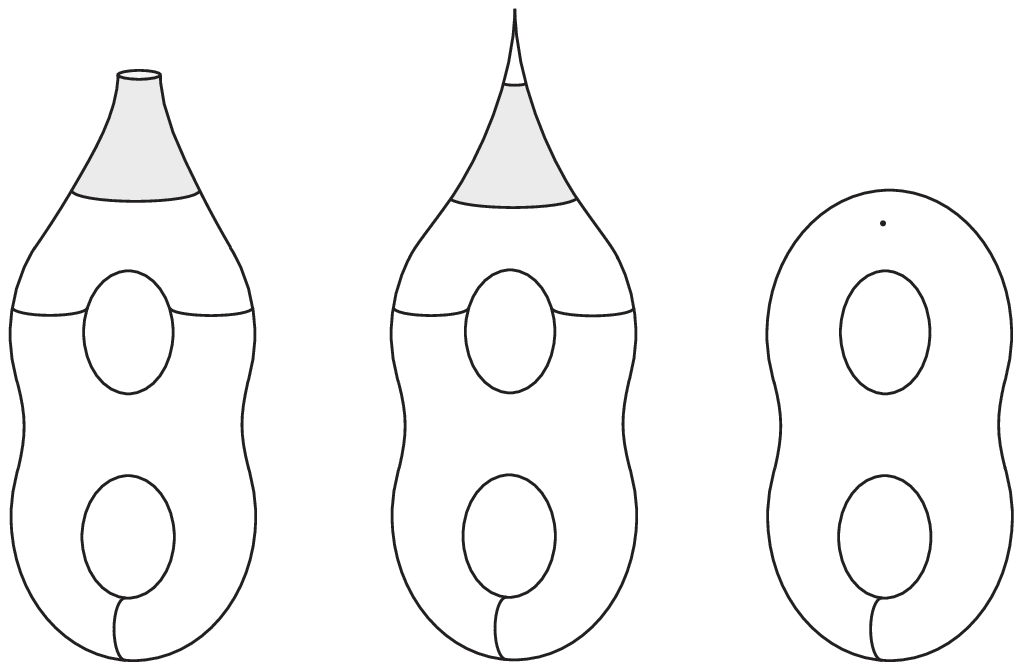} }
 \end{center}
 \vspace{-0pt}
 \caption{\label{fig:gamma_deg} The bordered Riemann surface $S_1$, its Fenchel-Nielsen limit and the one point compactification viewed as a conformal surface.}
 \end{figure}

\subsection{Convergence of the Jacobians for separating geodesics }
\label{sec:harmonic_energy}
We now compare the Jacobian of $S$ with the Jacobian of the limit surface on $\partial \mathcal{M}_g$ obtained by either the Fenchel-Nielsen or the grafting construction. \bbref{Fig.}{fig:gamma_deg} illustrates the first case for the part $S_1$ with Fenchel-Nielsen limit $S_1^{\sF}$. The pair of pants, or Y-piece, $Y_1\subset S_1$ with boundary geodesic $\gamma$ and half collar $C(\gamma)$ is quasi conformally embedded into the degenerated Y-piece $Y_1^0$ with ideal boundary point $q$ and cusp neighborhood $V(q)$, similarly to the case illustrated in \bbref{Fig.}{fig:quasiY}. The remaining parts of $S_1$ and $S_1^{\sF}$ are isometric. The punctured surface $S_1^{\sF}$ and its compactification $\widebar{1.5}{S_1^{\sF}}$, understood as a Riemann surface in the conformal sense, have the same $L^2$ harmonic forms. We denote by $P_{S_1^{\sF}}$ the Gram period matrix of $\widebar{1.5}{S_1^{\sF}}$ with respect to the homology basis $\alpha_1, \dots, \alpha_{2g_1}$; the matrices $P_{S_2^{\sF}}$, $P_{S_1^{\sG}}$, $P_{S_2^{\sG}}$ are defined in the same way. We now prove \bbref{Theorem}{thm:small_scg_intro} which we reproduce for convenience. The entries of the blocks $R_k^{\FG}$ are named $r_{ij}$, those of $\Omega$ and its transpose $\Omega^{\rm T}$
are named $\omega_{ij}$, and the numbering, disregarding the blocks, is $i,j=1,\dots, 2g$.

\begin{thm}\label{thm:small_scg}
Assume that the separating geodesic $\gamma$ has length $\ell(\gamma) \leq \frac{1}{2}$. Then

 \begin{equation*}%
P_S=
\left( {\begin{array}{*{2}c}
   P_{S_1^{\FG}} & 0  \\
   0 & P_{S_2^{\FG}}  \\
\end{array}} \right) + \left( {\begin{array}{*{20}c}
   R_1^{\FG} & \Omega  \rule{0pt}{14pt}\\
   \Omega^{\rm T} & R_2^{\FG}  \rule{0pt}{14pt}
\end{array}} \right),  \end{equation*}%

where $\{\}^{\FG}$ stands for either $\{\}^{\sF}$ or $\{\}^{\sG}$ and the entries of the remainder matrices have the following bounds,
\begin{enumerate}
\itemup{i)} $\vert \omega_{ij} \vert \leq e^{-2\pi^2(\frac{1}{\ell(\gamma)}-\frac{1}{2})} \sqrt{p_{ii}p_{jj}\mathstrut}$,
\itemup{ii)} $\vert r_{ij}^{\sG}\vert  \leq  e^{-4\pi^2(\frac{1}{\ell(\gamma)}-\frac{1}{2})} \sqrt{p_{ii}p_{jj}\mathstrut}$,
\itemup{iii)} $\vert r_{ij}^{\sF}\vert  \leq  6 \ell(\gamma)^2 \sqrt{p_{ii}p_{jj}\mathstrut}$.
\end{enumerate}
\end{thm}

 \begin{proof}%
Statement i) is an instance of \bbref{Theorem}{thm:small_blocks} (ii). For the proof of ii) we take the part concerning $P_{S_1^{\sG}}$. For simplicity, we shall identify the collar $C(\gamma)$ with the conformally equivalent $Z_M = ({-}M,M) \times \Sp_1$ ($M$ is from \beqref{eq:Llgamma}). 

We use that the surface $S' = \text{closure of } S_1 \cup C(\gamma)$, seen as a conformal surface, may be understood as a subset of the grafted surface $S_1^{\sG}$ in a natural way: $S'$ is $S_1$ with a copy of the flat cylinder $[0,M]\times \Sp_1$ attached along the boundary, and $S_1^{\sG}$ is $S_1$ with a copy of $[0,\infty)\times \Sp_1$ attached. For simplicity we write $S_1^{\sG}=\tilde{S}$.

Consider now a harmonic form $\sigma$ on $S$ with vanishing periods over the cycles of $S_2$ and let $\tilde{\sigma}$ on $\tilde{S}$ be the harmonic form that has the same cycles over $\alpha_1, \dots \alpha_{2g_1}$ as $\sigma$. For later reference we write $\tilde{\sigma} = \Psi(\sigma)$.

We shall dampen down $\sigma$ and $\tilde{\sigma}$ to test forms $s$ and $\tilde{s}$ with support on $S'$ and compare the energies. To obtain $s$ we apply to $\sigma$ the two step dampening down procedure used in   the proof of \bbref{Theorem}{thm:EngInC}. Thus, $s$ is $s''$ as in \beqref{eq:Engspp}. By \beqref{eq:Engspp} it satisfies 

 \begin{equation}\label{eq:PrfTsep1}%
E_S(s) \leq (1+e^{-8\pi M+6})E_S(\sigma).
 \end{equation}%
To obtain $\tilde{s}$ we remark that $\tilde{\sigma}$ on $Z_M$ has a representation $ \tilde{\sigma}= d h$, where we have $h(x,y) = \sum_{n=1}^{\infty} (a_n\cos(ny) + b_n\sin(ny)) e^{-nx}$, i.e. $d h$ has vanishing linear and $d  h^{+}$ parts (owing to the fact that on the infinite cylinder $\tilde{\sigma}$ has finite energy). We now let $\tilde{s}$ be the test form on $\tilde{S}$ that coincides with $\tilde{\sigma}$ on $S_1\smallsetminus C(\gamma)$, is represented on $Z_M$ by  $d \chi \cdot h$ and vanishes on the rest. For $\chi$ we take the cut off function that goes linearly from 1 at $x=M-w$ to 0 at $x=M$, with $w = \frac{1}{8}$. By \bbref{Lemma}{lem:dampening}(i) applied to $h^{\texttt{\#}} = h = h^{-}$ with $\delta^{\texttt{\#}} = \delta_{\lef}=2M-w$ %

 \begin{equation}\label{eq:PrfTsep2}%
E_{\tilde{S}}(\tilde{s})  \leq (1+e^{-8\pi M+4})E_{\tilde{S}}(\tilde{\sigma}).
 \end{equation}%
Now $s$, having its support on $S'$, may be extended by zero to $\tilde{S}$ and thus also be seen as a test form for $\tilde{\sigma}$ on $\tilde{S}$. Similarly, $\tilde{s}$ may be seen as a test form for $\sigma$ on $S$. Hence, by \beqref{eq:PrfTsep1} and \beqref{eq:PrfTsep2}
 \begin{gather*}
E_{\tilde{S}}(\tilde{\sigma}) \leq E_{\tilde{S}}(s)=E_S(s) \leq (1+e^{-8\pi M+6})E_S(\sigma),\\
E_{S}(\sigma) \leq E_{S}(\tilde{s})=E_{\tilde{S}}(\tilde{s}) \leq (1+e^{-8\pi M+4})E_{\tilde{S}}(\tilde{\sigma}).
 \end{gather*}%
Now, the above mapping $\Psi$ that sends any harmonic $\sigma$ on $S$ with vanishing periods over the cycles of $S_2$ to $\tilde{\sigma} = \Psi(\sigma)$ is a linear isomorphism of vector spaces. Therefore, by \bbref{Lemma}{L:LinAlgIneq} below, for any $i,j = 1, \dots, 2g_1$,
 \begin{equation*}%
\vert \scp{\sigma_i}{\sigma_j} - \scp{\tilde{\sigma}_i}{\tilde{\sigma}_j} \vert \leq 
e^{-8\pi M + 6} \sqrt{p_{ii}p_{jj}}
 \end{equation*}%
and ii) follows from \beqref{eq:Llgamma} by elementary simplification.

For statement iii) we use that by \bbref{Theorem}{thm:psRGRF} there exists a quasiconformal homeomorphism $\phi : S_1^{\sF} \to S_1^{\sG}$ of dilatation $q_{\phi} \leq (1+2\ell(\gamma)^2)^2$ that extends conformally to the compactified surfaces (with the notation of \bbref{Theorem}{thm:psRGRF} our $\phi$ here is $\psi^{-1}$).  Using that $\ell(\gamma) \leq \frac{1}{2}$ we simplify the bound on the dilatation to $q_{\phi} \leq 1+5 \ell(\gamma)^2$. By \bbref{Lemma}{L:qcIneq} below we have  the following inequality for the harmonic forms $\sigma'_i$, $\sigma'_j$ in the cohomology classes of $\phi^{*}\sigma_i$, $\phi^{*}\sigma_j$,

 \begin{equation*}%
\vert \scp{\sigma'_i}{\sigma'_j} - \scp{\tilde{\sigma}_i}{\tilde{\sigma}_j} \vert \leq 5 \ell(\gamma)^2 \sqrt{E_{S_1^{\sG}}(\tilde{\sigma}_i) E_{S_1^{\sG}}(\tilde{\sigma}_j)}.
 \end{equation*}%
Now $(\sigma'_1, \dots, \sigma'_{2g_1})$ is our chosen dual basis of harmonic forms on $S_1^{\sF}$ whose scalar products $\scp{\sigma'_i}{\sigma'_j}$ are the entries of the Gram period matrix $P_{S_1^{\sF}}$, and $(\tilde{\sigma}_1, \dots, \tilde{\sigma}_{2g_1})$ is the dual basis on $S_1^{\sG}$ with $\scp{\tilde{\sigma}_i}{\tilde{\sigma}_j}$ being the entries of $P_{S_1^{\sG}}$.

Together with ii) this yields iii) by elementary simplification. 
\end{proof}%

 \begin{lem}\label{L:LinAlgIneq}
Let $U$, $V$ be Euclidean vector spaces whose norms and scalar products we denote by $\Vert \phantom{.} \Vert$ and $\scp{\phantom{.}}{\phantom{.}}$, let $\phi :U \to V$ be a linear mapping and let $\varepsilon >0$. If $(1-\varepsilon)\Vert u \Vert^2 \leq \Vert \phi(u) \Vert^2 \leq (1+\varepsilon) \Vert u \Vert^2$, for all $u \in U$, then %
 \begin{equation*}%
\big\vert \scp{\phi(u)}{\phi(v)} - \scp{u}{v} \big\vert \leq \varepsilon \Vert u \Vert \pt \Vert v \Vert, \quad \forall u,v, \in U.
 \end{equation*}%
 \end{lem}
 \begin{proof}%
We sketch the proof of this know fact, for convenience. The hypothesis and the conclusion are both scaling invariant, we may therefore assume that $\Vert u \Vert = \Vert v \Vert =1$. Writing $\phi(u) = u'$ and $\phi(v) = v'$ we get, using the polarization identity,
 \begin{align*}%
\scp{u'}{v'} &= \frac{1}{4}\big(\Vert u'+v'\Vert^2-\Vert{u'-v'}\Vert^2\big)\\
&\leq
\frac{1}{4}\big(\Vert u+v\Vert^2-\Vert{u-v}\Vert^2\big) + \frac{1}{4}\varepsilon^2\big(\Vert u+v\Vert^2+\Vert{u-v}\Vert^2\big)
=
\scp{u}{v} + \varepsilon.
 \end{align*}%
In the same way one shows that $\scp{u'}{v'} \geq \scp{u}{v} - \varepsilon$.
 \end{proof}%

As a corollary one has

 \begin{lem}\label{L:qcIneq}%
Let $\phi : R' \to R$ be a quasi conformal homeomorphism of compact Riemann surfaces, $\omega, \eta \in H^1(R,\R)$ and $\omega', \eta' \in H^1(R',\R)$ the harmonic forms in the cohomology classes of the induced forms $\phi^{*}\omega$, $\phi^{*}\eta$ on $R'$. If $\phi$ has dilatation $1 + \varepsilon$, then

 \begin{equation}\label{eq:qcIneq1}
\vert \scp{\omega'}{\eta'} - \scp{\omega}{\eta} \vert \leq \varepsilon \sqrt{E_R(\omega) E_R(\eta)}.
 \end{equation}%
 \end{lem}
 \begin{proof}%
Again we sketch the argument; for the analytic details we refer, e.g. to \cite[proof of Theorem 2]{Mi74}. Let first $\tau$ be any closed $1$-form on $R$. By the bound on the dilatation we have, pointwise,

 \begin{equation}\label{eq:qcIneq2}
\Vert \phi^{*}\tau \Vert^2\,  d\vol(R') \leq (1+\varepsilon) \Vert \tau \Vert^2\,  d\vol(R),
 \end{equation}%
where $\Vert \ \Vert$ and $d \vol$ are the pointwise norms and volume elements on the respective surfaces. Integration yields $E_{R'}(\phi^{*} \tau) \leq (1+\varepsilon)E_R(\tau)$. If now $\tau$ is harmonic, and $\tau'$ the harmonic form in the cohomology class of $\phi^{*} \tau$ then, by the minimizing property of harmonic forms, $E_{R'}(\tau') \leq (1+\varepsilon)E_R(\tau)$. Repeating these arguments for $\phi^{-1}:R \to R'$ we get altogether (for harmonic $\tau$ and $\tau'$),
 \begin{equation}\label{eq:qcIneq31}
(1 - \varepsilon)E_R(\tau)\leq \frac{1}{(1 + \varepsilon)}E_R(\tau) \leq E_{R'}(\tau') \leq (1 + \varepsilon)E_R(\tau).
 \end{equation}%
The inequality for the scalar product now follows from \bbref{Lemma}{L:LinAlgIneq} \end{proof}%

\section{The forms $\pmb{\sigma_1}$ and $\pmb{\mT_1}$}\label{sec:sigT1}

Here begins the second part of the paper. In all that follows $S$ is a compact hyperbolic Riemann surface of genus $g \geq 2$ and $\gamma$ is a \emph{nonseparating} simple closed geodesic on $S$. In addition, a geodesic canonical homology basis ${\rm A} = (\alpha_1, \alpha_{2},\ldots,\alpha_{2g-1},\alpha_{2g})$ is given, with the intersection convention that, as in \bbref{Section}{sec:Introd}, $\alpha_{2k-1}$ intersects $\alpha_{2k}$, $k=1,\dots,g$. Furthermore, $\alpha_2$ coincides with $\gamma$.

When, in a limit process, $\gamma$ becomes short, then the first two members $\sigma_1$, $\sigma_2$ of the dual basis of harmonic forms become more and more singular, while the other members are hardly affected. In this section we analyse $\sigma_1$, beginning with energy bounds in terms of a certain conformal capacity which too varies only little when $\ell(\gamma) \to 0$.

\subsection{The capacity of $\pmb{S^{\times}}$}\label{sec:capS}

We denote by $S^{\times}$  the surface $S$ cut open along the nonseparating geodesic $\gamma$. The \emph{capacity} of $S^{\times}$, more generally the capacity $\capa(M)$ of any connected surface $M$ with two disjoint boundary components $\partial_1$, $\partial_2$ is defined as the infimum
 \begin{equation}\label{eq:capaM}%
\capa(M) = \inf\{ E_M(d f) \mid  f\vert_{\partial_1}=0, \,  f\vert_{\partial_2}=1\},
 \end{equation}%
where the competing functions $f$ are piecewise smooth. We shall estimate $\capa(S^{\times})$ when $\gamma$ becomes small while all other Fenchel-Nielsen parameters of $S$ are kept fixed. For this we look at a suitable subset $\frak{M}$  of $S^{\times}$ whose capacity remains bounded under such a deformation. We shall define it as follows (see also \bbref{Section}{sec:Introd}, \bbref{Fig.}{fig:curvesSG}).

Let $\gamma_1, \gamma_2$ be the two boundary geodesics of $S^{\times}$. For $i=1,2$ we have the half collars $C(\gamma_i)$ consisting of the points at distance $< \cl(\ell(\gamma))$ from $\gamma_i$ (see \beqref{eq:clfctn}). The larger boundary, $b_i$, of $C(\gamma_i)$ is a parallel curve of length $\geq 2$, and we let $c_i$ in $C(\gamma_i)$ be the parallel curve that has length $\ell(c_i) = 1$. It splits $C(\gamma_i)$ into two ring domains: $C_i$ from $\gamma_i$ to $c_i$ and $B_i$ from $c_i$ to $b_i$. In what follows we shall understand $C_i$ and $B_i$ to be the closures of these domains. The subsurface $\frak{M}$ is obtained by cutting a way the parts $C_1, C_2$ and taking the closure:

 \begin{equation}\label{eq:kernelK}%
\frak{M} = S^{\times} \setminus (C_1\cup C_2) \cup c_1 \cup c_2.
 \end{equation}%
In \bbref{Section}{sec:TwistInf} we shall define $\frak{M}$ in a similar way for a Riemann surface with a pair of cusps. We set
 \begin{equation}\label{eq:inversecap}%
\Gamma = \frac{1}{\capa(\frak{M})}.
 \end{equation}%
Then $[0,\Gamma] \times \Sp_1$ is the cylinder that has the same capacity as $\frak{M}$. We shall call $\frak{M}$ the \emph{main part} of $S$.

To compare this with the capacity of $S^{\times}$ we look at the lengths of the half collars. The conformal mapping $\psi_{\gamma}$ as in
\beqref{eq:psigamma},   with $F_{\gamma}$ as in  \beqref{eq:Fgamma} sends the closure of $C(\gamma_1)$ onto the flat cylinder $[0,F_{\gamma}(\cl(\ell(\gamma)))]\times \Sp_1$. The parts $C_1, B_1$ go to the parts
\begin{equation*}
 [0, \frac{1}{2}L_{\gamma}]\times \Sp_1 \text{ \ and \ } [\frac{1}{2}L_{\gamma},\frac{1}{2}L_{\gamma}+d_{\gamma}]\times \Sp_1,
\end{equation*}
respectively. Observe the notation: the quantity $\frac{1}{2}L_{\gamma}+d_{\gamma}$ we are using here is the same as $M(\ell(\gamma))$ in \beqref{eq:Lleta}. From this equation it follows by elementary computation that
 \begin{equation}\label{eq:lengcyl}%
L_{\gamma} = \frac{\pi}{\ell(\gamma)}-\frac{2\arcsin \ell(\gamma)}{\ell(\gamma)}>4, \quad d_{\gamma}\geq \frac{1}{2}. 
 \end{equation}%
and the inequality holds under the assumption that $\ell(\gamma) \leq \frac{1}{2}$. For $C_2,B_2$ we have the same properties. Since $B_1, B_2$ are disjoint subsets of $\frak{M}$ that are conformally equivalent to $[0,d_{\gamma}]\times \Sp_1$ it follows immediately that
 \begin{equation}\label{eq:Gambg1}%
 \capa(\frak{M}) \leq 1, \quad \Gamma \geq 1. 
 \end{equation}%
However, the lower bounds of $\capa(\frak{M})$ (e.g.\ \bbref{Section}{sec:RelaxDual}) are more important.

From a conformal point of view, $S^{\times}$ is $\frak{M}$ with two copies of $[0,L_{\gamma}/2]\times \Sp_1$ attached along $c_1, c_2$. In the following, $S_L^{\times}$ is $\frak{M}$ with two copies $Z_1, Z_2$ of $[0,L/2]\times \Sp_1$ attached, for some arbitrary $L > 0$.

 \begin{lem}\label{lem:capStildaL}%
There exists a universal constant $\frak{z} < 2.3$ such that for any $L>0$
 \begin{equation*}%
\frac{1}{L + \frak{z} \Gamma} \leq \capa(S_L^{\times}) \leq \frac{1}{L+\Gamma}.
 \end{equation*}%
 \end{lem}%
 \begin{proof}%
The upper bound is obtained out of the test function $F$ that grows linearly from 0 to $\frac{L/2}{L+\Gamma}$ on $Z_1$, linearly from $1-\frac{L/2}{L+\Gamma}$ to 1 in $Z_2$, and is harmonic with boundary values $\frac{L/2}{L+\Gamma}, 1-\frac{L/2}{L+\Gamma}$ on $\frak{M}$.

For the lower bound we let $f$ be the harmonic function that is constant equal to 0 on the left boundary of $S_L^{\times}$ and constant equal to 1 on the right boundary. For the harmonic 1-form $\omega = d f$ we then have
 \begin{equation}\label{eq:capSL1}%
E(\omega) = \capa(S_L^{\times}).
 \end{equation}%
With constants $a, b \in \R$ (to be dealt with at the end) we have the decomposition of $f$ into its linear and nonlinear parts
 \begin{equation}\label{eq:capSL2}%
f(x,y) = a x + h_1(x,y)
 \end{equation}%
on $Z_1$ and
 \begin{equation}\label{eq:capSL3}%
f(x,y) = b x + 1-b L/2 + h_2(x,y)
 \end{equation}%
on $Z_2$. (For either cylinder we adopt the convention that $x \in [0, L/2]$ and the constant terms are deduced from the fact that the nonlinear parts have mean values $\int_0^1 h_i(x,y)d y =0$.). We remark that \beqref{eq:capSL2} and \beqref{eq:capSL3} also hold on the adjacent ring domains  $\psi_{\gamma}(B_1)=[L/2,L/2+d_{\gamma}]\times \Sp_1$ and $\psi_{\gamma}(B_2) = [-d_{\gamma},0]\times \Sp_1$. Based on this we construct a form that serves as test form on $\frak{M}$ by dampening down the nonlinear parts of $f$ as follows. First on $Z_1 \cup B_1$: using the function $\chi$ that is $0$ on $[0,L/2]\times \Sp_1$ and then grows linearly from $0$ to $1$ on $[L/2, L/2 + d_{\gamma}] \times \Sp_1$ we set
 \begin{equation*}%
f_{\chi}(x,y)= a x + \chi(x,y)h_1(x,y)
 \end{equation*}%
on $Z_1 \cup B_1$. We proceed similarly on $B_2 \cup Z_2$ and complete the definition setting $f_{\chi}=f$ on $\frak{M} \smallsetminus (B_1 \cup B_2)$. Now $f_{\chi}$ is a test function on $\frak{M}$ that assumes the constant value $aL/2$ on the boundary component $c_1$ of $\frak{M}$ and $1-b L/2$ on $c_2$. The energy of the harmonic function that is constant $r$ on one boundary and constant $r+s$ on the other boundary of $\frak{M}$ can be obtained by rescaling the function that satisfies the capacity problem and is equal to $s^2 \cdot \capa( \frak{M})$. For $\omega_{\chi}=d f_{\chi}$ we have therefore
 \begin{equation}\label{eq:capSL4}%
E_\frak{M}(\omega_{\chi}) \geq (1-\frac{a+b}{2}L)^2 \capa(\frak{M}).
 \end{equation}%
By the dampening down \bbref{Lemma}{lem:dampening}\;i) (with $\delta = 0$ and $D_{int} = Z_l= [0, d_{\gamma}] \times \Sp_1$ conformally equivalent to $B_i$) we get $E_{B_i}(\omega_{\chi})\leq \frak{z}E_{B_i}(\omega)$, $i=1,2$, with $\frak{z} = 2 + 2 \pi^{-2}$.  Hence,
 \begin{equation}\label{eq:capSL5}%
E_\frak{M}(\omega_{\chi}) \leq \frak{z}E_\frak{M}(\omega).
 \end{equation}%
By the \bbref{Orthogonality lemma}{thm:lemorthogonal} we further have
 \begin{equation}\label{eq:capSL6}%
E(\omega) \geq \frac{a^2L}{2} + \frac{b^2L}{2} + E_\frak{M}(\omega).
 \end{equation}%
Bringing \beqref{eq:capSL4}--\beqref{eq:capSL6} together we obtain
 \begin{equation}\label{eq:capSL7}%
E(\omega) \geq \frac{a^2 + b^2}{2}L + \frac{1}{\frak{z} \Gamma}(1-\frac{a+b}{2}L)^2.
 \end{equation}%
Now, $a$ and $b$ are not known. We therefore replace the right hand side of \beqref{eq:capSL7} by the infimum over all real values of $a, b$. This infimum is readily seen to be achieved for $a = b = 1/(L+ \frak{z}\Gamma)$, and the corresponding value of the right hand side of \beqref{eq:capSL7}, somewhat accidentally, also equals $1/(L+ \frak{z}\Gamma)$. This completes the proof.
\end{proof}%

If we paste the two boundary geodesics of $S_L^{\times}$ together, with respect to some arbitrary twisting parameter, then for the resulting surface $S_L$ we have an analogous result. To make this precise we assume that $\alpha_1, \dots, \alpha_{2g}$ and $\sigma_1, \dots, \sigma_{2g}$ have the same meanings for $S_L$ as for $S$, and that $\alpha_2$ is the curve into which the two boundary curves of $S^{\times}$ have been pasted together. We shall, however, write $\sigma_1 = \sigma_{1,L}$ to indicate that the form is defined on $S_L$.

 \begin{lem}\label{lem:capStildaL2}%
There exists a universal constant $\frak{z} < 2.3$ such that for any $L>0$
 \begin{equation*}%
\frac{1}{L + \frak{z} \Gamma} \leq E(\sigma_{1,L}) \leq \frac{1}{L+\Gamma}.
 \end{equation*}%
 \end{lem}%
 \begin{proof}%
The proof is identical to the preceding one, owing to the fact that $\sigma_{1,L}$ has a primitive $f_L$ on $S_L^{\times}$. We fix its additive constant such that $f_L$ has mean value 0 on the left boundary of $S_L^{\times}$ and mean value 1 on the right. The splitting into linear and nonlinear parts is then again $f_L(x,y) = a x + h_1(x,y)$ on $Z_1 \cup B_1$, $f_L(x,y) = b x + 1-b L/2 + h_2(x,y)$ on $B_2 \cup Z_2$ and the earlier arguments go through.
\end{proof}
\subsection{The form $\pmb{\mT_1}$}\label{sec:sigma1Q}

When $\ell(\gamma) \to 0$ or when $L\to \infty$ the energy of $\sigma_1$ goes to 0, and $\sigma_1$ disappears in the limit. We introduce therefore a renormalised form that ``survives''. To this end we let $\mT_1$ be $\sigma_1$ multiplied with a constant factor such that $\mT_1$ on the collar $C(\gamma)$ has a representation

 \begin{equation}\label{eq:EsjQ4}%
\mT_1 =  d x + d h_1^{\nlin},
 \end{equation}%
i.e. $\mT_1$ is normalized such that the linear part of its representations on the flat cylinders $Z_1 \cup B_1$ and $Z_2 \cup B_2$ is equal to $d x$. 

The form $\mT_1$ is defined similarly on any $S_L$ and is then written $\mT_{1,L}$. We observe that, by \bbref{Theorem}{L:sigma1} below, $\mT_{1,L}$ depends only on the conformal class of $S_L$ and not the particular description we are using here.

We shall show in \bbref{Section}{sec:HvsR}, \bbref{Theorem}{thm:Converge}  that as $L \to \infty$, $\mT_{1,L}$ converges, locally uniformly, to the unique harmonic form $\mT_1^{\sG}$ on $S^{\sG}$ that has zero periods and poles with linear parts $d x$ in the cusps.

As $L\to \infty$ the energy of $\mT_{1,L}$ and the absolute value of its period over $\alpha_1$ go to infinity. This is caused by the linear part in the cylinder. We introduce therefore  an \emph{essential energy} $\Ess(\mT_{1,L})$ and an \emph{essential period} $\mathscr{P}_1(\mT_{1,L})$ setting
\begin{equation}\label{eq:Essigma2}%
 \begin{aligned}
\Ess(\mT_{1,L}) &= E_\frak{M}(\mT_{1,L}) + E_C(\mT_{1,L}^{\nlin}),\\
\mathscr{P}_1(\mT_{1,L}) &= \int_{\alpha_1 \cap \pt\frak{M}} \mT_{1,L} + \int_{\alpha_1\cap \pt C} \mT_{1,L}^{\nlin},
 \end{aligned}
\end{equation}%
where $C$ is the inserted cylinder conformally equivalent to and identified with $[0,L] \times \Sp_1$ and for any harmonic form $\omega$ on $C$, $\omega^{\nlin}$ denotes its nonlinear part.

We first make the following observation concerning $\sigma_1$. %

 \begin{thm}\label{L:sigma1}%
For $\sigma_{1,L}$ on $S_L$ the linear part in the cylinders is $\alpha d x$ with $\alpha = E_{S_L}(\sigma_{1,L})$.
 \end{thm}%

 \begin{proof}%
By a well known formula for the period matrix (e.g.\ \cite[III.2.4]{FK92}) we have
 \begin{equation*}%
\int_{S_L} \sigma_{1,L}\wedge \star \sigma_{1,L} = \int_{\alpha_2}\star \sigma_{1,L}.
 \end{equation*}%
The linear part of $\star \sigma_{1,L}$ is $\alpha d  y$. Hence, the integral on the right hand side equals $\alpha$ and the Theorem follows.
 \end{proof}%
Since $\mT_{1,L}=\frac{1}{\alpha}\sigma_{1,L}$ an immediate consequence is that 
 \begin{equation}\label{eq:EnSigma1}%
E(\mT_{1,L}) = \frac{1}{\alpha^2}E(\sigma_{1,L}) = \frac{1}{\alpha}=\frac{1}{E(\sigma_{1,L})}
 \end{equation}%
By additivity the period of $\mT_{1,L}$ over $\alpha_1$ is the sum of $\mathscr{P}_1(\mT_{1,L})$ and $L$. On the other hand, $\mT_{1,L}$ has period $\frac{1}{\alpha}$. Hence,
 \begin{equation}\label{eq:eq:kappa1TL}%
\mathscr{P}_1(\mT_{1,L}) =\frac{1}{\alpha}-L.
 \end{equation}%
By the Orthogonality lemma 
\begin{equation}\label{eq:EssTL}%
\Ess(\mT_{1,L}) = E(\mT_{1,L}) -L = \frac{1}{E(\sigma_{1,L})}-L = \frac{1}{\alpha}-L
 \end{equation}%
from which it follows that

\ref{eq:Esskap}
 \begin{equation}\label{eq:Esskap}%
\Ess(\mT_{1,L}) = \mathscr{P}_1(\mT_{1,L}).
 \end{equation}%
This allows us, among other things, to write
 \begin{equation}\label{eq:sigmaT1}%
\sigma_{1,L}=\frac{1}{\mathscr{P}_1(\mT_{1,L})+L}\, \mT_{1,L},
\quad
E(\sigma_{1,L}) = \frac{1}{ \mathscr{P}_1(\mT_{1,L})+L}.
 \end{equation}%
Combining \beqref{eq:EssTL} with \bbref{Lemma}{lem:capStildaL2} we get the following.

\begin{lem}\label{L:sigma1ess}%
The essential energy of $\mT_{1,L}$ on $S_L$ has the following bounds, where $\Gamma = 1/\capa(\frak{M})$ and $\frak{z} \leq 2.3$ is the constant as in \bbref{Lemmata}{lem:capStildaL}, \bbref{}{lem:capStildaL2}
 \begin{equation*}%
\Gamma \leq \Ess(\mT_{1,L}) = \mathscr{P}_1(\mT_{1,L}) \leq \frak{z} \Gamma. \end{equation*}%
\par\vspace{-24pt}
\hfill \qedsymbol
 \end{lem}%
In \bbref{Section}{sec:TwoCusps} we shall compare $\Ess(\mT_{1,L})$ with the corresponding essential energy of the harmonic form $\mT_1^{\sG}$ on the limit surface $S^{\sG}$. For this we prove here the following preparatory lemma.
 \begin{lem}\label{L.kappa1}%
Assume $L \geq 1$. Then, for any $\tilde{L} \geq L$ we have
 \begin{equation*}%
\vert \Ess(\mT_{1,L})-\Ess(\mT_{1,\tilde{L}})\vert < e^{-2\pi L} \min\{ \Ess(\mT_{1,L}), \Ess(\mT_{1,\tilde{L}}) \}.
 \end{equation*}%
 \end{lem}%
 \begin{proof}%
The technique is to begin with the forms $\sigma_{1,L}$ on $S_L$ and $\sigma_{1,\tilde{L}}$ on $S_{\tilde{L}}$ which both have period 1 over the first base cycle, then dampen their nonlinear parts down so as to obtain comparison test forms and finally translate the result to $\mT_{1,L}$ and $\mT_{1,\tilde{L}}$ using \beqref{eq:EssTL}.

In the flat cylinder $Z = [-\frac{1}{2}(L+1),\frac{1}{2}(L+1)] \times \Sp_1$ the form $\sigma_{1,L}$ has a representation $\sigma_{1,L} = d H = \alpha d x + d h$, where $h$ is the nonlinear part of $H$. With a constant $v \leq \frac{1}{2}(L+1)$ to be determined later (we shall take $v=\frac{1}{8}$) we set  
 \begin{equation*}%
d = \tfrac{1}{2}(L+1)-v, \quad D_{\text{int}} = [-v,v] \times \Sp_1= [-\tfrac{1}{2}(L+1)+d,\tfrac{1}{2}(L+1)-d]\times \Sp_1
 \end{equation*}%
and let $\chi$ be the cut off function in $Z$ that is equal to 1 outside of the part $D_{\text{int}}$, goes linearly down from 1 to 0 on $[-v,0] \times \Sp_1$ and then linearly up again from 0 to 1 on $[0,v] \times \Sp_1$. We let $\sigma_{1,\chi}$ be the form that coincides with $\sigma_{1,L}$ outside of $D_{\text{int}}$ and inside has the representation $\sigma_{1,\chi} = \alpha d x + d (\chi h)$. Its period over the first base cycle is again equal to 1. By \bbref{Lemma}{lem:dampening},iii) applied twice (with $w = v$ and $\delta = d$) its energy satisfies %
 \begin{equation*}%
E(\sigma_{1,\chi}) \leq E(\sigma_{1,L}) + 2(1+\frac{1}{2 \pi^2 v^2}) e^{-4\pi d}(E(\sigma_{1,L})-\alpha^2(L+1)).
 \end{equation*}%
Here $E(\sigma_{1,L})-\alpha^2(L+1)$ is an upper bound for the energy of the nonlinear part of $\sigma_{1,L}$ in $Z$. Abbreviating
 \begin{equation*}%
\mu =2(1+\frac{1}{2 \pi^2 v^2}) e^{-4\pi d}
 \end{equation*}%
and recalling that $E(\sigma_{1,L})=\alpha$ we rewrite the inequality in the form 
 \begin{equation*}%
E(\sigma_{1,\chi}) \leq  (1+m)\alpha \quad \text{with} \quad m = \mu \cdot (1-\alpha(L+1)).
 \end{equation*}%
In a similar way, on $\tilde{Z} = [-\frac{1}{2}(\tilde{L}+1),\frac{1}{2}(\tilde{L}+1)] \times \Sp_1$ the form $\sigma_{1,\tilde{L}}$ has a representation $\sigma_{1,\tilde{L}} = d \tilde{H} = \tilde{\alpha} d x + d \tilde{h}$, where $\tilde{h}$ is the nonlinear part of $\tilde{H}$. Here we set 
 \begin{equation*}%
\tilde{D}_{\text{int}} = [-\tfrac{1}{2}(\tilde{L}+1)+ d, \tfrac{1}{2}(\tilde{L}+1) - d]
 \end{equation*}%
and let $\tilde{\chi}$ be the cut off function on $\tilde{Z}$ that is equal to 1 outside of $\tilde{D}_{\text{int}}$, goes linearly down from 1 to 0 for $x \in [-\tfrac{1}{2}(\tilde{L}+1)+ d, -\tfrac{1}{2}(\tilde{L}+1)+ d+v]$, remains 0 until $x= \tfrac{1}{2}(\tilde{L}+1)- d -v$
and then goes linearly up again from 0 to 1 for $x \in [\tfrac{1}{2}(\tilde{L}+1)- d -v,\tfrac{1}{2}(\tilde{L}+1)- d]$. We let $\sigma_{1,\tilde{\chi}}$ be the form that coincides with $\sigma_{1,\tilde{L}}$ outside of $\tilde{D}_{\text{int}}$ and inside has the representation $\sigma_{1,\tilde{\chi}} = \tilde{\alpha}d x + d(\tilde{\chi} \tilde{h})$. Its period over the first base cycle is 1. In the same way as before we have the inequality
%
 \begin{equation*}%
E(\sigma_{1,\tilde{\chi}}) \leq (1+\tilde{m})\tilde{\alpha} \quad \text{with} \quad \tilde{m} = \mu \cdot (1-\tilde{\alpha}(\tilde{L}+1)).
 \end{equation*}%
On the part $D=[-\frac{1}{2}(\tilde{L}-L),\frac{1}{2}(\tilde{L}-L)] \times\Sp_1$ of $\tilde{Z}$ the form $\sigma_{1,\tilde{\chi}}$ coincides with $\tilde{\alpha} d x$. Therefore, if $\sigma_{1,\tilde{\chi}}^{\scriptscriptstyle\leftarrow}$ is the restriction of $\sigma_{1,\tilde{\chi}}$ to $S_{\tilde{L}} \setminus D$ and if we identify $S_{\tilde{L}} \setminus D$ with $S_L$ (by gluing the ends of $S_{\tilde{L}} \setminus D$ together with the same twist), then $\sigma_{1,\tilde{\chi}}^{\scriptscriptstyle\leftarrow}$ becomes a test form on $S_L$ with period over the first base cycle equal to $1-\tilde{\alpha}(\tilde{L}-L)$  and energy
 \begin{equation*}%
E(\sigma_{1,\tilde{\chi}}^{\scriptscriptstyle\leftarrow}) = E(\sigma_{1,\tilde{\chi}}) - \tilde{\alpha}^2(\tilde{L}-L).
 \end{equation*}%
In a converse way we may cut open $Z$ in the middle, insert $D$ and extend $\sigma_{1,\chi}$ to a test form $\sigma_{1,\chi}^{\scriptscriptstyle\rightarrow}$ on $S_{\tilde{L}}$ by defining it to be equal to $\alpha d x$ on $D$. Then $\sigma_{1,\chi}^{\scriptscriptstyle\rightarrow}$ has period $1 + \alpha(\tilde{L}-L)$ over the first base cycle and its energy is
 \begin{equation*}%
E(\sigma_{1,\chi}^{\scriptscriptstyle\rightarrow}) = E(\sigma_{1,\chi}) +\alpha^2(\tilde{L}-L). 
 \end{equation*}%
To get comparison forms with periods 1 we set, abbreviating $\bar{L}=\tilde{L}-L$,
 \begin{equation*}%
s_{1,\tilde{\chi}}^{\scriptscriptstyle\leftarrow} = \tfrac{1}{1-\tilde{\alpha}\bar{L}}\sigma_{1,\tilde{\chi}}^{\scriptscriptstyle\leftarrow},
\quad
s_{1,\chi}^{\scriptscriptstyle\rightarrow} = \tfrac{1}{1+\alpha \bar{L}}\sigma_{1,\chi}^{\scriptscriptstyle\rightarrow}.%
\end{equation*}%

We now have $\alpha = E(\sigma_{1,L}) \leq E(s_{1,\tilde{\chi}}^{\scriptscriptstyle\leftarrow}) =(1-\tilde{\alpha}\bar{L})^{-2}E(\sigma_{1,\tilde{\chi}}^{\scriptscriptstyle\leftarrow})$ and $\tilde{\alpha} = E(\sigma_{1,\tilde{L}}) \leq E(s_{1,\chi}^{\scriptscriptstyle\rightarrow}) = (1+\alpha \bar{L})^{-2}E(\sigma_{1,\chi}^{\scriptscriptstyle\rightarrow})$.
From the earlier relations we get
 \begin{equation*}%
\alpha \leq \frac{\tilde{\alpha}}{(1-\tilde{\alpha}\bar{L})^2}(1+\tilde{m}-\tilde{\alpha}\bar{L}),
\quad
\tilde{\alpha} \leq \frac{\alpha}{(1+\alpha \bar{L})^2}(1+m+ \alpha\bar{L}).
 \end{equation*}%
We are ready to estimate the difference $\Delta = \Ess(\mT_{1,\tilde{L}}) - \Ess(\mT_{1,L})$. By \beqref{eq:EssTL}, $\Delta =\frac{1}{\tilde{\alpha}}-\frac{1}{\alpha}-\bar{L}$, and the above inequalities yield
 \begin{equation*}%
 -\frac{m}{\alpha}\leq-\frac{m(1+\alpha \bar{L})}{\alpha(1+m+\alpha\bar{L})}\leq\Delta \leq \frac{\tilde{m}(1-\tilde{\alpha}\bar{L})}{\tilde{\alpha}(1+\tilde{m}-\tilde{\alpha} \bar{L})} \leq \frac{\tilde{m}}{\tilde{\alpha}}.
 \end{equation*}%
By \beqref{eq:EssTL} and the definition of $m$ and $\tilde{m}$ we have $\frac{m}{\alpha} \leq \frac{\mu}{\alpha}(1 - \alpha L) = \mu \Ess(\mT_{1,L})$, $\frac{\tilde{m}}{\tilde{\alpha}} \leq \frac{\mu}{\tilde{\alpha}}(1 - \tilde{\alpha} \tilde{L}) = \mu \Ess(\mT_{1,\tilde{L}})$. This yields, anticipating on the right hand side that $\mu < 1$,
 \begin{equation*}%
\vert \Ess(\mT_{1,L}) -  \Ess(\mT_{1,\tilde{L}})\vert \leq \mu \, \max\{ \Ess(\mT_{1,L}), \Ess(\mT_{1,\tilde{L}}) \} \leq \frac{\mu}{1-\mu}\min\{ \Ess(\mT_{1,L}), \Ess(\mT_{1,\tilde{L}}) \}.
 \end{equation*}%
The value of $\mu$ depends on the choice of $v$ and is close to its minimum for $v = 1/8$ for which case we have the upper bound 
$\mu  \leq \exp(-2\pi L-2)$ and the lemma follows.
 \end{proof}%

In the next lemma we estimate the variation of the essential energy of $\mT_1$ under quasiconformal mappings. It will be used in \bbref{Section}{sec:TwoCusps} to compare $\Ess(\mT_1^{\sG})$ with $\Ess(\mT_1^{\sF})$. We prove the lemma for the following situation. $\phi : R' \to R$ is a  $(1+\varepsilon)$-quasi conformal homeomorphism between compact hyperbolic surfaces of genus $g$. On $R'$ and $R$ nonseparating simple closed geodesics $\gamma'$, $\gamma$ of lengths $\ell(\gamma') = \ell(\gamma)$ are given and $\phi$ maps the standard collar $C(\gamma')$ conformally (and, hence, isometrically) to $C(\gamma)$. We denote by $\sigma_1$, $\sigma'_1$ the first elements of the dual bases corresponding to homology bases on $R$, $R'$ that have as second member $\alpha_2 = \gamma$, $\alpha_2' = \gamma'$, and $\mT_1$, $\mT'_1$ are the renormalisations of $\sigma_1$, $\sigma_1'$ whose linear parts in the collars are equal to $d x$.

 \begin{lem}\label{L.kapqcf}%
Under the above conditions for the compact Riemann surfaces $R$, $R'$ we have%
 \begin{equation*}
\vert \Ess(\mT_1)-  \Ess(\mT'_1) \vert  \leq\frac{\varepsilon}{1-\varepsilon}\min\{ \Ess(\mT_1),\Ess(\mT'_1)\}.
 \end{equation*}%
 \end{lem}%
 \begin{proof}%
We let $\frak{M} \subset R$, $\frak{M}' \subset R'$ be the main parts with respect to the cylinders $C(\gamma)$, $C(\gamma')$ as in \beqref{eq:kernelK}. By \bbref{Lemma}{L:sigma1}  the linear parts of $\sigma_1$ and $\sigma'_1$ in the collars are $\alpha d x$ and $\alpha' d x$ respectively, with $\alpha = E_R(\sigma_1)$ and $\alpha' = E_{R'}(\sigma'_1)$. Now $\phi^{*} \sigma_1$ is in the cohomology class of $\sigma'_1$ and therefore $\alpha' \leq E_{R'}(\phi^{*} \sigma_1)$.  Since $\phi$ is $(1+\varepsilon)$-quasi conformal on $\frak{M}'$ and conformal on $R' \setminus \frak{M}'$ we get
 \begin{equation*}%
\alpha' \leq (1+\varepsilon)E_{\frak{M}}(\sigma_1) + E_{R \setminus \frak{M}}(\sigma_1) = \varepsilon E_\frak{M}(\sigma_1) + \alpha   
 \end{equation*}%
(using that pointwise, as in \beqref{eq:qcIneq2}, $\Vert (\phi^{*} \sigma_1)(p)\Vert^2 \, d \text{vol}(R') \leq (1+\varepsilon)\Vert \sigma_1(\phi(p)) \Vert ^2\, d \text{vol}(R)$) for $p \in \frak{M}'$). By the Orthogonality lemma
 \begin{equation*}%
E_\frak{M}(\sigma_1) \leq E_R(\sigma_1) - \alpha^2 L = \alpha - \alpha^2 L.
 \end{equation*}%
It follows that
 \begin{equation*}%
\alpha' \leq \alpha + \varepsilon(\alpha - \alpha^2 L).
 \end{equation*}%
Using that the essential energies are given by $\Ess(\mT_1) = \frac{1}{\alpha} -L$ and $\Ess(\mT'_1) = \frac{1}{\alpha'} -L$ \beqref{eq:EssTL} we get
 \begin{equation*}%
\Ess(\mT_1) - \Ess(\mT'_1) = \frac{1}{\alpha}-\frac{1}{\alpha'}\leq \frac{1}{\alpha}\cdot \frac{\varepsilon(1-\alpha L)}{1+\varepsilon(1-\alpha L)}\leq \frac{\varepsilon(1-\alpha L)}{\alpha} = \varepsilon \Ess(\mT_1).
 \end{equation*}%
Hence, $(1-\varepsilon) \Ess(\mT_1) \leq \Ess(\mT'_1)$. Applying this result to $\phi^{-1}$ we have, in the same way, $(1-\varepsilon) \Ess(\mT'_1) \leq \Ess(\mT_1)$. The two inequalities together yield \bbref{Lemma}{L.kapqcf}.
 \end{proof}%

\subsection{Bounds for cap($\frak{M}$) and $\pmb{\Gamma}$}\label{sec:bcapM}

The estimates in \bbref{Lemmata}{lem:capStildaL2}, \bbref{}{L:sigma1ess}, \bbref{}{L.kappa1} and \bbref{}{L.kapqcf} are in terms of $\Gamma = 1/\capa(\frak{M})$ which is an analytic constant. We further estimate $\capa(\frak{M})$ geometrically using the following quantities: $d_{\frak{M}}$ is the distance between the two boundary components and $\rho_{\frak{M}} := \frac{1}{4}\min\{1, \sys(\frak{M})\}$, where $\sys(\frak{M})$ is the \emph{systole}, i.e.\ the length of the shortest closed geodesic on $\frak{M}$. With these constants we have the following rough upper and lower bounds for the capacity:
 \begin{thm}\label{eq.bdscapaM}%
 \begin{equation*}%
\frac{\rho_{\frak{M}}}{d_{\frak{M}}} < \capa(\frak{M}) = \frac{1}{\Gamma} <  \min\{ \frac{4\pi(g-1)}{d_{\frak{M}}^2},1 \}.
 \end{equation*}%
 \end{thm}%
 \begin{proof}%
The inequality $\frac{1}{\Gamma} \leq 1$ is given in \eqref{eq:Gambg1}. For the second inequality we follow the lines of \cite{Mu17} where a stronger result is proved.

For the lower bound we consider the harmonic function $h$ on $\frak{M}$ that realises the boundary conditions $h\vert_{c_1} = 0$, $h\vert_{c_2} = 1$ and restrict it to the ribbon $W=\{p \in \frak{M} \mid \dist(p,A) \leq \rho_{\frak{M}}$, where $A$  is the shortest geodesic from $c_1$ to $c_2$. This geodesic has length $d_{\frak{M}}$ and is orthogonal to the boundary at the endpoints. By the size of $\rho_{\frak{M}}$ the geodesic arcs of length $\rho_{\frak{M}}$ emanating orthogonally from $A$ are pairwise disjoint. This follows by the lines of the proof of the main theorem in \cite{Ch77}, where a corresponding result is shown for the diameter of a Riemann surface. Introducing Fermi coordinates for $W$ based on $A$ we get therefore a coordinate description of $W$ that makes it isometric to the set 
\begin{equation*}
W' = \{(x,y) \in \R^2 \mid 0 \leq x \leq d_{\frak{M}}, \;-\rho_{\frak{M}} \leq y \leq \rho_{\frak{M}} \}
\end{equation*}
endowed with the Riemannian metric $d s^2 = \cosh^2(y)d x^2 + d y^2$. This is true except for a negligible difference at the left and right end of $W'$ which will not affect our rough estimates. Now $W'$ in turn is conformally equivalent to the flat strip (with the standard Euclidean metric) $W'' = \{(x,y) \in \R^2 \mid 0 \leq x \leq d_{\frak{M}}, -\rho'' \leq y \leq \rho'' \}$, where $\rho'' = 2\arctan \tanh \frac{1}{2} \rho_{\frak{M}}$ (use \beqref{eq:Fgamma} with $\ell(\gamma)$ replaced by 1). Identifying $W$ with $W''$ we can now write
 \begin{align*}%
\capa(\frak{M}) &> \int_{W''} \Vert d h \Vert^2 = 
 \int_{-\rho''}^{\rho''}\int_0^{d_{\frak{M}}} \left( \frac{\partial h}{\partial x}\right)^2+\left( \frac{\partial h}{\partial y}\right)^2 d x d y\\
&\geq \int_{-\rho''}^{\rho''}\int_0^{d_{\frak{M}}} \left( \frac{\partial h}{\partial x}\right)^2 \geq \int_{-\rho''}^{\rho''}\frac{1}{d_{\frak{M}}}d y =  \frac{2\rho''}{d_{\frak{M}}} > \frac{\rho_{\frak{M}}}{d_{\frak{M}}}.
 \end{align*}%
For the upper bound we take the test function $f = \frac{1}{d_{\frak{M}}}F$ on $\frak{M}$, where $F$ is defined by the condition $F(p) = \dist(p,c_1)$ if $d(p,c_1) \leq \frac{1}{2}d_{\frak{M}}$, $F(p) = d_{\frak{M}}-\dist(p,c_2)$ if $d(p,c_2) \leq \frac{1}{2}d_{\frak{M}}$ and $F(p) = \frac{1}{2}d_{\frak{M}}$ in the remaining cases. Now $\Vert d F\Vert  \leq 1$, $\Vert d f \Vert \leq \frac{1}{d_{\frak{M}}}$ and $\capa(\frak{M}) \leq \int_{\frak{M}}\Vert d f \Vert^2 \leq \area(\frak{M})d_{\frak{M}}^{-2}=4\pi(g-1)d_{\frak{M}}^{-2}$.
\end{proof}%

\section{Relaxed dual basis}\label{sec:RelaxDual}

As pointed out in \bbref{Section}{sec:Introd} two difficulties arise in the nonseparating case: the homology changes in the limit and the normalisation condition $\int_{\alpha_1}\sigma_j = 0$, for $j=2,\dots,2g$ disappears. In this section we deal with the second problem by introducing a weaker concept of dual basis of harmonic forms $\sigma_1, \tau_2, \tau_3, \dots, \tau_{2g}$,  where the $\tau_j$ have the same periods as the $\sigma_j$. However, the condition $\int_{\alpha_1} \tau_j=0$ is dropped and replaced by the condition that the energy of these forms is minimal. This allows us to compare the energy of elements in a dual basis on $S$ with the energy of elements of the relaxed basis on the limit surface. The harmonic forms of this basis have  some nice properties that are interesting by themselves. In \bbref{Section}{sec:TwoCusps} we use them to prove \bbref{Theorems}{thm:nonsepGIntro} and \bbref{}{thm:nonsepFIntro}.

\subsection{The forms $\pmb{\tau_j}$ for $j=2, \dots, 2g.$}
\label{sec:sigmajQ}

Let again $S$ be a compact Riemann surface of genus $g\geq 2$, $\gamma$ a  nonseparating simple closed geodesic on $S$, $\alpha_1, \dots, \alpha_{2g}$ a canonical homology basis with $\alpha_2 = \gamma$ enumerated such that $\alpha_{2j-1}$ intersects $\alpha_{2j}$, $j=1,\dots,g$, and $\sigma_1, \dots, \sigma_{2g}$  the dual basis of harmonic 1-forms. As earlier we extend these notions to any surface in the family $\big(S_L)_{L>0}$, write, however, most of the time $\sigma_i$ instead of $\sigma_{i,L}$, $i=1,\dots, 2g$, and similarly for $\mT_1$ and the forms $\tau_j$ introduced next.

We introduce the modified basis $\sigma_1, \tau_2, \tau_3, \dots, \tau_{2g}$ on $S_L$ by the following stipulations, for $j = 2, \dots, 2g$:
\begin{enumerate}
\item[(a)] $\int\limits_{\alpha_i} \tau_j = \delta_{ij}$, \quad $i = 2, \dots, 2g$;
\item[(b)] $\tau_j$ is the harmonic form with the minimal energy that satisfies (a).
\end{enumerate}
We call $(\sigma_1, \tau_2, \dots, \tau_{2g})$ the \emph{relaxed dual basis}. (An alternative is to view $(\mT_1, \tau_2, \dots, \tau_{2g})$ as the relaxed dual basis despite the fact that  $\mT_1$ does not have period 1.)

It is easy to see that $\tau_j$ is also an energy minimizer among \emph{all} closed 1-forms that satisfy (a). We introduce the periods
\begin{equation}\label{eq:kappaj}
\kappa_j=\kappa_{j,L} = -\int\limits_{\alpha_1} \tau_j, \quad j= 2,\dots, 2g.
\end{equation}
Again, if it is clear from the context we will use $\kappa_i$ instead of $\kappa_{j,L}$, $j=2,\dots, 2g$. By (a) and because $\sigma_1, \dots, \sigma_{2g}$ is a dual basis we then have
\begin{equation}\label{eq:sigmakapj}
\sigma_j = \tau_j + \kappa_j \sigma_1, \quad j= 2,\dots, 2g.
\end{equation}

 \begin{lem}\label{lem:sigma1ortho}%
For $j=2,\dots,2g$ we have 
 \begin{equation*}%
\scp{\tau_j}{\sigma_1} = 0.
 \end{equation*}%
%
 \end{lem}%
 \begin{proof}%
If $\omega$ is a harmonic 1-form on $S$ that, for given $j$, satisfies the condition $\int_{\alpha_i} \omega = \delta_{ij}$, for $i = 2, \dots, 2g$, then $\omega = \sigma_j + t \sigma_1$, for some $t \in \R$. The function $f(t) = E(\sigma_j + t \sigma_1) = E(\sigma_j) + 2t \scp{\sigma_j}{\sigma_1} + t^2 E(\sigma_1)$ has a unique minimum at $t_0 = -\scp{\sigma_j}{\sigma_1}/E(\sigma_1)$. Hence,
 \begin{equation*}%
\tau_j =\sigma_j + t_0 \sigma_1 = \sigma_j -\frac{\scp{\sigma_j}{\sigma_1}}{E(\sigma_1)}\sigma_1,
 \end{equation*}%
and the lemma follows. 
 \end{proof}%
Taking in \beqref{eq:sigmakapj} the scalar product with $\sigma_1$ we now get

 \begin{equation}\label{eq:Fkapp1}%
 \kappa_j = \frac{\scp{\sigma_1}{\sigma_j}}{\scp{\sigma_1}{\sigma_1}}, \quad j= 2,\dots, 2g.\end{equation}%

For the next lemma we consider an annular neighbourhood $\mathscr{C}$ of $\gamma$ ($\mathscr{C}$ need not be a distance neighbourhood) and a conformal mapping $\psi$ that sends $\mathscr{C}$ onto a flat cylinder $Z = ({-}m_1,m_2) \times \Sp_1$ for some $m_1,m_2 > 0$. For the cylinder we use the notation as in \bbref{Section}{sec:HarmCy}. In particular, $x$ is the variable that runs through $[-m_1,m_2]$ and $y$ runs through $\Sp_1 = \R/ \Z$. The mapping $\psi$ is assumed to be such that the image of $\gamma$ is the circle $\{0\} \times \Sp_1$. The couple $(Z, \psi)$ serves as a coordinate system for $\mathscr{C}$, and by abuse of notation we write $f = f\circ \psi^{-1}$ for the representation of functions on $\mathscr{C}$, and $\omega = (\psi^{-1})^{*} \omega$ for 1-forms on $\mathscr{C}$.
 \begin{lem}\label{lem:nolinear}%
On $\mathscr{C}$ the $\tau_j$ have vanishing linear $d x$-terms, $j=2, \dots, 2g$.
 \end{lem}%
 \begin{proof}%
Without loss of generality we assume that $m_1=m_2 = m$. We first consider the case $j \geq 3$. Then $\tau_j$ on $\mathscr{C}$ has a representation
 \begin{equation}\label{eq:EsjQ1}%
\tau_j = a d x + d h^{\nlin},
 \end{equation}%
where $d h^{\nlin}$ is the nonlinear part with a harmonic function $h^{\nlin}$. We have to show that $a=0$. We shall achieve this by dampening down, using the function
 \begin{equation*}%
F(x,y) =
\begin{cases}
0, \quad & \text{for } x \leq 0\\
a x \quad & \text{for } 0 \leq x \leq m/2\\
a m/2 \quad & \text{for }  x \geq m/2.\\
\end{cases}
 \end{equation*}%
Interpreting $F$ as a function on $\mathscr{C}$ we define the closed form $\omega_F = dF$ on $\mathscr{C}$ and $\omega_F = 0$ on $S\smallsetminus \mathscr{C}$. The difference $t_j = \tau_j - \omega_F$ has periods $\int_{\alpha_i} t_j = \delta_{ij}$, for $i=2,\dots, 2g$. By the energy minimizing property of $\tau_j$ we have 
 \begin{equation}\label{eq:EsjQ2}%
E(t_j) \geq E(\tau_j).
 \end{equation}%
Outside of the part $D = [0, \frac{m}{2}] \times \Sp_1$ the forms $\tau_j$ and $t_j$ coincide. On $D$ we may write
 \begin{equation*}%
\tau_j = t_j + dF = t_j + a d x,
 \end{equation*}%
where by \beqref{eq:EsjQ1}  $t_j = d h^{\nlin}$.   By the Orthogonality lemma, 
 \begin{equation*}%
E_D(\tau_j) = E_D(t_j) + a^2 E_D(d x).
 \end{equation*}%
Since $\tau_j$ and $t_j$ coincide outside of $D$ we also have
 \begin{equation}\label{eq:EsjQ3}%
E(\tau_j) = E(t_j) + a^2 E_D(d x).
 \end{equation}%
Now \beqref{eq:EsjQ2} and \beqref{eq:EsjQ3} imply $a =0$. The proof for the case $j=2$ is the same except that instead of \beqref{eq:EsjQ1} we have $\tau_2 = a d x +  d y + dh^{\nlin}$, and on $D$ we have $t_2 = dh^{\nlin} +  d y$, where we recall that $d x$ and $d y$ are orthogonal on $D$.
 \end{proof}%
 \begin{rem}\label{rem:T1t2star}%
As a consequence of \bbref{Lemma}{lem:nolinear} $\mT_1$ can be written as a linear combination
 \begin{equation*}%
\mT_1 = \Hstar \tau_2 + \sum_{k=3}^{2g} c_k \tau_k.
 \end{equation*}%
For the proof we note that, by the lemma, $\mT_1 - \Hstar \tau_2$ and $\tau_3, \dots, \tau_{2g}$ has vanishing linear parts, while $\sigma_1$ has linear part $c d x$ with $c \neq 0$ and $\tau_2$ has linear part $d y$. Hence, in the basis representation  $\mT_1 - \Hstar \tau_2 = c_1 \sigma_1 + c_2 \tau_2 + \sum_{k=3}^{2g}c_k \tau_{k}$ we must have $c_1 = c_2 = 0$.
 \end{rem}%

\subsection{Essential energy and the periods $\kappa_j$}\label{sec:kappaj}

In this section we investigate the convergence of the periods $\kappa_j$ (see \beqref{eq:kappaj})
as $\ell(\gamma) \to 0$ and also under the grafting construction. In order to avoid repetitions we prove the main technical lemma (\bbref{Lemma}{lem:capabreve}) in a slightly more general context. Furthermore, in order to remain in the realm of compact surfaces we carry out the comparison with compact surfaces that have a thin handle and cover the cusp case by a limiting argument in \bbref{Section}{sec:TwoCusps}.

The setting for the lemma is as follows. $S_L^{\times}$ is the surface as in the preceding section consisting of the main part $\frak{M}$ \beqref{eq:kernelK} to which two flat cylinders $[0,L/2] \times \Sp_1$ are attached, and $S_L$ results from $S_L^{\times}$ by pasting together the two boundaries. These surfaces shall be compared with a second pair, $\tilde{S}^{\times}_{\tilde{L}}$, $\tilde{S}_{\tilde{L}}$, defined in a similar way with some main  part $\tilde{\frak{M}}$ to which cylinders of lengths $\tilde{L}/2$ are attached. Furthermore, it is assumed that
%
 \begin{equation}\label{eq:cond4L}%
\tilde{L} \geq L > 4
 \end{equation}%
and that a marking preserving $q_{\phi}$-quasiconformal boundary coherent homeomorphism $\phi : \frak{M} \to \tilde{\frak{M}}$ is given that acts conformally near the boundary and thus extends to an embedding of $S_L^{\times}$ into $\tilde{S}_L^{\times}$ that acts isometrically on the flat cylinders. Finally, we assume that the pastings are such that the twist parameters of $S_L$ and $\tilde{S}_{\tilde{L}}$ are the same. By this we mean that, as in the grafting construction,  if line $[0,L] \times \{y_0\}$ of the inserted cylinder $[0,L] \times \Sp_1$ in $S_L$ is pasted to the boundary points $p_1$, $p_2$ of $\frak{M}$, then line $[0,\tilde{L}] \times \{y_0\}$ of the inserted cylinder $[0,\tilde{L}] \times \Sp_1$ in $\tilde{S}_{\tilde{L}}$ is pasted to the boundary points $\phi(p_1)$, $\phi(p_2)$  of $\tilde{\frak{M}} = \phi(\frak{M})$.

For the grafting construction we shall have $\tilde{\frak{M}} = \frak{M}$ with $\phi$ the identity mapping; in the Fenchel-Nielsen construction $\tilde{\frak{M}}$ is the main part of the surface $\tilde{S}$ with geodesic $\tilde{\gamma}$ that results from the construction and $\phi$ is the $q$-quasiconformal mapping as in \bbref{Theorem}{thm:psRGRF} with $q \leq (1+2\ell^2(\gamma))^2$, and $\tilde{L} = L_{\tilde{\gamma}}$. 

As for the bases, the quasiconformal mapping $\phi : \frak{M} \to \tilde{\frak{M}}$ extends to a topological mapping $\varphi : S_L \to \tilde{S}_{\tilde{L}}$ that operates on the cylinders by stretching them in the $x$-direction. We define $\tilde{\alpha}_i = \varphi(\alpha_i)$, for $i=1,\dots,2g$, and let $\tilde{\sigma}_1, \dots, \tilde{\sigma}_{2g}$ be the dual basis of harmonic forms for $\tilde{\alpha}_1, \dots, \tilde{\alpha}_{2g}$ on $\tilde{S}_{\tilde{L}}$. Finally, for $j=2,\dots, 2g$, $\tilde{\tau}_j$ on $\tilde{S}_{\tilde{L}}$ is the analog of $\tau_j$ on $S_L$. 

As $L$ becomes large, the energies of $\tau_3, \dots, \tau_{2g}$ remain bounded owing to the fact that their linear parts vanish in the cylinder $[0,L] \times \Sp_1$ (using \bbref{Lemma}{lem:nolinear}). This is not the case for $\tau_2$. For this reason we define for $\tau_2$ and, more generally, for any linear combination $\tau = a_2 \tau_2 + \dots + a_{2g}\tau_{2g}$ an \emph{essential energy} $\Ess(\tau)$ as in \beqref{eq:Essigma2} setting
 \begin{equation}\label{eq:EssEng}%
\Ess(\tau) = E_\frak{M}(\tau) + E_C(\tau{\nlin}),
 \end{equation}%
where $C$ is the cylinder $C = S_L \smallsetminus \frak{M}$ which we identify with the conformally equivalent $[0,L] \times \Sp_1$ and ${..}^{\nlin}$ denotes the nonlinear part. For later use we also define the scalar product 
\begin{equation}\label{eq:EssScalPeod}%
\scp{\omega}{\eta}_{\frak{M}}  \defeq \int_{\frak{M}}\omega \wedge \Hstar \eta + \int_{C(\gamma)} \omega^{\nlin} \wedge \Hstar \eta^{\nlin}.
 \end{equation}%
for $\omega$ and $\eta$ in the span of $(\tau_i)_{i=2,\ldots,2g}$.

By the Orthogonality lemma $\Ess(\tau)$ may also be defined as follows,
 \begin{equation}\label{eq:EssEnerg}%
\Ess(\tau) = E(\tau) - a_2^2 L.
 \end{equation}%
On $\tilde{S}_{\tilde{L}}$ we define $\Ess$ in an analog way.  For the comparison of these energies and also for the comparison of the periods $\kappa_j$ of $\tau_j$ over $\alpha_1$ and $\tilde{\kappa}_j$ of $\tilde{\tau}_j$ over $\tilde{\alpha_j}$, $j=2, \dots, 2g$ we introduce the following constants,
 \begin{equation}\label{eq:muL}%
\mu_L = e^{-2 \pi L}, \quad \nu_L = \frac{1}{12}\mu_L q_{\phi} + q_{\phi}-1, \quad \rho_L = \left[q_{\phi}(1+\frac{1}{12}\mu_L)\right]^2-1.
 \end{equation}%

 \begin{lem}\label{lem:capabreve}%
Under the condition of the above setting we have, for any $\tau = a_2 \tau_2 + \dots + a_{2g}\tau_{2g}$ on $S_L$ and the corresponding $\tilde{\tau} = a_2 \tilde{\tau}_2 + \dots + a_{2g}\tilde{\tau}_{2g}$ on $\tilde{S}_{\tilde{L}}$,
 \begin{equation}\label{eq:capabreve1}%
 \vert    \Ess(\tau) - \Ess(\tilde{\tau})\vert  < \nu_L\min\{ \Ess(\tau),  \Ess(\tilde{\tau})\}.\end{equation}%
For $\omega,\eta$ in the span of  $(\tau_i)_{i=2,\ldots,2g}$ and $\tilde{\omega},\tilde{\eta}$ in the span of  $(\tilde{\tau_i})_{i=2,\ldots,2g}$ we have
 \begin{equation}\label{eq:capabreve15}%
\big\vert\scp{\omega}{\eta}_{\frak{M}} - \scp{\tilde{\omega}}{\tilde{\eta}}_{\frak{M}} \big\vert \leq \nu_L \min\{ \Ess^{1/2}(\omega)\Ess^{1/2}(\eta), \Ess^{1/2}(\tilde{\omega})\Ess^{1/2}(\tilde{\eta}) \} .
 \end{equation}%

Furthermore, for  $j=2, \dots, 2g$, 
 \begin{equation}\label{eq:capabreve2}%
\vert \kappa_j - \tilde{\kappa}_j \vert^2 <  3 \rho_L(1+ \mathfrak{z}\Gamma) \min\{ \Ess(\tau_j), \Ess(\tilde{\tau}_j)\}.
 \end{equation}%
 \end{lem}%
 \begin{proof}%
By \bbref{Lemma}{lem:nolinear}, $\tau$ and $\tilde{\tau}$ have vanishing linear $d x$ terms. Furthermore, the $d y$-terms of $\tau_j$ and $\tilde{\tau}_j$ are $\delta_{2j} d y$ because $\int_{\alpha_2}\tau_j = \int_{\tilde{\alpha}_2}\tilde{\tau_j}=\delta_{2j}$, $j=2,\dots,2g$. Hence, $\tau$ and $\tilde{\tau}$ have the same linear term $a_2 d y$ on the flat cylinders. Taking $w = \frac{1}{8}$  we dampen down the nonlinear terms of $\tau$ and $\tilde{\tau}$ using the same cut off function $\chi$ that equals 1 up to distance $L/2-w$ from $\frak{M}$, respectively $\tilde{\frak{M}}$, then goes down linearly to 0 from distance $L/2-w$ to distance $L/2$ and remains 0 on the rest (distances with respect to the flat metric of $Z = [-L/2,L/2] \times \Sp_1$ and $\tilde{Z} = [-\tilde{L}/2,\tilde{L}/2] \times \Sp_1$).

In the following we think of $S_L$ and $S_L^{\times}$ as ``sitting on the left hand side,'' while $\tilde{S}_{\tilde{L}}$ and $\tilde{S}^{\times}_{\tilde{L}}$ ``sit on the right hand side''. The mapping $\phi$ then goes from the left to the right.

The dampened down form $\tilde{\tau}_{\chi}$ may  be restricted to the subset $\tilde{S}^{\times}_L$ of $\tilde{S}^{\times}_{\tilde{L}}$,  then be pushed in the left direction to $S_L^{\times}$ via the quasiconformal mapping $\phi^{-1}$ and, upon arrival, be extended to $S_L$. The so extended form shall be denoted by $\tilde{\tau}^{\scriptscriptstyle\leftarrow}$. As a result of the dampening down procedure $\tilde{\tau}^{\scriptscriptstyle\leftarrow}$ is a test form in the cohomology class of $\tau$. 
In a converse way, we push in the right direction the dampened down form  $\tau_{\chi}$ to $\tilde{S}^{\times}_L$ and extend it by $a_2 d y$ to a test form $\tau^{\scriptscriptstyle\rightarrow}$ on $\tilde{S}_{\tilde{L}}$ in the cohomology class of $\tilde{\tau}$. It follows, using the \bbref{Orthogonality lemma}{thm:lemorthogonal} that
 \begin{equation}\label{eq:capabreve3}%
\Ess(\tau) \leq \Ess(\tilde{\tau}^{\scriptscriptstyle\leftarrow}), \quad
\Ess(\tilde{\tau}) \leq \Ess(\tau^{\scriptscriptstyle\rightarrow}).
 \end{equation}%
Observe that the dampening down process takes place in a subset of the respective cylinders $[-d_{\gamma}-L/2,L/2 +d_{\gamma}]\times \Sp_1$ and $[-d_{\tilde{\gamma}}-\tilde{L}/2,\tilde{L}/2 +d_{\tilde{\gamma}}]\times \Sp_1$, where by \beqref{eq:lengcyl} $d_{\gamma}, d_{\tilde{\gamma}} > \frac{1}{2}$. Hence, \bbref{Lemma}{lem:dampening} (with $l = L/2 + \frac{1}{2}$, $dh^{\texttt{\#}}$ the nonlinear term of $\tilde{\tau}$, respectively $\tilde{\tau}$, $\delta_{\texttt{\#}} = l-w$, and $w = \frac{1}{8}$) yields that $\Ess(\tau_{\chi}) < (1+\tfrac{1}{12}\mu_L)\Ess(\tau)$ and $ \Ess(\tilde{\tau}_{\chi}) < (1+\tfrac{1}{12}\mu_L)\Ess(\tilde{\tau})$. Since $\phi$ and $\phi^{-1}$ act isometrically on the cylinders and are $q_{\phi}$-quasiconformal on $\frak{M}$ and $\tilde{\frak{M}}$, it follows from \cite{Mi74} (or \bbref{Lemma}{L:qcIneq}) that the essential energy increases at most by a factor $q_{\phi}$. This yields
 \begin{equation}\label{eq:capabreve4}%
\Ess(\tilde{\tau}^{\scriptscriptstyle\leftarrow}) \leq q_{\phi}(1+\tfrac{1}{12}\mu_L)\Ess(\tilde{\tau}), 
\quad
\Ess(\tau^{\scriptscriptstyle\rightarrow}) < q_{\phi}(1+\tfrac{1}{12}\mu_L)\Ess(\tau).
 \end{equation}%
Now \beqref{eq:capabreve3} and \beqref{eq:capabreve4} imply \beqref{eq:capabreve1}.\\
The operator that associates to any $\omega = c_2 \tau_{2} + \dots + c_{2g} \tau_{2g} \in H^1(S_L,\R)$ the form $\tilde{\omega} = c_2 \tilde{\tau}_2 + \dots + c_{2g} \tilde{\tau}_{2g} \in  H^1(\tilde{S}_{\tilde{L}},\R)$ is linear. Hence, using \beqref{eq:capabreve1} in \bbref{Lemma}{L:LinAlgIneq} inequality \beqref{eq:capabreve15} follows.

We turn to the proof of \beqref{eq:capabreve2} and apply the preceding notations and results to the forms $\tau = \tau_j$, $\tilde{\tau} = \tilde{\tau}_j$. Since the dampening down process on $S_L$ and $\tilde{S}_{\tilde{L}}$ acts as the identity in the neighbourhoods of the boundaries of the collars, and since $\phi$ preserves line integrals, the period $\tilde{\kappa}_j$ is not affected  and, hence,

 \begin{equation*}%
\int_{\alpha_1} (\tilde{\tau}_j^{\scriptscriptstyle\leftarrow} - \tau_j) = -\tilde{\kappa}_j + \kappa_j.
 \end{equation*}%
Since the periods of $\tilde{\tau}_j^{\scriptscriptstyle\leftarrow} - \tau_j$ over the remaining base cycles vanish we have therefore on $S_L$

 \begin{equation*}%
\tilde{\tau}_j^{\scriptscriptstyle\leftarrow} = \tau_j - (\tilde{\kappa}_j - \kappa_j)\sigma_{1,L} + d f_j
 \end{equation*}%
for some function $f_j$. By \bbref{Lemma}{lem:sigma1ortho} the forms on the right hand side are mutually orthogonal so that
 \begin{equation*}%
E(\tilde{\tau}_j^{\scriptscriptstyle\leftarrow}) = E(\tau_j) + (\tilde{\kappa}_j - \kappa_j)^2 E(\sigma_{1,L}) + E(d f_j).
 \end{equation*}%
Since $E(\tilde{\tau}_j^{\scriptscriptstyle\leftarrow}- \tau_j) = (\tilde{\kappa}_j - \kappa_j)^2 E(\sigma_{1,L}) + E(d f_j)$ and since $\tilde{\tau}_j^{\scriptscriptstyle\leftarrow}$ and $\tau_j$ have the same linear part ($d y$ if $j = 2$ and 0 if $j \geq 3$) we get
 \begin{equation*}%
E(\tilde{\tau}_j^{\scriptscriptstyle\leftarrow}- \tau_j) = E(\tilde{\tau}_j^{\scriptscriptstyle\leftarrow})-E(\tau_j) =\Ess(\tilde{\tau}_j^{\scriptscriptstyle\leftarrow})-\Ess(\tau_j).
 \end{equation*}%
Now if $\Ess(\tilde{\tau}_j) \leq \Ess(\tau_j)$, then by the first inequality in \beqref{eq:capabreve4}
 \begin{equation*}%
\Ess(\tilde{\tau}_j^{\scriptscriptstyle\leftarrow})- \Ess(\tau_j) \leq \Ess(\tilde{\tau}_j^{\scriptscriptstyle\leftarrow}) - \Ess(\tilde{\tau}_j) \leq (q_{\phi}(1+\tfrac{1}{12}\mu_L)-1)\Ess(\tilde{\tau}_j).
 \end{equation*}%
If, on the other hand, $\Ess(\tilde{\tau}_j) > \Ess(\tau_j)$, then by the first inequality in \beqref{eq:capabreve4}, the second inequality in \beqref{eq:capabreve3} and the second inequality in \beqref{eq:capabreve4}
 \begin{equation*}%
\Ess(\tilde{\tau}_j^{\scriptscriptstyle\leftarrow})- \Ess(\tau_j) \leq q_{\phi}(1+\tfrac{1}{12}\mu_L)\Ess(\tilde{\tau}_j) - \Ess(\tau_j)\leq \big(\left[q_{\phi}(1+\tfrac{1}{12}\mu_L)\right]^2-1\big)\Ess(\tau_j).
 \end{equation*}%
Altogether
 \begin{equation*}%
E(\tilde{\tau}_j^{\scriptscriptstyle\leftarrow}- \tau_j)\leq \rho_L \min\{ \Ess(\tau_j), \Ess(\tilde{\tau}_j)\}.
 \end{equation*}%
Now $\tilde{\tau}_j^{\scriptscriptstyle\leftarrow}$ and $\tau_j$ have vanishing linear $d x$-part. We dampen down $(\tilde{\tau}_j^{\scriptscriptstyle\leftarrow}- \tau_j)$ using the cut off function $\chi'$ that goes linearly from $1$ to $0$ within distance $\frac{1}{2}$ from $\frak{M}$ and then remains zero. The resulting form is denoted by $\theta_j$. Its period over $\alpha_1$ is still $-(\tilde{\kappa}_j - \kappa_j)$. By \bbref{Lemma}{lem:dampening} applied twice (with $l = \frac{1}{4}$, $\delta_{\texttt{\#}}=\delta = 0$ and $w=\frac{1}{2}$) the energy of $\theta_j$ has the upper bound
 \begin{equation*}%
E(\theta_j) \leq 3 E(\tilde{\tau}_j^{\scriptscriptstyle\leftarrow}- \tau_j) \leq 3 \rho_L \min\{ \Ess(\tau_j), \Ess(\tilde{\tau}_j)\}.
 \end{equation*}%
That \bbref{Lemma}{lem:dampening} is indeed applicable follows from the fact that $\phi$ acts isometrically on the cylinders, whose lengths, by \beqref{eq:cond4L}, are $\geq L/2 > 2$ and, hence, $\theta_j$ is harmonic in the domain where the dampening down process takes place. Now $\theta_j$ may also be interpreted as a test form on the surface $S_1$ that is obtained by attaching cylinders $[0,\frac{1}{2}] \times \Sp_1$ to $\frak{M}$ and then pasting the boundaries together with the same twist parameter as $S_L$. This means that if points $p \in c_1$, $q \in c_2$ on the boundary of $\frak{M}$ are connected by the straight line  $[0,L] \times \{y_0\}$ in $[0,L] \times \Sp_1 \subset S_L$ then they are connected by the straight line $[0,1] \times \{y_0\}$ in $[0,1] \times \Sp_1 \subset S_1$. Only this time the cutting and pasting takes place at $L/2 \times \Sp_1$ and $\frac{1}{2} \times \Sp_1$, respectively. On $S_1$ we have the following decomposition (where $\sigma_{1,1}$ is $\sigma_{1,L}$ with $L=1$): $\theta_j = -(\tilde{\kappa}_j - \kappa_j)\sigma_{1,1} + dF_j$ for some function $F_j$. By \bbref{Lemma}{lem:capStildaL2}, applied to the case $L=1$, we have $E(\sigma_{1,1}) \geq \frac{1}{1 + \frak{z} \Gamma}$. Altogether we have
 \begin{equation*}%
 \vert \kappa_j - \tilde{\kappa}_j \vert^2 \frac{1}{1 + \frak{z} \Gamma} \leq 3 \rho_L \min\{ \Ess(\tau_j), \Ess(\tilde{\tau}_j)\}
 \end{equation*}%
and \beqref{eq:capabreve2} follows.
 \end{proof}

\subsection{Bounds for $\pmb{\kappa_j}$.}\label{sec:Bkappa}
 \begin{lem}\label{L.Bkappa}%
For $j= 3,\dots,2g$ we have on $S_L$\
 \begin{equation*}%
\vert \kappa_j \vert \leq \sqrt{\Ess(\tau_2)\, E(\tau_j)} < \sqrt{\Ess(\tau_2)\, E(\sigma_j)}.
 \end{equation*}%
 \end{lem}%
 \begin{proof}%
We use the following identity for closed differentials $\theta$, $\eta$,
 \begin{equation}\label{eq:FK92III23}%
\int_{S_L} \theta \wedge  \eta = \sum_{k=1}^g \Big[
\int_{\alpha_{2k}}\!\!\theta \int_{\alpha_{2k-1}}\!\!\eta\;\;   -  \int_{\alpha_{2k-1}}\!\!\theta \int_{\alpha_{2k}}\!\!\eta\, \Big]
 \end{equation}%
(e.g.\ \cite[III.2.3]{FK92}). Taking $\theta = \tau_2$, $\eta = \tau_j$, ($j \geq 3$), we get $\int_{S_L} \tau_2 \wedge \tau_j = \int_{\alpha_1}\tau_j = -\kappa_j$. Hence,
 \begin{equation}\label{eq:Bkapj2}%
\kappa_j = \scp{ \tau_2}{\star\tau_j}, \quad j=3, \dots, 2g.
 \end{equation}%
By \bbref{Lemma}{lem:nolinear} $\tau_j$ has vanishing linear part in the cylinder $C = S_L \smallsetminus \frak{M}$ (which we identify with the conformally equivalent $[0,L] \times \Sp_1$). Using the Orthogonality lemma we may thus write
 \begin{equation*}%
\int_C \tau_2 \wedge \tau_j = \int_C \tau_2^{\nlin} \wedge \tau_j,
\end{equation*}%
where $(\phantom{.})^{\nlin}$ is the non linear part. Setting
 \begin{equation*}%
\hat{\tau_2} =
\begin{cases}
\tau_2^{\nlin} & \text{on $C$}\\
\tau_2 & \text{on $S_L \smallsetminus C$}
\end{cases}
 \end{equation*}%
we get
 \begin{equation*}
\int_{S_L} \tau_2 \wedge \tau_j = \int_{S_L} \hat{\tau_2} \wedge \tau_j.
 \end{equation*}%
(The fact that $\hat{\tau}_2$ is discontinuous at the boundary of $C$ is not a problem here because we are only looking at integrals.) We now use the Cauchy-Schwarz inequality:
 \begin{equation*}
\vert \int_{S_L} \hat{\tau}_2 \wedge \tau_j \vert^2 \leq E(\hat{\tau}_2) E (\tau_j). 
 \end{equation*}%
On the right hand side $E(\hat{\tau_2}) = \Ess(\tau_2)$. This yields%
 \begin{equation*}%
\vert \scp{\tau_2}{\star \tau_j} \vert \leq \sqrt{\Ess(\tau_2)\, E(\tau_j)}.
 \end{equation*}%
The first inequality of the lemma follows together with \beqref{eq:Bkapj2}; the second inequality follows from \beqref{eq:sigmakapj} and \bbref{Lemma}{lem:sigma1ortho} (which also holds on any $S_L$): $E(\sigma_j) = E(\tau_j + \kappa_j \sigma_1) = E(\tau_j) + \kappa_j^2 E(\sigma_1)$.
\end{proof}%
 \begin{rem}\label{R.Bkapj1}%
For $j = 2$ relation \beqref{eq:FK92III23} is not conclusive. Indeed, by replacing $\alpha_1$ with $\alpha_1 + k \alpha_2$, $k \in \Z$, one may arbitrarily increase $\vert \kappa_2\vert$ without affecting $\Ess(\tau_2), E(\tau_j), E(\sigma_j)$, $j=3,\dots, 2g$. Hence, there is no a priori bound for $\vert \kappa_2\vert$ in terms of these quantities.

However, if in \beqref{eq:FK92III23} we take $\theta = \star \mT_1$ and $\eta = \tau_2$, we get using \bbref{Lemma}{lem:sigma1ortho} and observing that $\int_{\alpha_2}\star \mT_1 = 1$ (because $\mT_1$ has linear part $d x$ in $C$)
 \begin{equation}\label{eq:relkappa2}%
\kappa_2 = \int_{\alpha_1}\star \mT_1.
 \end{equation}%
The relation shows that if in the choice of the canonical homology basis on $S$ one begins with $\alpha_2\, (= \gamma)$ and $\alpha_1$, then $\kappa_2$ is independent of the choice of $\alpha_3, \dots, \alpha_{2g}$.

A rough estimate of $\kappa_2$ in a different form shall be given in \bbref{Lemma}{lem:Bdkappa2}.
\end{rem}%
\subsection{Bounds for the dual basis.}\label{sec:Bdualb}
The estimates in \bbref{Lemmata}{lem:capabreve} and \bbref{}{L.Bkappa} are in terms of the harmonic quantities $\Ess(\mT_1)$, $\Ess(\tau_2)$ and $E(\sigma_3), \dots, E(\sigma_{2g})$. In \bbref{Section}{sec:sigT1} a geometric upper bound for $\Ess(\mT_1)$ is obtainable via  \bbref{Lemma}{L:sigma1ess} and \bbref{Theorem}{eq.bdscapaM}. We conclude this section with geometric bounds for the remaining quantities.

The first lemma provides bounds for the energies of $\tau_k$ on $S_L$ and for the matrix entries $p_{kk}(S_L)$, $k \geq 3$, in terms of the lengths of the geodesics $\alpha_3, \ldots,\alpha_{2g-1},\alpha_{2g}$ of the canonical homology basis on the original surface $S$. In view of the intersection convention for this basis (\bbref{Section}{sec:Introd}) we set, for $k=1,\dots,2g$
 \begin{equation*}%
\hat{k} =
\begin{cases}
k-1& \text{if $k$ is even,}\\
k+1& \text{if $k$ is odd.}
\end{cases}
 \end{equation*}%
In the following we write $\tau_{k,L}$, $\kappa_{k,L}$ for $\tau_k$ and $\kappa_k$ on $S_L$, $k=2, \dots, 2g$.
 \begin{lem}\label{lem.Bdpkk}%
For $k = 3, \dots, 2g$ we have 
\begin{equation*}%
E(\tau_{k,L}) =p_{kk}(S_L)-\kappa_{k,L}^2 p_{11}(S_L)
, \quad p_{kk}(S_L) \leq \frac{\ell(\alpha_{\hat{k}})}{\pi-2\arcsin \tanh \frac{1}{2}\ell(\alpha_{\hat{k}})}.
\end{equation*}%
\end{lem}%
 \begin{proof}%
The expression for $E(\tau_{k,L}) = \scp{\tau_{k,L}}{\tau_{k,L}}$ follows from \beqref{eq:sigmakapj} and \bbref{Lemma}{lem:sigma1ortho}. To prove the inequality we construct a test form for $\sigma_k$ on $S_L$ using the hyperbolic metric on $\frak{M}$. The standard collar $C(\alpha_{\hat{k}})$ on $S$ is contained in $\frak{M}$, by the Collar lemma. It may therefore also be seen as a subset of $S_L$. There is a conformal homeomorphism $\psi$ from the closure of $C(\alpha_{\hat{k}})$ to $Z = [-M,M] \times \Sp_1$ with $M$ as in \beqref{eq:Lleta}. Now let $f$ on $Z$ be defined by  $f(x,y) = \frac{1}{2M}(x+M)$, for $x\in [-M,M]$, $y \in \Sp_1$. Then $d f$ and its pullback $\eta = \psi^{*}d f$ on $C(\alpha_{\hat{k}})$ have energy $\frac{1}{2M}$. Setting $\psi = 0$ outside $C(\alpha_{\hat{k}})$ we get an admissible test form for $\sigma_k$, and the upper bound follows.
 \end{proof}%
We add that for $\ell(\alpha_{\hat{k}}) \to \infty$ the upper bound grows asymptotically 
like $\frac{1}{4} \ell(\alpha_{\hat{k}}) \exp(\frac{1}{2}\ell(\alpha_{\hat{k}}))$.

\medskip

In the second lemma we prove a geometric upper bound for $\Ess(\tau_{2,L})$. It is more elaborate since it involves the entire homology basis ${\rm A} = (\alpha_1, \dots, \alpha_{2g})$. The geometric quantity we use is
 \begin{equation}\label{eq:widthA}%
w_{\rm A} = \min\Big\{\frac{1}{4}, \frac{1}{4}\sys(\frak{M}), \cl(\ell(\alpha_3)), \dots, \cl(\ell(\alpha_{2g}))\Big\}
 \end{equation}%
with the systole of $\frak{M}$ as in \bbref{Section}{sec:bcapM} and the widths $\cl$ of the standard collars  as \beqref{eq:clfctn}.
 \begin{lem}\label{lem:BdEsstau2}%
For any $L > 0$, we have 
 \begin{equation*}%
\Ess(\tau_{2,L}) =E(\sigma_{2,L}) - L -\kappa_{2,L}^2 p_11(S_L), \quad
E(\sigma_{2,L}) < L+\frac{\pi\pt g}{w_{\rm A}^2}.
 \end{equation*}%
\end{lem}%
 \begin{proof}%
The first statement is as in \bbref{Lemma}{lem.Bdpkk}. For the second statement we construct a test form $s_{2,L}$ for $\sigma_{2,L}$ on $S_L$ whose representation in the cylinder $Z_L = \{(x,y) \mid x \in [0,L], y \in \Sp_1 \}$ attached to $\frak{M}$ along $c_1$, $c_2$ coincides with $d y$.

For this we take the shortest geodesic $\mathcal{A}$ in $\frak{M}' = \frak{M} \smallsetminus (\alpha_3 \cup \dots \cup \alpha_{2g})$ that connects the boundary curves $c_1$, $c_2$ with each other. Such a geodesic exists because the boundary of $\frak{M}'$ has interior angles $<\pi$ at all of its vertices; if there are several such geodesics we take any of them. By the Collar theorem $\mathcal{A}$ does not intersect any of the collars $C(\alpha_k)$ and thus has distance $\dist(\mathcal{A},\alpha_k) \geq w_{\rm A}$, for $k=3, \dots, 2g$. As in the proof of \bbref{Theorem}{sec:bcapM} the geodesic segments of length $w_{\rm A}$ emanating orthogonally from $\mathcal{A}$ (in both directions) are pairwise disjoint and sweep out a ribbon $\mathcal{W}$. In Fermi coordinates based on $\mathcal{A}$ the ribbon becomes 
\begin{equation*}
\{(\rho,t) \mid  -w_{\rm A} \leq \rho \leq w_{\rm A}, 0 \leq t \leq \ell(\mathcal{A}) \},
\end{equation*}
with the metric tensor $d s^2 = d \rho^2 + \cosh^2(\rho)d t^2$. We let $f$ be the function on $\mathcal{W}$ that is defined in these coordinates by $f(\rho,t) = \frac{\rho}{2 w_{\rm A}}$. Its differential has norm $\Vert d f(\rho,t) \Vert = \frac{1}{2 w_{\rm A}}$ and as $\mathcal{W} \subset S$ its energy has the upper bound
 \begin{equation}\label{eq:engribf}%
E_{\mathcal{W}}(df) = \frac{1}{4 w_{\rm A}^2}\area(\mathcal{W})< \frac{1}{4 w_{\rm A}^2}\area(S) =  \frac{1}{4 w_{\rm A}^2}4\pi(g-1).
 \end{equation}%

We now interpolate $f$ in the annular regions $B_1, B_2 \subset \frak{M}$ (adjacent to the boundary of $\frak{M}$, see \bbref{Section}{sec:capS}, prior to \beqref{eq:kernelK}) so that its differential matches with the differential $d y$ on the attached cylinder $Z_L$. We complete the construction of $s_{2,L}$ by setting it to be $d y$ on $Z_L$, $df$ on the domain in $\frak{M}$ where $f$ (together with its interpolation) is defined, and  setting $s_{2,L} = 0$ on the remaining part of $\frak{M}$. The form is in the cohomology class of $\sigma_{2,L}$ and has the energy bound 
 \begin{equation*}
E(s_{2,L}) \leq L + \frac{1}{4 w_{\rm A}^2}4\pi(g-1) + \frak{p},
 \end{equation*}%
where $\frak{p}$ is the energy cost of the interpolation. A rough but somewhat tedious argument which we omit shows that the interpolation can be carried out with $\frak{p} < 2$, and the Lemma follows.
\end{proof}
In the next lemma we write $\kappa_{2,L} = \kappa_2(S_L)$, for clarity.
 \begin{lem}\label{lem:Bdkappa2}%
There exists an effective constant $K$ depending on $g$, on a lower bound of $\sys(\frak{M})$ and on an upper bound of $\ell(\alpha_1)$ on $S$ such that for any $S_L$, $L \geq L_{\gamma}$ one has $\vert \kappa_2(S_L) \vert \leq K$.
 \end{lem}
\begin{proof}
By \beqref{eq:Fkapp1} using Cauchy-Schwarz we have on $S = S_{L_{\gamma}}$
 \begin{equation*}
\vert \kappa_2(S) \vert^2 \leq \frac{\scp{\sigma_2}{\sigma_2}}{\scp{\sigma_1}{\sigma_1}},
 \end{equation*}%
where by \bbref{Lemma}{lem:capStildaL2} $ \frac{1}{\smscp{\sigma_1}{\sigma_1}} \leq L_{\gamma}+ \frak{z} \Gamma$ and by \bbref{Theorem}{eq.bdscapaM} $\Gamma$ is bounded above in terms of the systole of $\frak{M}$ and $g$. Furthermore, by \beqref{eq:capabreve2}%
 \begin{equation*}
\vert \kappa_2(S) - \kappa_2(S_L) \vert^2 <  3 \rho_{L_{\gamma}}(1+ \mathfrak{z}\Gamma) \Ess(\tau_2(S)),
 \end{equation*}%
where $\Ess(\tau_2(S))$ has a bound in terms of $g$ $\sys(\frak{M})$ and $\ell(\alpha_1)$ on $S$ by \bbref{Lemma}{lem:BdEsstau2}. It suffices therefore to prove an upper bound of $E(\sigma_2)$ on $S$. For this we need an embedded ribbon around $\alpha_1$.

Let $w$ be the supremum of all $\delta$ such that geodesic segments of length $\delta$ emanating orthogonally from $\alpha_1$ are pairwise disjoint. Then there exists a segment $\eta$ of length $2w$ on $S$ that meets $\alpha_1$ orthogonally on both endpoints. Lift $\alpha_1$ to the universal covering of $S$. On this lift there is an arc $a$ of length $\ell(\alpha_1)$ having two lifts $\eta'$, $\eta''$ of $\eta$ at its endpoints. On their other endpoints there are further lifts $\alpha_1'$, $\alpha_1''$ of $\alpha_1$. Since $\alpha_1$ is a simple geodesic the latter do not intersect each other. Hence a right angled geodesic hexagon with three consecutive sides of lengths $2w$, $\ell(\alpha_1)$, $2w$. This hexagon consists of two isometric pentagons each having a pair of consecutive sides $2w$, $\frac{1}{2}\ell(\alpha_1)$. By trigonometry $\sinh(2w) \sinh(\frac{1}{2}\ell(\alpha_1)) > 1$. Hence, a lower bound of $w$ in terms of $\ell(\alpha_1)$.

Taking the function $f_2$ on the ribbon of width $w$ that grows linearly along the orthogonal geodesics from 0 on one boundary to 1 on the other and setting $s_2 = df_2$ on the ribbon and $s_2 = 0$ outside we get a test function for $\sigma_2$ on $S$ with the required energy bound.
\end{proof}

\section{Period matrix for a surface with two cusps}\label{sec:TwoCusps}

If we pinch a nonseparating simple closed geodesic then the limit surface has two cusps. A cusp neighbourhood is conformally equivalent to a flat cylinder $Z_1 = ({-}\infty,0]\times \Sp_1$  of infinite length. In this section we first look at the $L^2$ harmonic forms on this surface to which we add the harmonic forms whose linear part is $dx$ and $dy$, respectively in the cylinder and show how these can be seen as the limits of forms in $(S_L)_L$. Then we look at the compacified limit surface that has genus $g-1$ instead of $g$. Two difficulties arise: the rank of the homology and the dimension of the Gram period matrix are not the same as for $S$ and, in addition, the condition $\int_{\alpha_1}\sigma_j = 0$ for the members $\sigma_j$ of the dual basis of harmonic forms disappears.

As announced in the Introduction (\bbref{Section}{sec:Introd}) we overcome the first difficulty by introducing an array of invariants attributed to the limit surface as an ersatz for the missing Gram period matrix. The array may also be translated into a parametrized family of matrices that constitutes a variant of this ersatz (\bbref{Section}{sec:Pmatpar}). With the help of these invariants we then prove the main theorems \bbref{Theorem}{thm:nonsepGIntro} and \bbref{Theorem}{thm:nonsepFIntro} from the Introduction.

\subsection{The space $\Hv{\sR}$}\label{sec:HvsR}
In the grafting construction the limit surface $S^{\sG}$ is represented in the form $S^{\sG} = Z_1 \cup \frak{M} \cup Z_2$, where $\frak{M}$ is the main part of $S$ and $Z_1$, $Z_2$ are the attached infinite cylinders which we normalise here in the form
 \begin{equation*}%
Z_1 = ({-}\infty,0]\times \Sp_1, \quad Z_2 = [0,\infty)\times \Sp_1
 \end{equation*}%
With respect to the complete hyperbolic metric in its conformal class (albeit unknown explicitly) the surface has also a representation \mbox{$S^{\sG} = Z_1 \cup \frak{M}^{\sG} \cup  Z_2$}, where $\frak{M}^{\sG}$ is the main part of $S^{\sG}$ obtained by cutting away the cusps along the horocycles of length 1. 

To cover both representations we consider in this section, more generally, an arbitrary conformal Riemann surface $\sR$ of signature $(g-1,0;2)$, represented in the form 
 \begin{equation}\label{eq:RKZZ}%
\sR = Z_1 \cup \frak{K} \cup Z_2
 \end{equation}%
where $\frak{K}$ is a conformal Riemann surface of signature $(g-1,2;0)$ and the cylinders $Z_1$, $Z_2$, are attached along the two boundaries of $\frak{K}$  with respect to some fixed pasting rules. We denote by $\widebar{2}{\sR}$ the two point compactification of $\sR$ obtained by adding ideal points at infinity $v_1$, $v_2$ to $Z_1$, $Z_2$, respectively.

The space of $L^2$ harmonic forms on $\sR$ is naturally identified with the space $H^1(\widebar{2}{\sR},\R)$ of harmonic 1-forms on $\widebar{2}{\sR}$ and is denoted by $H^1(\sR,\R)$.

In addition to these we consider harmonic forms with \emph{logarithmic poles}. By the existence theorem of real harmonic functions with logarithmic poles (\cite[Theorem II.4.3]{FK92}) there exists on $\sR$, up to a change of sign, a unique exact harmonic 1-form $\mT_1$ with logarithmic poles at $v_1$, $v_2$. With the appropriate choice of sign this form has the following representation in the cylinders $Z_1$, $Z_2$,
 \begin{equation}\label{eq:TinZ}%
\mT_1 = d x + \mT_1^{\nlin},
 \end{equation}%
where the nonlinear part $\mT_1^{\nlin}$ is exponentially decaying for $x \to -\infty$ in $Z_1$ and for $x \to +\infty$ in $Z_2$. We now define the space
 \begin{equation}\label{eq:HvsR}%
\Hv{\sR} = {\rm span}\{ H^1(\sR,\R), \mT_1, \Hstar \mT_1 \}.
 \end{equation}%
Observe that in $Z_1 \cup Z_2$ any $\omega \in \Hv{\sR}$ has a representation 
 \begin{equation*}%
\omega = a d x + b d  y + \omega^{\nlin}
 \end{equation*}%
%
with the same $a, b$ in $Z_1$ as in $Z_2$, and that $a=b=0$ iff $\omega \in H^1(\sR,\R)$. Analogously to \beqref{eq:EssScalPeod} we define the finite scalar product for $\omega, \eta \in \Hv{\sR}$,
 \begin{equation}\label{eq:Kproduct}%
\scp{\omega}{\eta}_{\frak{K}} \defeq \int_{\frak{K}}\omega \wedge \Hstar \eta + \int_{Z_1 \cup Z_2} \omega^{\nlin} \wedge \Hstar \eta^{\nlin}.
 \end{equation}%
By the exponential decay of the nonlinear parts the integral on the right is well defined. Since any harmonic form that vanishes in $\frak{K}$ vanishes everywhere the product defined by \beqref{eq:Kproduct} is positive definite and, hence, a scalar product. It is not intrinsic but shall serve its purpose. 

Since $\Hstar(\omega^{\nlin})=(\Hstar \omega)^{\nlin}$ we have
 \begin{equation}\label{eq:starKp}%
\scp{\Hstar\omega}{\Hstar\eta}_{\frak{K}} = \scp{\omega}{\eta}_{\frak{K}}.
 \end{equation}%
We also introduce the \emph{essential energy} with respect to this product, respectively $\frak{K}$,%
 \begin{equation}\label{eq:EssK}%
\Esss{\frak{K}}(\omega) \defeq \scp{\omega}{\omega}_{\frak{K}}.
 \end{equation}%
We now prove two properties of $\mT_1$. For the proof we introduce the following notation which will also be used later on. For any $L>0$ we consider the finite cylinders $Z_{1,L} = [-\frac{1}{2}L,0] \times \Sp_1$, $Z_{2,L} = [0,\frac{1}{2}L] \times \Sp_1$, and let 
 \begin{equation*}%
\sR_L^{\times} = Z_{1,L} \cup \frak{K} \cup Z_{2,L}
 \end{equation*}%
be the subsurface of $\sR$ obtained by attaching the two cylinders to $\frak{K}$ with the same pasting rules as for $Z_1$, $Z_2$ in the definition of $\sR$.
 \begin{lem}\label{lem:T1ortw}%
\textup{i)}\; If $\omega \in {\rm span}\{ H^1(\sR,\R), \Hstar \mT_1 \}$ then $\scp{\mT_1}{\omega}_{\frak{K}} = 0$.

\textup{ii)}\; If $\omega \in H^1(\sR,\R)$ then $\scp{\Hstar \mT_1}{\omega}_{\frak{K}}=0$.
 \end{lem}%
 \begin{proof}%
By virtue of the Orthogonality lemma we have
 \begin{equation*}%
\scp{\mT_1}{\omega}_{\frak{K}} =\lim_{L \to \infty}\Big(\int_{\frak{K}}\mT_1 \wedge \Hstar \omega + \int_{Z_{1,L}\cup Z_{2,L}} \mT_1^{\nlin} \wedge \Hstar \omega^{\nlin}\Big) = \lim_{L \to \infty} \int_{\sR_L^{\times}}\mT_1 \wedge \Hstar \omega.
 \end{equation*}%
$\mT_1$ being exact it is of the form $\mT_1 = d F$ with a harmonic function $F$ on $\sR$ whose nonlinear part in $Z_1$, $Z_2$ is exponentially decaying for $x \to -\infty$ respectively, $x \to + \infty$. By Stokes' theorem and because $\Hstar \omega$ is of the form $\Hstar \omega = a d x + \Hstar \omega^{\nlin}$ for some constant $a$ we have
 \begin{equation*}%
\int_{\sR_L^{\times}}\mT_1 \wedge \Hstar \omega = \int_{\sR_L^{\times}}d F \wedge \Hstar \omega = \int_{\sR_L^{\times}} d(F \Hstar \omega)=\int_{\partial \sR_L^{\times}}F \Hstar \omega = \int_{\partial \sR_L^{\times}}F \Hstar \omega^{\nlin}.
 \end{equation*}%
Since $\vert F \vert$ has at most linear growth while $\Vert \Hstar \omega^{\nlin} \Vert$ is exponentially decaying for $x \to -\infty$, respectively $x \to +\infty$, the right hand side goes to zero as $L\to \infty$ and point i) of the lemma follows. Point ii) follows from i) and \beqref{eq:starKp} since for $\omega \in H^1(\sR,\R)$ we also have $\Hstar \omega \in H^1(\sR,\R)$.
 \end{proof}%
 \begin{lem}\label{lem:EssT}%
 \begin{equation*}%
\Esss{\frak{K}}(\mT_1)=\lim_{L\to \infty}E_{\sR_L^{\times}}(\mT_1)-L.
 \end{equation*}%
 \end{lem}%
 \begin{proof}%
This is straightforward, using the Orthogonality lemma:
 \begin{align*}%
\Esss{\frak{K}}(\mT_1) = E_{\frak{K}}(\mT_1)+\lim_{L\to\infty}E_{Z_{1,L}\cup Z_{2,L}}(\mT_1^{\nlin})&=E_{\frak{K}}(\mT_1)+\lim_{L\to\infty}\big(E_{Z_{1,L}\cup Z_{2,L}}(\mT_1)-L\big)\\
&=
\lim_{L\to \infty}\big(E_{\sR_L^{\times}}(\mT_1)-L\big).
 \end{align*}%
\par\vspace{-24pt}
 \end{proof}
For the convergence of the period matrices we need the convergence of harmonic forms. To cover all cases we consider the grafting in terms of $\sR$. To this end we define $\sR_L$ to be the compact Riemann surface obtained from $\sR_L^{\times}$ by pasting together the two boundary curves by the rule that any point $(\frac{1}{2}L,y)$ on the boundary of $Z_{2,L}$ is identified with the point $(-\frac{1}{2}L,y)$ on the boundary of $Z_{1,L}$. In $\sR_L$ the pasted cylinders form the subset 
 \begin{equation*}%
Z_L \defeq Z_{2,L}\cup_{\rm \tiny pg} Z_{1,L} \subset \sR_L,
 \end{equation*}%
where $\cup_{\rm \tiny pg}$ means disjoint union modulo pasting. By abuse of notation we write $Z_L = [0,L]\times \Sp_1$. As in the case of $S$, for any $L > 0$ there is a marking homeomorphism $\phi_L : \sR_1 \to \sR_L$ that acts as the identity from $\frak{K} \subset \sR_1$ to $\frak{K} \subset \sR_L$ and sends $[0,1]\times \Sp_1$ to $[0,L]\times \Sp_1$ by the rule $\phi_L(x,y) = (L x, y)$.

On $\sR_1$ we choose a canonical homology basis $(\alpha_1,\dots,\alpha_{2g})$ such that $\alpha_2$ is the boundary curve $c_1$ of $\frak{K}$. Setting $\alpha_{i,L} = \phi_L(\alpha_i)$, $i=1,\dots,2g$, we get a corresponding homology basis on $\sR_L$, and we let $(\sigma_{1,L}, \dots, \sigma_{2g,L})$ be the dual basis of harmonic forms on $\sR_L$, and $(\mT_{1,L}, \tau_{2,L}, \dots, \tau_{2g,L})$ the modified basis with  $\mT_{1,L}$ as in \bbref{Section}{sec:sigma1Q} and the $\tau_{j,L}$ as in \bbref{Section}{sec:sigmajQ}. We recall from \bbref{Remark}{rem:T1t2star} that $\mT_{1,L}$ can be written as linear combination $\mT_{1,L} = \Hstar \tau_{2,L} + \sum_{k=3}^{2g}c_{k,L}\tau_{k,L}$ and thus has linear part $d x$ in $Z_L$.

As $L \to \infty$ these forms converge. To bring this statement into correct form we associate to any harmonic form $\omega$ on $\sR_L$ a form $\omega^{\times}$ on $\sR$ in the following way: first we let $\omega^{\times}$ be the restriction of $\omega$ to $\sR_L^{\times}$, then we interpret $\sR_L^{\times}$ as subset of $\sR$ and finally we extend (discontinuously) $\omega^{\times}$ to all of $\sR$ defining it to be zero outside $\sR_L$. The statement is now
 \begin{thm}[Convergence theorem]\label{thm:Converge}%
For $L \to \infty$ we have
\begin{enumerate}
\itemup{i)} $\mT_{1,L}^{\times}$ converges uniformly on compact sets on $\sR$ to $\mT_1$;
\itemup{ii)} $\tau_{j,L}^{\times}$ converges uniformly on compact sets to a harmonic form $\tau_j \in \Hv{\sR}$, $j=2,\dots,2g$.
\end{enumerate}
 \end{thm}%
 \begin{proof}%
This is an adaption of \cite{DMP12}, \cite{Do12}, where a more general result is proved. We begin with the $\tau_{2,L}$. To get suitable test forms for them we fix a smooth closed 1-form $\theta_2$ on $\frak{K}$ that has vanishing periods over $\alpha_3, \dots, \alpha_{2g}$ and matches with $d y$ along the boundary of the inserted cylinder $Z_L = \{(x,y) \mid x \in [0,L], y \in \Sp_1\}$. Setting $\theta_{2,L} = \theta_2$ on $\frak{K}$ and $\theta_{2,L} = d y$ on $Z_L$ we get a test form with energy $E(\theta_{2,L}) = \vartheta +L$ where $\vartheta = E_{\frak{K}}(\theta_2)$. On the other hand, $E(\tau_{2,L}) \geq L$ by the Orthogonality lemma. Thus,
 \begin{equation*}%
L \leq E(\tau_{2,L}) \leq \vartheta + L.
 \end{equation*}%
It follows that the energies of all $\tau_{2,L}^{\times}$ on $\frak{K}$ and also the energies of their nonlinear parts in the infinite cylinders $Z_1$, $Z_2$ are uniformly bounded by $\vartheta$. Since the forms are harmonic one can now, using Harnack's inequality, apply the Arzel{\`a}-Ascoli diagonal argument to get a sequence of lengths $\{L_n\}_{n=1}^{\infty}$ with $L_n \to \infty$ such that $\tau_{2,L_n}^{\times}$ converges uniformly on compact sets to a harmonic form $\tau_2$ on $\sR$ satisfying $\Esss{\frak{K}}(\tau_2) \leq \vartheta$ and $\int_{\alpha_2}\tau_2 = 1$, $\int_{\alpha_k} \tau_2 = 0$, $k=3,\dots,2g$. Now $\tau_2$ is uniquely determined by these properties and the Arzel{\`a}-Ascoli argument applied a second time implies that the entire family $\tau_{2,L}^{\times}$ converges to $\tau_2$. Since all $\tau_{2,L}$ have linear part $dy$ (by \bbref{Lemma}{lem:nolinear}) the same is true for $\tau_2$. Finally, since $\Esss{\frak{K}}(\tau_2)$ is finite it follows that $\tau_2$ belongs to $\Hv{\sR}$. 

For $\tau_{k,L}$, $k=3,\dots, 2g$, the proof is the same except that the linear parts vanish. Here the energy of the limit forms $\tau_k$ are finite and so, by the lifting of isolated singularities property of $L^2$ harmonic differentials we have
 \begin{equation}\label{eq:taukH1}%
\tau_k \in H^1(\sR,\R), \quad k=3,\dots, 2g.
 \end{equation}%
For the proof of statement i) we write $\mT_{1,L} = \Hstar \tau_{2,L} + \sum_{k=3}^{2g}c_{k,L}\tau_{k,L}$ (\bbref{Remark}{rem:T1t2star}), where $c_{k,L} = -\int_{\alpha_k}\Hstar \tau_{2,L}$. By the convergence of the $\tau_{k,L}$, $k=3,\dots,2g$, and since they have vanishing linear parts it follows that $\Esss{\frak{K}}(\mT_{1,L})$ is bounded above by some constant that does not depend on $L$. We can therefore apply the same argument for $\mT_{1,L}$ as for $\tau_{2,L}$ plus the uniqueness property of $\mT_1$ to conclude the proof.
 \end{proof}%

\subsection{Twist at infinity}\label{sec:TwistInf}

Let $\sR$ be as before, but this time looked at with respect to its complete hyperbolic metric. We represent it in the form $\sR = V_1 \cup \frak{M}^{\ssR}  \cup V_2$ where $V_1$, $V_2$  are the cusp neighbourhoods of the ideal points at infinity $v_1$, $v_2$, cut off along the horocycles of length 1 
\begin{equation}\label{eq:cusphyp}%
\begin{aligned}
V_1 &= \{(x,y) \mid x \in ({-}\infty,0], y \in \Sp_1 \},\; \text{with the hyperbolic metric }  ds^2 = \frac{d x^2 + d y^2}{(x-1)^2},\\
V_2 &= \{(x,y) \mid  x \in [0, {+}\infty), y \in \Sp_1\},\; \text{with the hyperbolic metric }  ds^2 = \frac{d x^2 + d y^2}{(x+1)^2}.
\end{aligned}
\end{equation}
The curves $y = {\rm constant}$ in $V_1$, $V_2$ are geodesics. For convenient reference we shall call them \emph{cusp geodesics}.

The part $\frak{M}^{\ssR}$ is the closure of $\sR \setminus \{V_1 \cup V_2\}$ and shall be called the \emph{main part} of $\sR$ with respect to the hyperbolic metric. In the cases $\sR = S^{\sG}$ and $\sR = S^{\sF}$ we shall, respectively, use the notation $\frak{M}^{\sG}$ and $\frak{M}^{\sF}$ for it.

When $\sR$ arises e.g.\ as the Fenchel-Nielsen limit of a family $\big(S_t\big)$ in which $\alpha_{2,t}$ shrinks to zero then, as mentioned earlier, two problems arise: a degree of freedom is lost since the twist parameter along $\alpha_{2,t}$ disappears, and secondly, the homology of $\alpha_{1,t}$ disappears. We eliminate both problems by introducing classes of curves from $v_1$ to $v_2$ on $\sR$ that play the role of ``twist parameters at infinity''.

We shall say that a parametrised curve $\alpha : \R \to \sR$ is an \emph{admissible curve} if on $\widebar{2}{\sR}$ one has \mbox{$\displaystyle \lim_{s \to -\infty}\alpha(s) = v_1$}, $\displaystyle \lim_{s \to +\infty}\alpha(s) = v_2$ and if, furthermore, in either cusp $\alpha$ is asymptotic to a cusp geodesic. \emph{Asymptotic} in $V_1$ (and similarly in $V_2$) means that if the part of $\alpha$ in $V_1$ with respect to the above coordinates is written  in the form $\alpha(s) = (x(s), y(s))$, then $y(s)$ converges to some $y_0$ for $s \to - \infty$.
 \begin{rem}\label{rem:existintalpha}%
From the asymptotic property it follows that for any admissible curve $\alpha$ from $v_1$ to $v_2$ in $\sR$ and any harmonic form $\omega \in {\rm span}\{ H^1(\sR,\R), \Hstar \mT_1 \}$ the integral $\int_{\alpha} \omega$ is well defined and finite.
 \end{rem}%
A homotopy between admissible curves from $v_1$ to $v_2$ on $\sR$ shall be called an \emph{admissible homotopy} if in either cusp the curves remain asymptotic to their respective cusp geodesics. Such a homotopy is, of course, only possible if in either cusp the initial and the end curve is asymptotic to the same geodesic. 

It turns out that the concept of admissible homotopy is too restrictive and we shall slightly weaken it. Before doing so we prove the following, where we return to the representation of $\sR$ used in \bbref{Section}{sec:HvsR}
 \begin{lem}\label{lem:coolinasymp}
Assume $\sR$ is represented in the form $\sR = Z_1 \cup \frak{K}  \cup Z_2$ as in \beqref{eq:RKZZ}. Then the straight lines $y = {\rm const}$ in the cylinders are asymptotic to cusp geodesics.
 \end{lem}%

We note that this lemma is evident for the limit surface $S^{\sG} = V_1 \cup \frak{M}^{\sG} \cup V_2$, represented with respect to the complete hyperbolic metric that arises from the Uniformization theorem, where $V_1$, $V_2$ are the hyperbolic cusps as in \ref{eq:cusphyp} and 
$\frak{M}^{\sG}$ is the main part with respect to the hyperbolic metric of $S^{\sG}$.  However, we also need this statement for the representation $S^{\sG} = Z_1 \cup \frak{M} \cup Z_2$, where $\frak{M}$ is the main part of $S$ and the coordinate lines $y =  \text{const.}$ in the cylinders are not geodesics with respect to the hyperbolic metric of $S^{\sG}$.
 \begin{proof}%
This can be shown by using that passing from $V_i$ to $Z_i$ is a change of conformal coordinate systems in a neighbourhood of $v_i \in \widebar{2}{\sR}$ (i=1,2), and that for such changes the overlap map is biholomorphic. We use another approach that measures the angle between the coordinate lines and the geodesics in terms of $\mT_1$. We do, however, not estimate the constants.

We begin with $Z_2$. Since $\mT_1$ is exact and has linear part $d x$ it is of the form $\mT_1 = d h$ with a harmonic function $h(x,y) = x + h^{\nlin}(x,y)$, where by \bbref{Lemma}{thm:lem_flat} the nonlinear part satisfies $\vert h^{\nlin}(x,y) \vert \leq c_1 e^{-2 \pi x}$, $\Vert d h^{\nlin}(x,y)\Vert \leq c_2 e^{-2\pi x}$ for some constants $c_1$, $c_2$. For $x$ large enough the first property implies that $h$ grows asymptotically like $x$, and the second property implies that the gradient ${\rm grad}h(x,y)$ at $(x,y)$ and the ``horizontal'' coordinate line  $[0,+\infty) \times \{y\}$ through $(x,y)$ form an absolute angle $\leq 2c_2 e^{-2\pi x}$.

Now we turn to $V_2$, where we write the coordinates as $\tilde{x}$, $\tilde{y}$. Here $\mT_1$ is $d \tilde{h}$ with some harmonic function $\tilde{h}(\tilde{x},\tilde{y})$, and the horizontal coordinate lines are geodesics. We have the same angle estimates (with different constants). Furthermore, since $\tilde{h}(\tilde{x},\tilde{y})$ grows asymptotically like $\tilde{x}$ there exists a constant $c_3$ (possibly negative) such that 
if $(\tilde{x},\tilde{y})$ in $V_2$ and $(x,y)$ in $Z_2$ represent the same point $p$ of $\sR$, then $x \geq \tilde{x} + c_3$. 

The two angle estimates together with the inequality $x \geq \tilde{x} + c_3$ imply that for some positive constant $c_4$ and $\tilde{x}$ large enough the horizontal coordinate line (in the sense of $Z_2$) through $(\tilde{x},\tilde{y})$ and the cusp geodesic through $(\tilde{x},\tilde{y})$ form an angle $\angle(\tilde{x},\tilde{y}) \leq c_4 e^{-2\pi \tilde{x}}$. This is sufficient to prove that the coordinate line is asymptotic to one of the cusp geodesics. In $V_1$ the proof is the same.
\end{proof}%
We now slightly loosen the notion of admissible homotopy classes by allowing ``synchronised twists'' in the cusps. We define this in terms of an arbitrary description \mbox{$\sR = Z_1 \cup \frak{K}  \cup Z_2$} and then give an intrinsic characterisation. For this we use twist homeomorphisms $J_1, J_2 : \sR \to \sR$ defined as follows. Choose any constant $b \in \R$ and let $\beta : [0,1] \to \Sp_1$ be the path \mbox{$\beta(s) = b s \mod 1$} (e.g.\ if $b =2$ then $\beta$ goes twice around $\Sp_1$.) For any $x_0 \leq 0$ (respectively, $x_0 \geq 0$) we then have a path in $Z_1$ (respectively, $Z_2$) $s \mapsto (x_0,\beta(s))$, $s \in [0,1]$.

Next, choose any $a_1 \leq 0$, $a_2 \geq 0$ and define the cut off functions
 \begin{equation*}%
\chi_1(x) =
\begin{cases}
1&\text{if $x \leq a_1-1$,}\\
a_1-x&\text{if $a_1-1\leq x \leq a_1$,}\\
0&\text{if $a_1 \leq x \leq 0$,}
\end{cases}
\quad
\chi_2(x) =
\begin{cases}
0&\text{if $0 \leq x \leq a_2$,}\\
x-a_2&\text{if $a_2 \leq x \leq a_2+1$,}\\
1&\text{if $x \geq a_2+1$.}
\end{cases}
 \end{equation*}%
This allows us to define homeomorphisms $J_1, J_2 : \sR \to \sR$ setting
 \begin{equation*}%
\text{on $Z_1$}:\; J_1(x,y) = (x,y+\beta(\chi_1(x))),
\quad
\text{on $Z_2$}:\; J_2(x,y) = (x,y+\beta(\chi_2(x)))
 \end{equation*}%
and letting $J_1$ be the identity mapping on $\frak{K}\cup Z_2$, and $J_2$ the identity mapping on $Z_1 \cup \frak{K}$. Under the mappings $J_i$ the parts in the cylinders on the left of $a_1$ and on the right of $a_2$ are then both rotated by the angle $2\pi b$. We shall call these mappings \emph{twist homeomorphisms} with twist parameter $b$. 

The next lemma characterises $b$ in an intrinsic way. For the lemma we recall that by the remark subsequent to \beqref{eq:HvsR} any $\omega \in {\rm span}\{ H^1(\sR,\R), \Hstar \mT_1 \}$ has linear part $w d y$ in $Z_1$ and $Z_2$ for some $w \in \R$.

 \begin{lem}\label{lem:twistJb}%
Let $J_1, J_2:\sR \to \sR$ be twist homeomorphisms with twist parameter $b$ as above and $\omega \in {\rm span}\{ H^1(\sR,\R), \Hstar \mT_1 \}$ a harmonic form with linear part $w d y$ in $Z_1$, $Z_2$. Then for any admissible curve $\alpha$ from $v_1$ to $v_2$ we have
 \begin{equation*}%
\int_{J_1\circ\alpha}\omega = \int_{\alpha}\omega - wb,
\quad
\int_{J_2\circ\alpha}\omega = \int_{\alpha}\omega + wb.
 \end{equation*}%
 \end{lem}%
 \begin{proof}%
Let $a_1$, $a_2$ be the constants in the definition of $J_1$, $J_2$ as above. If we replace $a_1$, $a_2$, by some other values $a_1' \leq -1$, $a_2' \geq 1$, and keep $b$ fixed, then the twist homeomorphisms $J_1'$, $J_2'$ based on these new constants have the property that $J_i\circ\alpha$ and $J_i'\circ\alpha$ are homotopic with a homotopy that is the identity outside some compact region and thus $\int_{J_i\circ\alpha} \omega = \int_{J_i'\circ\alpha} \omega$ ($i=1,2)$. This implies that for $i=1,2$, the difference $d_i := \int_{J_i\circ\alpha} \omega - \int_{\alpha}\omega$, seen as functions of $a_i$ is the constant function. Now, if $a_1 \to -\infty$ and $a_2 \to +\infty$, then, by the exponential decay of the nonlinear part of $\omega$ in the cylinders, $d_1$ converges to $-b w$ and $d_2$ converges to $b w$. 
 \end{proof}%

We now define two admissible curves $\alpha$, $\alpha'$ from $v_1$ to $v_2$ to be \emph{twist equivalent} if there exists an admissible homotopy of $\alpha$ into a curve $\alpha''$ and a pair of twist homeomorphisms $J_1$, $J_2$ with the \emph{same} twist parameter $b$ such that $\alpha' = J_1\circ J_2 \circ \alpha''$.

It is a routine exercise to show that this is an equivalence relation thus splitting the admissible curves from $v_1$ to $v_2$ into \emph{twist equivalence classes}. The next lemma shows, among other things, that this definition does not depend on the choice of the representation $\sR = Z_1 \cup \frak{K} \cup Z_2$.
 \begin{lem}\label{lem:twistequiv}%
\textup{(i)} If two admissible curves $\alpha$, $\alpha'$ from $v_1$ to $v_2$ on $\sR$ are twist equivalent then $\int_{\alpha} \omega = \int_{\alpha'} \omega$ for any $\omega \in {\rm span}\{ H^1(\sR,\R), \Hstar \mT_1 \}$.

\textup{(ii)} If two admissible curves $\alpha$, $\alpha'$ from $v_1$ to $v_2$ on $\sR$ are homotopic on $\widebar{2}{\sR}$ then they are twist equivalent if and only if $\int_{\alpha} \Hstar \mT_1 = \int_{\alpha'} \Hstar \mT_1$.
\end{lem}%
 \begin{proof}%
(i). From the exponential decay of the nonlinear parts of $\omega$ in the cylinders it follows that if $\alpha$ and $\alpha''$ are homotopic by an admissible homotopy, then the integrals over $\alpha$ and $\alpha''$ are the same. By \bbref{Lemma}{lem:twistJb} the same is true for the integrals over $\alpha''$ and $J_1 \circ J_2 \circ \alpha''$ if $J_1$ and $J_2$ have the same twist parameter $b$.

(ii). Now assume that $\alpha$ and $\alpha'$ are homotopic on $\widebar{4}{\sR\pt}$ (meaning that after a parameter change the curves are extended to $v_1$, $v_2$ for endpoints and then the homotopy acts on the extended curves keeping $v_1$, $v_2$ fixed). This condition is equivalent to saying that there exists an admissible homotopy from $\alpha$ to some admissible curve $\alpha''$ and two twist homeomorphisms $J_1$, $J_2$ with some twist parameters $b_1$, $b_2$ (that do not coincide, in general) such that $\alpha' = J_1 \circ J_2 \circ \alpha''$. By (i) we have $\int_{\alpha}\Hstar \mT_1 = \int_{\alpha''}\Hstar \mT_1$, and by \bbref{Lemma}{lem:twistJb} $\int_{\alpha''}\Hstar \mT_1 = b_2-b_1 + \int_{J_1\circ J_2 \circ \alpha''} \Hstar \mT_1 = b_2-b_1 + \int_{\alpha'} \Hstar \mT_1$. If we now assume that $\int_{\alpha} \Hstar \mT_1 = \int_{\alpha'} \Hstar \mT_1$ then $b_1 = b_2$ and thus $\alpha$ and $\alpha'$ are twist equivalent. This proves (ii) in the ``{\ }if{\ }'' direction. The ``{\ }only if{\ }'' direction is a special case of (i). \end{proof} 
 \begin{figure}[h]
 \vspace{0pt}
 \begin{center}
 \leavevmode
 \SetLabels
(0.49*0.105) $\frak{M}^{\ssR}$\\
(0.17*0.59) $V_1$\\
(0.86*0.34) $V_2$\\
(0.40*0.685) $A_0$\\
(0.253*0.44) \raisebox{2pt}{$\scriptscriptstyle \frac{1}{2}$}$\vartheta$\\
(0.778*0.443) \raisebox{2pt}{$\scriptscriptstyle \frac{1}{2}$}$\vartheta$\\
(0.10*0.45) $A_1$\\
(0.91*0.61) $A_2$\\
(0.26*0.06) $c_1$\\
(0.75*0.06) $c_2$\\
 \endSetLabels
 \AffixLabels{
 \includegraphics{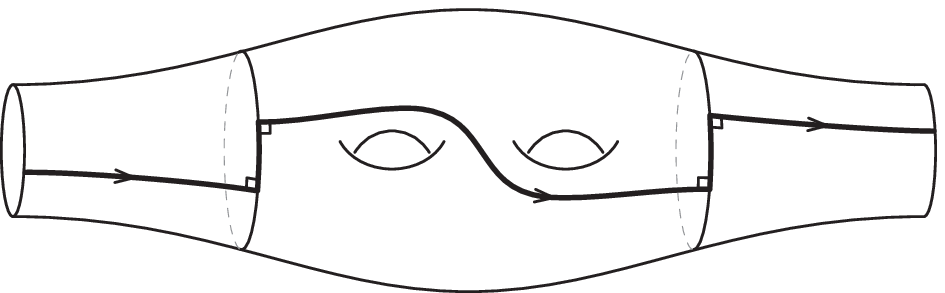} }
 \end{center}
 \vspace{-0pt}
 \caption{\label{fig:Normalform} Twist equivalence class represented by a polygon in normal form. The twist is $\vartheta$.}
 \end{figure}
For future reference we associate a twist parameter $\vartheta$ with any twist class using the representation $\sR = V_1 \cup \frak{M}^{\ssR} \cup V_2$ as in the beginning of this section.  Let $c_i$ be the common boundary of $\frak{M}^{\ssR}$ and $V_i$, $i=1,2$. For the arcs of curves in $\frak{M}^{\ssR}$ with initial point on $c_1$ and endpoint on $c_2$ we have the homotopy classes with endpoints moving freely on $c_1$, $c_2$. These classes are in natural one-to-one correspondence with the homotopy classes of curves from $v_1$ to $v_2$ in $\widebar{2}{\sR}$. Furthermore, in each class there is a unique \emph{orthogeodesic}, i.e. a geodesic arc that meets $c_1$, $c_2$ orthogonally at its endpoints.

Let now $A_0$ be an orthogeodesic from $c_1$ to $c_2$ (\bbref{Fig.}{fig:Normalform}) and take $\vartheta \in \R$. Along $c_1$ we attach an arc $a_{1,\vartheta}$ of oriented length $\vartheta$ whose endpoint coincides with the initial point of $A_0$; we attach a similar arc $a_{2,\vartheta}$ along $c_2$ whose initial point coincides with the endpoint of $A_0$. With respect to the earlier described coordinates for $V_1$, $V_2$ (where the $y$-coordinates are taken $\mod 1$) these arcs may be parametrised as $A_{1,\vartheta}(s) = (-1,\vartheta s +y_1)$, $A_{2,\vartheta}(s) = (1,\vartheta s +y_2)$, $s\in[0,1]$ for appropriate $y_1$, $y_2$ . By further attaching cusp geodesics $A_1 \subset V_1$ to the initial point of $A_{1,\vartheta}$ and  $A_2 \subset V_2$ to the endpoint of $A_{2,\vartheta}$ we get an admissible curve 
 \begin{equation}\label{eq:normalform}%
A_{1,\vartheta} = A_1 A_{1,\vartheta} A_0 A_{2,\vartheta} A_2  \end{equation}%
from $v_1$ to $v_2$ as depicted in \bbref{Fig.}{fig:Normalform}. By \bbref{Lemma}{lem:twistJb} and \bbref{}{lem:twistequiv} we have the following.
 \begin{rem}\label{rem:normalform}%
$A_{1,\vartheta}$ represents a twist equivalence class of admissible curves from $v_1$ to $v_2$ on $\sR$. Conversely, any twist equivalence class is represented by a polygon as in \beqref{eq:normalform} with uniquely determined orthogeodesic $A_0$ and uniquely determined $\vartheta$. We call $\vartheta$ the \emph{twist parameter} of the class and $A_{1,\vartheta}$ its representative in \emph{normal form}.

We also remark by virtue of \bbref{Lemma}{lem:twistJb} that for any fixed $A_0$ there exists a constant $c_0$ such that
 \begin{equation}\label{eq:Tandtwist}%
\int_{A_{1,\vartheta}} \Hstar \mT_1 = \vartheta + c_0, \quad \text{for all $\vartheta \in \R$.}
 \end{equation}%

 \end{rem}%
By fixing  an orthogeodesic $A_0$ or, what amounts to the same, a homotopy class of admissible curves from $v_1$ to $v_2$ on $\widebar{2}{\sR}$ we can blow up $\sR$ into a one parameter family by marking it with the twist equivalence classes occurring in \beqref{eq:normalform}. For $\sR$ marked with $A_{1,\vartheta}$ we may then call $\vartheta$ its \emph{twist parameter at infinity}.

\subsection{Period matrix with a parameter}\label{sec:Pmatpar}
Let \mbox{$\sR = V_1 \cup \frak{M}^{\ssR} \cup V_2$} be as above and mark it with the twist equivalence class of some admissible curve $\mathcal{A}_1$ from $v_1$ to $v_2$. Let $\alpha_2$ be the boundary curve of $V_1$ and select a canonical homology basis $\alpha_3, \dots, \alpha_{2g}$ of $\widebar{2}{\sR}$ satisfying
 \begin{equation*}%
\mathcal{A}_1 \cap \alpha_{j} = \emptyset, \quad j = 3, \dots, 2g.
 \end{equation*}%
To this configuration we shall associate a parametrized family of $2g$ by $2g$ matrices $P_{\sR}(\lambda)$, $\lambda>0$, where the notation suppresses the mention of the curve system. Here the parameter $\lambda$ is related to the length of the attached cylinders. 

Let $(\tau_3, \dots, \tau_{2g})$ be the dual basis of harmonic forms for $(\alpha_3, \dots, \alpha_{2g})$ on $\widebar{2}{\sR}$ (restricted to $\sR$), $\mT_1$ the exact form with linear parts $d x$ in $V_1$, $V_2$ as in \bbref{Section}{sec:HvsR} and set 
 \begin{equation}\label{eq:tautwo}%
\tau_2 = \Hstar \mT_1 - \sum_{k=3}^{2g}\big({\textstyle \int_{\alpha_k}} \Hstar \mT_1 \big) \tau_k
 \end{equation}%
so that $\tau_2$ has period 1 over $\alpha_2$ and period 0 over $\alpha_3, \dots, \alpha_{2g}$. With these forms we define
 \begin{equation}\label{eq:Pmatpar1}%
q_{ij}=\int_{\sR}\tau_i \wedge \Hstar \tau_j, \quad i,j = 2,\dots,2g, \; (i,j) \neq (2,2)
 \end{equation}%
The matrix $Q = \big(q_{i,j}\big)_{i,j \geq 3}$ is the Gram period matrix of $\widebar{2}{\sR}$ with respect to the homology basis $(\alpha_3, \dots, \alpha_{2g})$. For $j \geq 3$ the integrals $q_{2j}$ and $q_{j2}$ are well defined by the Orthogonality lemma since the $\tau_j$ have vanishing linear parts in the cusps (given that $\tau_j \in L^2(\sR)$, $j\geq 3$). For $(i,j) = (2,2)$ the integral is infinite and we define an ersatz
 \begin{equation}\label{eq:Pmatpar2}%
\pi_{22} = E(\tau_2 - \Hstar \mT_1)
 \end{equation}%
which is again finite given that $\tau_2$ and $\Hstar \mT_1$ have the same linear parts. 
We also note that, for any representation $\sR = Z_1 \cup \frak{K} \cup Z_2$,
 \begin{equation}\label{eq:Formulap22}%
\pi_{22} = \Ess_{\frak{K}}(\tau_2) - \Ess_{\frak{K}}(\mT_1).
 \end{equation}%
Writing $\tau_2 = \Hstar \mT_1 + \omega_2$, where $\omega_2$ and also $\Hstar \omega_2$ belong to $H^1(\sR,\R)$ we have, by \bbref{Lemma}{lem:T1ortw}, \linebreak $\Ess_{\frak{K}}(\tau_2)=\scp{\Hstar \mT_1}{\Hstar \mT_1}_{\frak{K}} + \scp{\omega_2}{\omega_2}_{\frak{K}} = \Ess_{\frak{K}}(\Hstar \mT_1) + E(\omega_2)$.

We complete the definitions setting 
 \begin{equation}\label{eq:Pmatpar3}%
q_{1j} = q_{j1} = 0, \quad j=1,\dots , 2g.
 \end{equation}%
Finally, we need the constants
 \begin{equation}\label{eq:Pmatpar4}%
\kappa_1 \defeq 1, \quad \kappa_j \defeq -\int_{\mathcal{A}_1}\tau_j, \quad j=2, \dots, 2g.
 \end{equation}%
In the case where $\sR$ is the limit surface $S^{\sG}$ or $S^{\sF}$ these constants will, respectivley, be denoted by
 \begin{equation}\label{eq:Pmatpar5}%
q_{ij}^{\sG}, \pi_{22}^{\sG}, \kappa_j^{\sG},
\quad
q_{ij}^{\sF}, \pi_{22}^{\sF}, \kappa_j^{\sF}, \quad i,j = 1,\dots, 2g,\; (i,j) \neq (2,2).
 \end{equation}%
 \begin{defi}\label{def:Pmatpar}%
$P_{\sR}(\lambda)$, $\lambda>0$, is the matrix function with the following entries based on \beqref{eq:Pmatpar1}-\beqref{eq:Pmatpar4},
 \begin{align*}%
p_{ij}(\lambda) &= q_{ij}+\frac{\kappa_i\kappa_j}{\lambda}, \quad i, j = 1,\dots,2g, \quad (i,j) \neq (2,2),\\
p_{22}(\lambda) &= \pi_{22} + \lambda + \frac{\kappa_2^2}{\lambda}.
 \end{align*}%
 \end{defi}%
The analog of \bbref{Theorem}{thm:small_scg} in the non separating case is now as follows, where the $\mathcal{O}_{\rm A}$-symbol is as in \beqref{eq:OAsymbol}. (The theorem is a restatement of \bbref{Theorem}{thm:nonsepGIntro}, the proof follows in  \bbref{Section}{sec:ProofThmNonsep}).
 \begin{thm}\label{thm:nonsepG}%
Let $S$ be a compact Riemann surface of genus $g \geq 2$, $\gamma$ a non separating geodesic of length $\ell(\gamma) \leq \frac{1}{4}$ on $S$ and $\big(S_L\big)_{L\geq 4}$ with limit $S^{\sG}$ the corresponding grafting family. Choose a canonical homology basis $\rm A = (\alpha_1, \alpha_2, \ldots, \alpha_{2g})$ with $\alpha_2 = \gamma$ and let $P_{S_L}$, $P_{S^{\sG}}(\lambda)$ be the corresponding Gram period matrices. Finally set $\frak{m}^{\sG} = \Ess_{\frak{M}}(\mT_1^{\sG})$, where $\frak{M}$ is the main part of $S$ with respect to $\gamma$. Then for all $L \geq 4$,
 \begin{equation*}%
P_{S_L} = P_{S^{\sG}}(\frak{m}^{\sG}+L) + \rho^{\sG}(L),
 \end{equation*}%
where the entries of the error term $\rho^{\sG}(L)$ have the following growth rates
\begin{alignat*}{2}
\nonumber
& \rho_{1j}^{\sG}(L)  = \mathcal{O}_{\rm A} \left( \tfrac{1}{L}e^{-2\pi L} \right) ,
\quad  &&\text{  \ \ for \ \ } j = 1, \dots, 2g, \,\, j \neq 2,\\
& \rho_{2j}^{\sG}(L)  = \mathcal{O}_{\rm A} \left( \tfrac{1}{L}e^{-\pi L} \right), \quad  &&\text{  \ \ for \ \ } j = 1, \dots, 2g,\\
\nonumber
& \rho_{ij}^{\sG}(L)  = \mathcal{O}_{\rm A} \left( e^{-2\pi L} \right), \quad &&\text{ \ \ for \ \ } i,j = 3, \dots, 2g.
\rule{40ex}{0ex} 
\end{alignat*}%
\end{thm}%
The Gram period matrix $P$ of a compact Riemann surface is known to be symplectic where, in accordance with our conventions for a canonical homology basis (\bbref{Section}{sec:Introd}), $P$ is called \emph{symplectic} in our context,  if $P {\bf J}P={\bf J}$ for the block diagonal matrix
 \begin{equation}\label{eq:Jblock}%
{\bf J}={\scriptstyle
\begin{pmatrix}
0&1&   &   &\\
-1&0&  & &   \rule{-4ex}{0pt}       \smash{\text{\raisebox{1.5ex}{$0$}}}\\
 & &  \ddots\\
 &  & &  0&1\\
& \rule{-4ex}{0pt}       \smash{\text{\raisebox{6pt}{$0$}}}    & &  -1&0
\end{pmatrix}}.
 \end{equation}%
As a consequence of the theorem we have the following.
 \begin{cor}\label{cor:symplectic}%
The matrices $P_{\sR}(\lambda)$ are symplectic for all $\lambda>0$.
 \end{cor}%
 \begin{proof}%
Any $\sR$ can be understood as a limit $S^{\sG}$ of some family $\big(S_L\big)_{L\geq 4}$. Using that all $P_{S_L}$ are symplectic we have, writing $S^{\sG} = \sR$, $\frak{m}^{\sG}+L = \lambda$ and $\rho^{\sG}(L) = r(\lambda)$,
 \begin{equation*}%
{\bf J} = P_{S_L} {\bf J} P_{S_L} = P_{\sR}(\lambda) {\bf J} P_{\sR}(\lambda) + \Omega(\lambda)
 \end{equation*}%
with $\Omega(\lambda) = P_{\sR}(\lambda) {\bf J} r(\lambda) + r(\lambda){\bf J}P_{\sR}(\lambda) + r(\lambda) {\bf J} r(\lambda)$. Now ${\bf J}$ is constant, $P_{\sR}(\lambda) {\bf J} P_{\sR}(\lambda)$ is a rational function of $\lambda$ and $\Omega(\lambda)$ is exponentially decaying as $\lambda \to \infty$. This is only possible if $P_{\sR}(\lambda) {\bf J} P_{\sR}(\lambda)$ is constant and $\Omega(\lambda)$ vanishes identically.
 \end{proof}%

For $t>0$, we set

 \begin{equation*}
L_t = \frac{\pi}{t}-\frac{2\arcsin t}{t}  \text{ \ with \ } \lim_{t \to 0} \frac{2\arcsin t}{t} = 2.
 \end{equation*}%

Then we obtain for the Fenchel-Nielsen construction:

\begin{thm}\label{thm:nonsepF}%
Let $\big(S_t\big)_{t \leq \ell(\gamma)}$ with limit $S^{\sF}$ be the Fenchel-Nielsen family of $S$ and set $\frak{m}^{\sF} = \Ess_{\frak{M^{\sF}}}(\mT_1^{\sF})$, where $\frak{M^{\sF}}$ is the main part of $S^{\sF}$. Then we have the following comparison for the Gram period matrices $P_{S_t}$ and $P_{S^{\sF}}$,
 \begin{equation*}%
P_{S_t} = P_{S^{\sF}}(\frak{m}^{\sF}+L_t) + \rho^{\sF}(t),
 \end{equation*}%
where  the entries of the error term $\rho^{\sF}(t)$ have the following growth rates
\begin{enumerate}
\item[{}]%
$\rho_{11}^{\sF}(t) = \mathcal{O}_{\rm A}(t^4)$, \quad
$\rho_{12}^{\sF}(t) = \mathcal{O}_{\rm A}(t^2)$\quad
$\rho_{1j}^{\sF}(t) = \mathcal{O}_{\rm A}(t^3)$, \quad $j=3, \dots, 2g$,
\item[{}]%
$\rho_{ij}^{\sF}(t)=\mathcal{O}_{\rm A}(t^2)$, \quad $i,j = 2,\dots,2g$.
\end{enumerate}
 \end{thm}%

\subsection{Proof of Theorems \ref{thm:nonsepGIntro} and \ref{thm:nonsepFIntro}}\label{sec:ProofThmNonsep}
We now prove \bbref{Theorems}{thm:nonsepG} and \bbref{}{thm:nonsepF} respectively  \bbref{Theorems}{thm:nonsepGIntro} and \bbref{}{thm:nonsepFIntro}.
 \begin{proof}[Proof of Theorem \ref{thm:nonsepG}]%
In  this proof the following notation is used: $\mT_{1,L}, \tau_{2,L}, \dots, \tau_{2g,L}$ is the relaxed dual basis on $S_L$, and $\mT_1^{\sG}, \tau_2^{\sG}, \dots, \tau_{2g}^{\sG}$ are the harmonic forms on $S^{\sG}$ as above for the case $\sR = S^{\sG}$. We recall the center part $\frak{M}$ common to all $S_L$ which also serves as center part of $S^{\sG}$ and abbreviate the essential energies
 \begin{equation*}%
\frak{m}_L = \Ess_{\frak{M}}(\mT_{1,L}), \quad \frak{m}^{\sG} = \Ess_{\frak{M}}(\mT_1^{\sG}).
 \end{equation*}%
The $q_{ij}$ in the inequalities that follow are the constants in the \bbref{Definition}{def:Pmatpar} for the matrix family $P_{S^{\sG}}(\lambda)$, 
\begin{equation}\label{eq:prf01}
q_{ij}(S^{\sG}) = \scp{\tau_i^{\sG}}{\tau_j^{\sG}}_{\frak{M}}, \quad i,j = 2,\dots,2g,
\end{equation}
where we recall that 
 \begin{equation}%
\scp{\omega}{\eta}_{\frak{M}}  \defeq \int_{\frak{M}}\omega \wedge \Hstar \eta + \int_{C(\gamma)} \omega^{\nlin} \wedge \Hstar \eta^{\nlin},
 \end{equation}%
for $\omega, \eta \in \Hv{S^{\sG}}$ (\beqref{eq:Kproduct} with $\frak{K} = \frak{M}$). In particular, $q_{ii}(S^{\sG}) = E(\tau_i^{\sG})$, $i=3,\dots,2g$, and $q_{22}(S^{\sG}) = \Esss{\frak{M}}(\tau_2^{\sG})$, where by \beqref{eq:Formulap22}
\begin{equation}\label{eq:prf02}
\Esss{\frak{M}}(\tau_2^{\sG}) = \pi_{22}+\Esss{\frak{M}}(\tau_1^{\sG})=\pi_{22} + \frak{m}^{\sG}.
\end{equation}
Finally, we recall the periods introduced in \beqref{eq:kappaj} which we now write $\kappa_{2,L}, \dots, \kappa_{2g,L}$ while $\kappa_2^{\sG}, \dots, \kappa_{2g}^{\sG}$ denote the corresponding periods on $S^{\sG}$.\\

By \bbref{Lemma}{L:sigma1ess}, \bbref{Lemma}{L.kappa1} and the  \bbref{Convergence theorem}{thm:Converge} 

\begin{equation}\label{eq:prf03}
\vert \frak{m}_L - \frak{m}^{\sG} \vert < e^{- 2 \pi L} \min\{\frak{m}_L, \frak{m}^{\sG}\},
 \quad 
\Gamma \leq \frak{m}_L, \frak{m}^{\sG} \leq \frak{z} \Gamma, 
\end{equation}%
with $\frak{z} < 2.3$.  Furthermore, for $j=2$ an application of \beqref{eq:capabreve2} in \bbref{Lemma}{lem:capabreve} together with \bbref{Theorem}{thm:Converge} yields %
\begin{equation}\label{eq:prf04}
\vert \kappa_{2,L} - \kappa_2^{\sG} \vert^2 <  3 \rho_L(1+ \mathfrak{z}\Gamma) q_{22}(S^{\sG}).
\end{equation}%
Here $\rho_L < e^{-2 \pi L}$ is from \beqref{eq:muL}. 
For $j \geq 3$ we have a stronger inequality via inequality \beqref{eq:capabreve15} in \bbref{Lemma}{lem:capabreve}: taking $\omega = \tau_{2,L}$, $\eta = \star \tau_{j,L}$ and   $\tilde{\omega} = \tau^{\sG}_{2}$, $\tilde{\eta} = \star \tau^{\sG}_{j}$ we get, in combination with \beqref{eq:Bkapj2} and \bbref{Theorem}{thm:Converge},
 \begin{equation}\label{eq:prf05}
| \kappa_{j,L} - \kappa_j^{\sG} |   = \vert \scp{\tau_{2,L}}{ \star \tau_{j,L}}_{\frak{M}} - \ \scp{\tau^{\sG}_{2}}{ \star \tau^{\sG}_{j}}_{\frak{M}}\vert \leq \nu_L \ssqrt{q_{22}(S^{\sG})} \ssqrt{q_{jj}(S^{\sG})}, \quad j=3, \dots, 2g,
 \end{equation}%
 
where $\nu_L < e^{-2 \pi L}$ is again from \beqref{eq:muL}.
Finally, we recall \bbref{Lemma}{lem:Bdkappa2} and \bbref{Lemma}{L.Bkappa}
\begin{equation}\label{eq:prf06}
\vert \kappa_{2,L}\vert \leq K, \quad \vert \kappa_{j,L} \vert \leq \sqrt{\Ess(\tau_{2,L})\, E(\tau_{j,L})}, \quad  j = 3, \dots, 2g,
\end{equation}%
where $K$ is a constant estimated in terms of $g$, $\sys(\frak{M})$ and $\ell(\alpha_1)$ on $S$. Here $\sys(\frak{M})$ is the length of the shortest closed geodesic on $\frak{M}$ and thus the shortest closed geodesic on $S$ different from $\gamma$. With the notation used for the definition of the $\mathcal{O}_{\rm A}$-symbol (below \beqref{eq:OAsymbol}): $\sys(\frak{M}) = \sys_{\gamma}(S)$.

By \beqref{eq:Esskap} and \beqref{eq:sigmaT1} (where $\Ess = \Ess_{\frak{M}}$)  we have
\begin{equation*}
p_{11}(S_L) = E(\sigma_{1,L})=\frac{1}{\Ess_{\frak{M}}(\mT_{1,L})+L}=\frac{1}{\frak{m}_L+L}.
\end{equation*}%
Together with \beqref{eq:prf03} this yields
 \begin{equation}\label{eq:prf08}
p_{11}(S_L) = \frac{1}{\frak{m}^{\sG}+L}+\rho_{11}^{\sG}(L)  \text{ \ with \ } \vert \rho_{11}^{\sG}(L) \vert \leq \frac{e^{-2\pi L}}{\frak{m}^{\sG}+L}.
\end{equation}%
By \beqref{eq:sigmakapj}, \beqref{eq:Fkapp1}  and \bbref{Lemma}{lem:sigma1ortho} we have
 \begin{equation}\label{eq:prf09}%
p_{1j}(S_L) = \kappa_{j,L} p_{11}(S_L), \quad p_{ij}(S_L) = \scp{\tau_{i,L}}{\tau_{j,L}} + \kappa_{i,L}\kappa_{j,L}p_{11}(S_L), \quad i,j = 2,\dots, 2g.
\end{equation}%
We now write, for $j=2,\dots,2g$,
\begin{equation}\label{eq:prf10}
p_{1j}(S_L) =  \frac{\kappa_j^{\sG}}{\frak{m}^{\sG}+L} + \rho_{1j}^{\sG}(L), \quad\text{with}\quad \rho_{1j}^{\sG}(L)=\frac{\kappa_{j,L}-\kappa_j^{\sG}}{\frak{m}^{\sG}+L} + \frac{\kappa_{j,L}(\frak{m}^{\sG}-\frak{m}_L)}{(\frak{m}_L + L)(\frak{m}^{\sG}+L)},
\end{equation}
By \bbref{Lemma}{lem.Bdpkk} and \bbref{Lemma}{lem:BdEsstau2} the energies $q_{kk}(S^{\sG})$, $\Ess(\tau_{2,L})$, $E(\tau_{j,L})$ in \beqref{eq:prf05}, \beqref{eq:prf06} are effectively bounded above in terms of $g$,  $\sys_{\gamma}(S)=\sys(\frak{M})$ and $\max\{\ell(\alpha_1), \dots, \ell(\alpha_{2g})\}$. The same holds, by \bbref{Theorem}{eq.bdscapaM}, for $\Gamma$. Hence, by \beqref{eq:prf03} - \beqref{eq:prf06}
\begin{equation}\label{eq:prf11}
\rho_{12}^{\sG}(L) = \mathcal{O}_{\rm A}(\tfrac{1}{L}e^{-\pi L}), 
\quad
\rho_{1j}^{\sG}(L) = \mathcal{O}_{\rm A}(\tfrac{1}{L}e^{-2\pi L}), \ j=3,\dots,2g. 
\end{equation}
In a similar way, using \beqref{eq:prf09}, we write, for $i,j = 2, \dots, 2g$, 
\begin{align*}
&p_{ij}(S_L) = \scp{\tau_i^{\sG}}{\tau_j^{\sG}}_{\frak{M}}+[L]_{ij} + \frac{\kappa_i^{\sG}\kappa_j^{\sG}}{\frak{m}^{\sG}+L} + \rho_{ij}^{\sG}(L), \quad \text{with}\\
&\rho_{ij}^{\sG}(L) = r_{ij}^{\sG}(L) +\frac{\kappa_i^{\sG}r_j + \kappa_j^{\sG}r_i + r_i r_j }{\frak{m}^{\sG}+L+ r_L}- \frac{\kappa_i^{\sG}\kappa_j^{\sG}r_L}{(\frak{m}^{\sG}+L+r_L)(\frak{m}^{\sG}+L)},\notag
 \end{align*}%
where $[L]_{22} = L$, $[L]_{ij} = 0$, if $(i,j) \neq (2,2)$, and
\begin{equation*}
r_{ij}^{\sG}(L) = \scp{\tau_{i,L}}{\tau_{j,L}}_{\frak{M}}-\scp{\tau_i^{\sG}}{\tau_j^{\sG}}_{\frak{M}}, \quad r_i = \kappa_{i,L}-\kappa_i^{\sG}, \text{ \ \ and \ \ } r_L = \frak{m}_L-\frak{m}^{\sG}. 
\end{equation*}
Using \beqref{eq:prf05} we get
 \begin{equation*}
\vert r_{ij}^{\sG}(L)\vert \leq \nu_L \Ess_{\frak{M}}^{1/2}(\tau_i^{\sG})\Ess_{\frak{M}}^{1/2}(\tau_j^{\sG}) = \nu_L\ssqrt{q_{ii}(S^{\sG})}\ssqrt{q_{jj}(S^{\sG})}, \quad i,j = 2,\dots, 2g.
 \end{equation*}%

By \beqref{eq:prf01}, \beqref{eq:prf02}, $\scp{\tau_i^{\sG}}{\tau_j^{\sG}}_{\frak{M}}+[L]_{ij} = q_{ij}(S^{\sG})$, if $(i,j) \neq (2,2)$, and $\scp{\tau_2^{\sG}}{\tau_2^{\sG}}_{\frak{M}}+[L]_{22}=\frak{m}^{\sG}+L + \pi_{22}$. We thus have
 \begin{align}\label{eq:prf14}
p_{ij}(S_L) &= q_{ij}(S^{\sG}) + \frac{\kappa_i^{\sG}\kappa_j^{\sG}}{L} + \rho_{ij}^{\sG}(L), \ i,j \geq 2, (i,j) \neq (2,2),\notag\\
p_{22}(S_L) &= \frak{m}^{\sG}+L + \pi_{22}^{\sG} + \dfrac{(\kappa_2^{\sG})^2}{\frak{m}^{\sG}+L} + \rho_{22}^{\sG}(L),
 \end{align}%
and, in a similar way as before,
\begin{equation}\label{eq:prf15}
\rho_{22}^{\sG}(L) = \mathcal{O}_{\rm A}(\tfrac{1}{L}e^{- \pi L}), \quad \rho_{ij}^{\sG}(L) = \mathcal{O}_{\rm A}(e^{- 2\pi L}), \ i,j=2, \dots, 2g, (i,j) \neq (2,2).
\end{equation}
With \beqref{eq:prf08} and the tandems \beqref{eq:prf10}, \beqref{eq:prf11} and \beqref{eq:prf14}, 
\beqref{eq:prf15} the proof of \bbref{Theorem}{thm:nonsepGIntro} respectively, \bbref{Theorem}{thm:nonsepG} is now complete.
 \end{proof}%

\begin{proof}[Proof of Theorem \ref{thm:nonsepF}]%
In  this proof the notation is used as in the proof before, except that in this case we use the Fenchel-Nielsen construction, which changes any notation ${\cdot}^{\sG}$ to ${\cdot }^{\sF}$. In this case we have
\begin{equation*}
             L = L_t = \frac{\pi}{t}-\frac{2\arcsin t}{t}. 
\end{equation*}
We recall that
 \begin{equation*}%
\frak{m}_L = \Ess_{\frak{M}}(\mT_{1,L}), \quad \frak{m}^{\sF} = \Ess_{\frak{M}}(\mT_1^{\sF}).
 \end{equation*}%
In view of the geometric bounds  from \bbref{Section}{sec:sigT1} and \bbref{}{sec:RelaxDual} the theorem is a consequence of the following inequalities.  
 \begin{align}\label{eq:formulasnonsepF}%
\nonumber
&\vert  \rho_{11}^{\sF}(L) \vert = \mathcal{O}_{\rm A} \left( t^4 \right) ,  \quad \vert  \rho_{12}^{\sF}(L) \vert = \mathcal{O}_{\rm A} \left( t^2 \right), \quad \vert  \rho_{1j}^{\sF}(L) \vert = \mathcal{O}_{\rm A} \left( t^3 \right) ,
\quad  \text{  \ \ for \ \ }  j = 3, \dots, 2g,\\
&\vert \rho_{ij}^{\sF}(L) \vert = \mathcal{O}_{\rm A} \left( t^2 \right), \quad \text{ \ \ for \ \ } i,j = 2, \dots, 2g.
\rule{40ex}{0ex} 
\end{align}%
Following the lines of the proof of \bbref{Lemma}{L.kappa1}, but adding a quasiconformal map to compare the energies on $S_L$ and $S^{\sF}$ together with the   \bbref{Convergence theorem}{thm:Converge} we obtain now 

\begin{equation}\label{eq:prf01f}%
\vert \frak{m}_L - \frak{m}^{\sF} \vert = \mathcal{O}_{\rm A}(t^2).
 \end{equation}%

The estimate for $\rho_{11}^{\sF}(L)$ then follows from \beqref{eq:EnSigma1} and the above inequality. \\  
The remaining inequalities follow in the same steps as in the previous proof, noting that now $\nu_L$ and $\rho_L$ are of order $ \mathcal{O}_{\rm A}(t^2)$ as the quasi-conformal map is involved. Furthermore by definition $\frac{1}{L} = \frac{1}{L_t}$ is of order $ \mathcal{O}_{\rm A}(t)$.
\end{proof}%

\newpage

\noindent Peter Buser \\
\noindent Department of Mathematics, Ecole Polytechnique F\'ed\'erale de Lausanne\\
\noindent Station 8, 1015 Lausanne, Switzerland\\
\noindent e-mail: \textit{peter.buser@epfl.ch}\\
\\
\\
\noindent Eran Makover\\
\noindent Department of Mathematics, Central Connecticut State University\\
\noindent 1615 Stanley Street, New Britain, CT 06050, USA\\
\noindent e-mail: \textit{makovere@ccsu.edu}\\
\\
\\
\noindent Bjoern Muetzel \\
\noindent Department of Mathematics, Dartmouth College \\
\noindent 27 N. Main Street, Hanover, NH 03755, USA\\
\noindent e-mail: \textit{bjorn.mutzel@gmail.com}\\ 
\\
\\
Robert Silhol\\	
\noindent Department of Mathematics, Universit\'e Montpellier 2 \\
\noindent place Eug\`ene Bataillon, 34095 Montpellier cedex 5, France \\
\noindent e-mail: \textit{robert.silhol@gmail.com}\\

\end{document}